\documentclass[12pt,a4paper]{amsart}
\usepackage[all,cmtip]{xy}
\usepackage{enumerate, comment}
\usepackage{verbatim}
\usepackage{ amssymb, latexsym, amsmath}
\usepackage{fullpage}
\usepackage{url}
\usepackage{hyperref}
\usepackage{ crimson }
\usepackage{amsfonts}
\usepackage{tikz-cd}
\usepackage{amssymb}
\usepackage{amsthm}
\usepackage{amsmath}
\usepackage{amscd}
\usepackage[mathscr]{eucal}
\usepackage{indentfirst}
\usepackage{graphicx}
\usepackage{graphics}
\usepackage{pict2e}
\usepackage{epic}
\usepackage{xfrac}
\usepackage[OT2,T1]{fontenc}
\numberwithin{equation}{section}
\usepackage[margin=3.4cm]{geometry}

\theoremstyle{plain}
\newtheorem{Th}{Theorem}[section]
\newtheorem{Lemma}[Th]{Lemma}
\newtheorem{Cor}[Th]{Corollary}
\newtheorem{Prop}[Th]{Proposition}
\newtheorem{Remark}[Th]{Remark}
 \theoremstyle{definition}
\newtheorem{Def}[Th]{Definition}

\newtheorem{?}[Th]{Problem}

\newcommand{\Hom}{\rm{Hom}}

\newcommand*{\rom}[1]{\uppercase\expandafter{\romannumeral #1\relax}}
\newcommand{\Q}{\mathbb{Q}}
\newcommand{\Z}{\mathbb{Z}}

\newcommand{\op}[1]{\operatorname{#1}}
\newcommand{\ZU}{\text{Z}(\text{U}_{\alpha})}
\newcommand{\F}{\mathbb{F}}
\newcommand{\p}{\varphi}

\newcommand{\GSp}{\operatorname{GSp}}
\newcommand{\g}{\operatorname{Ad}^0\bar{\rho}}
\newcommand{\GL}{\operatorname{GL}}
\newcommand{\G}{\operatorname{G}_{\mathbb{Q},S}}

\DeclareSymbolFont{cyrletters}{OT2}{wncyr}{m}{n}
\DeclareFontFamily{U}{wncy}{}
    \DeclareFontShape{U}{wncy}{m}{n}{<->wncyr10}{}
    \DeclareSymbolFont{mcy}{U}{wncy}{m}{n}
    \DeclareMathSymbol{\Sh}{\mathord}{mcy}{"58}

\newcommand\mtx[4] { \left( {\begin{array}{cc}
   #1 & #2 \\
   #3 & #4 \\
  \end{array} } \right)}
\begin{document}

\title{Deformations of Certain Reducible Galois Representations, \rom{3}}

\author[Anwesh Ray]{Anwesh Ray}
\address{Cornell University \\ Department of Mathematics \\
  Malott Hall, Ithaca, NY 14853-4201 USA} 

\email{ar2222@cornell.edu}
\begin{abstract}Let $p$ be an odd prime and $q$ a power of $p$. We examine the deformation theory of reducible and indecomposable Galois representations $\bar{\rho}:\operatorname{G}_{\Q}\rightarrow \GSp_{2n}(\F_q)$ that are unramified outside a finite set of primes $S$ and whose image lies in a Borel subgroup. We show that under some additional hypotheses, such representations have geometric lifts to the Witt vectors $\text{W}(\F_q)$. The main theorem of this manuscript is a higher dimensional generalization of the result of Hamblen-Ramakrishna \cite{hamblenramakrishna}.
\end{abstract}

\maketitle

\section{Introduction} 
Let $p$ be an odd prime number and let $\F_q$ be the finite field with $q=p^M$ elements. Denote by $\text{W}(\F_q)$ the ring Witt vectors of $\F_q$. For each prime number $l$, choose an embedding $\iota_l:\bar{\Q}\hookrightarrow \bar{\Q}_l$. Denote by $\operatorname{G}_l$ the absolute Galois group $\operatorname{Gal}(\bar{\Q}_l/\Q_l)$. Let $S$ be a finite set of prime numbers containing $p$. Denote by $\Q_S$ the maximal algebraic extension of $\Q$ which is unramified at each prime $l\notin S$ and set $\operatorname{G}_{\Q,S}=\operatorname{Gal}(\Q_S/\Q)$. The weak form of Serre's conjecture states that an odd and irreducible two-dimensional Galois representation 
\[\bar{\varrho}:\operatorname{G}_{\Q,S}\rightarrow \text{GL}_2(\F_q)\]is modular, i.e., lifts to a characteristic zero representation $\varrho$ attached to an eigencuspform. The strong form asserts that the eigencuspform may be chosen to have optimal level equal to the prime to $p$ part of the Artin conductor of $\bar{\varrho}$. Ribet \cite{ribetLL} proved via a level lowering argument that the weak form implies the strong form. Khare-Wintenberger \cite{KW2} went on to prove the full statement, building on Ribet's work.
\par 
In \cite{hamblenramakrishna}, Hamblen and Ramakrishna prove a generalization of the weak form of Serre's conjecture for reducible two-dimensional Galois representations. They impose some conditions on a two-dimensional representation
$\bar{\varrho}:\operatorname{G}_{\Q,S}\rightarrow \op{GL}_2(\F_q)$, namely,
\begin{enumerate}
\item $\bar{\varrho}$ is reducible of the form 
\[\bar{\varrho}=\mtx{\varphi}{\ast}{}{1},\] where $\varphi:\operatorname{G}_{\Q,S}\rightarrow \F_q^{\times}$ is a Galois character.
\item The representation $\bar{\varrho}$ is odd, i.e. if $c\in \operatorname{G}_{\Q}$ denotes complex conjugation, \[\det \bar{\varrho}(c)=-1.\]
\item The representation $\bar{\varrho}$ is indecomposable, i.e. the cohomology class \[\ast\in H^1(\G, \F_q(\varphi))\] is non-zero.
\item 
The Galois character $\varphi$ is stipulated to satisfy some conditions, for instance, $\varphi\neq  \bar{\chi},\bar{\chi}^{-1}$ where $\bar{\chi}$ is the mod $p$ cyclotomic character and $\varphi^2\neq 1$. Further, the image of $\varphi$ is stipulated to span $\F_q$ over $\F_p$.
\item 
There are further conditions on the restriction $\bar{\varrho}_{\restriction \operatorname{G}_p}$. The reader may refer to condition 5 of Theorem 2 in \cite{hamblenramakrishna} for further details.
\end{enumerate} 
Hamblen and Ramakrishna show that if $\bar{\varrho}$ satisfies the above mentioned conditions, then on enlarging the set of ramification $S$ by a finite set of primes $X$, $\bar{\varrho}$ has an odd, irreducible, $p$-ordinary lift $\varrho$ which is unramified outside $S\cup X$
 \[\begin{tikzpicture}[node distance = 2.0cm, auto]
      \node (GSX) {$\operatorname{G}_{\Q,S\cup X}$};
      \node (GS) [right of=GSX] {$\operatorname{G}_{\Q,S}$};
      \node (GL2) [right of=GS]{$\text{GL}_2(\F_q).$};
      \node (GL2W) [above of= GL2]{$\text{GL}_2(\text{W}(\F_q))$};
      \draw[->] (GSX) to node {} (GS);
      \draw[->] (GS) to node {$\bar{\varrho}$} (GL2);
      \draw[->] (GL2W) to node {} (GL2);
      \draw[dashed,->] (GSX) to node {$\varrho$} (GL2W);
      \end{tikzpicture}\]This lift is \textit{geometric} in the sense of Fontaine and Mazur \cite{fontainemazur}. By the result of Skinner and Wiles in \cite{skinnerwiles}, the representation $\varrho$ arises from a $p$-ordinary eigencuspform. This settles the weak form of Serre's conjecture for such $\bar{\varrho}$.
\par
The prospect of generalizing this lifting result leads us to examine higher dimensional Galois representations with image in a smooth group-scheme $\operatorname{G}$ over $\text{W}(\F_q)$. Assume that $\op{G}_{\restriction \F_q}$ is split and reductive and choose a split Borel $\operatorname{B}_{/\F_q}\subset \op{G}_{\restriction \F_q}$. Let $\bar{\rho}$ be a homomorphism \[\bar{\rho}:\operatorname{G}_{\Q}\rightarrow \operatorname{G}(\F_q).\] Let $\mathfrak{g}$ denote the Lie-algebra of the adjoint group $\operatorname{G}_{\restriction \F_q}^{ad}$ and $\Phi(\operatorname{G}_{\restriction \F_q}^{ad})$ be a root system compatible with the choice of Borel. Denote by $\mathfrak{n}\subset \mathfrak{g}$ the span of the positive roots. The $\F_q$-vector space $\mathfrak{g}$ acquires an adjoint Galois action
\[\g:\operatorname{G}_{\Q}\rightarrow \operatorname{Aut}_{\F_q}(\mathfrak{g}).\] Denote by $\g$ the Galois module with underlying vector space $\mathfrak{g}$. It is imperative that $\bar{\rho}$ is \textit{odd}. For an involution $\tau\in \op{Aut}(\operatorname{G}_{\restriction \F_q})$, let $(\g)^{\tau}$ denote the subspace of $\g$ fixed by $\tau$. It was shown by E. Cartan that
\begin{equation*}
 \dim  (\g)^{\tau}\geq\dim  \mathfrak{n}
\end{equation*}
(see \cite[Proposition 2.2]{Yun} for further details).
The representation $\bar{\rho}$ is \textit{odd} if equality is achieved for the involution $\op{ad} \bar{\rho}(c)$, i.e.
\begin{equation}\label{oddness}
 \dim  (\g)^{\op{ad}\bar{\rho}(c)}=\dim  \mathfrak{n}.
\end{equation}
In particular, the group $\operatorname{G}$ must contain an element $h=\bar{\rho}(c)$ for which equality $\ref{oddness}$ is achieved. Such an element is said to induce a \textit{Chevalley involution}. When $n>2$, the general linear group $\GL_n(\F_q)$ contains no such element. Hence there are no odd representations for the group $\GL_n(\F_q)$ when $n>2$. On the other hand, the general symplectic group $\GSp_{2n}(\F_q)$ does contain elements which induce Chevalley involutions. 
\par Ramakrishna in \cite{RamLGR} and \cite{RamFM} showed that odd, irreducible representations $\bar{\rho}:\op{G}_{\Q}\rightarrow \op{GL}_2(\bar{\F}_p)$ satisfying some additional hypotheses exhibit characteristic zero lifts which are \textit{geometric} in the sense of Fontaine and Mazur. These results provided evidence for Serre's conjecture, before it was proved by Khare and Wintenberger. Taylor in \cite{taylor} introduced a reformulation of Ramakrishna's method, by showing that the vanishing of a certain \textit{dual Selmer group} is sufficient in asserting the existence of global Galois deformations with fixed local conditions. This new formulation paved the way for higher dimensional generalizations. In \cite{partikisthesis}, Patrikis generalized Ramakrishna's lifting theorem to odd representations with \textit{big image} in $\GSp_{2n}(\F_q)$. Fakhruddin, Khare and Patrikis studied more general odd, irreducible representations in \cite{FKP1} and \cite{FKP2}.
\par We assume that our representation has image in $\GSp_{2n}(\F_q)$ for $n\geq 2$. Associate to a commutative $\text{W}(\F_q)$-algebra $R$, a non-degenerate alternating form on $R^{2n}$ prescribed by the matrix\[J:=\mtx{}{\operatorname{Id}_n}{-\operatorname{Id}_n}{}.\] The group of general symplectic matrices $\GSp_{2n}(R)$ consists of matrices $X$ which preserve this form up to a scalar i.e. satisfy
$X^t J X\in R^{\times} \cdot J$. The similitude character $\nu:\GSp_{2n}(R)\rightarrow R^{\times}$ is defined by the relation $X^t J X=\nu(X)\cdot J$. The space $\g$ is an $\F_q[\operatorname{G}_{\Q,S}]$-module with underlying space $\operatorname{sp}_{2n}(\F_q)$. The Galois action is prescribed by
 \[g\cdot X=\bar{\rho}(g) X \bar{\rho}(g)^{-1}\]where $g\in \G$ and $X\in \operatorname{sp}_{2n}(\F_q)$. Let $\operatorname{B}(R)$ be the Borel subgroup consisting of matrices
\[M=\mtx{C}{CD}{}{\xi (C^t)^{-1}}\]where $C\in \GL_n(R)$ is upper triangular, $D\in \GL_n(R)$ is symmetric and $\xi\in R^{\times}$. Note that in this setting, $\op{B}$ is defined over $\text{W}(\F_q)$. Denote by $U_1\subset \op{B}$ the unipotent subgroup.
\par
Let $\bar{\rho}:\G\rightarrow \GSp_{2n}(\F_q)$ be a continuous Galois representation with image in $\op{B}(\F_q)$. Composing $\bar{\rho}$ with the similitude-character $\bar{\nu}$ defines a Galois character denoted by $\bar{\kappa}$. Denote by $\mathcal{T}\subseteq \GSp_{2n}$ the diagonal torus and $e_{i,j}\in \op{GL}_{2n}(\F_q)$ the matrix with $1$ in the $(i,j)$-position and $0$ in all other positions. Set $\mathfrak{t}$ for the $\F_q$-span of $H_1,\dots, H_n$, where $H_i:=e_{i,i}-e_{n+i,n+i}$. Let $L_1, \dots, L_n\in \mathfrak{t}^*$ be the dual basis of $H_1,\dots, H_n$. An integer linear combination $\lambda$ of $L_1, \dots, L_n$ is viewed as character on the torus $\mathcal{T}(\F_q)$, which is trivial on the center of $\GSp_{2n}(\F_q)$. Via the natural quotient map $\op{B}\rightarrow \mathcal{T}$, a character on $\mathcal{T}$ induces a character on $\op{B}$. The character on $\op{B}$ induced by $\lambda$ is denoted by \[\omega_{\lambda}:\op{B}(\F_q)\rightarrow \F_q^{\times}.\] Associated to $\omega_{\lambda}$ is the Galois character
\[\sigma_{\lambda}=\omega_{\lambda}\circ \bar{\rho}:\G\rightarrow \F_q^{\times}.\] Let "$1$" be a formal symbol for the trivial linear combination of $L_1,\dots, L_n$ and set $\sigma_1$ equal to the trivial character. The roots $\Phi=\Phi(\g, \mathfrak{t})$ are specified by
\[\begin{split}
\Phi=& \{\pm 2L_1, \dots, \pm 2 L_n\} \\\cup &\{\pm(L_i+L_j)\mid 1\leq i<j\leq n\}\\ \cup &\{\pm(L_i-L_j)\mid 1\leq i<j\leq n\}.
\end{split}\]
The choice of the Borel $\operatorname{B}\subset \GSp_{2n}$ prescribes the following choice of simple roots $\Delta=\{\lambda_i\mid i=1,\dots, n\}$, with $\lambda_i:=L_i-L_{i+1}$ for $i<n$ and $\lambda_n:=2L_n$. The root
$2L_1=2\left(\sum_{i=1}^{n-1} \lambda_i\right)+\lambda_n$
is the highest root and the unique root of height $2n-1$. Denote by $\mathfrak{b}$ and $\mathfrak{n}$ the Lie subalgebras of $\g$ corresponding to the Borel and unipotent subgroups respectively. Set $\chi$ for the $p$-adic cyclotomic character, $\bar{\chi}$ its mod-$p$ reduction and let $c$ denote complex conjugation.
\begin{Th} \label{main} Let $\bar{\rho}:\operatorname{G}_{\Q,S}\rightarrow \operatorname{B}(\F_q)$ be a Galois representation of the form:

\begin{equation}\label{introducingbarrho}
\bar{\rho}= \begin{pmatrix}
  \p_1 & * & * & \cdots & * & *  & \cdots & *\\
   & \p_2 & * & \cdots & * & * & \cdots & *\\
    &  & \ddots & \vdots & \vdots & \vdots & \vdots & \vdots\\  
    &  &    & \p_n & * & * & \cdots & *\\
    &  &    &  & \p_1^{-1}\bar{\kappa} &  &  & \\
    &  &    &  & * &  \p_2^{-1}\bar{\kappa} &  & \\
    &  &    &  & \vdots & \vdots & \ddots & \\
    &  &    &  & * & * & \cdots & \p_n^{-1}\bar{\kappa}\\
   \end{pmatrix}
\end{equation} and let $S$ be a finite set of primes which contains $p$. Assume that the following conditions are satisfied:
\begin{enumerate}
\item\label{thc1}$p>2n$,
\item\label{thc2}$\bar{\rho}$ is odd, i.e.
$\dim (\g)^{\op{ad}\bar{\rho}(c)}=\dim \mathfrak{n}$.
\item\label{thc3}The image of $\bar{\rho}$ contains the unipotent subgroup $\operatorname{U}_1(\F_q)$.
\item\label{thc4}Both the following conditions on the distinctness of the characters $\{\sigma_{\lambda}\}$ are satisfied:
\begin{enumerate}
    \item For $\lambda, \lambda'\in \Phi\cup \{1\}$ such that $\lambda\neq \lambda'$, $\sigma_{\lambda}$ is not a $\op{Gal}(\F_q/\F_p)$-twist of $\sigma_{\lambda'}$.
    \item Moreover for $\lambda, \lambda'\in \Phi\cup \{1\}$ not necessarily distinct, $\sigma_{\lambda}$ is not a $\op{Gal}(\F_q/\F_p)$-twist of $\bar{\chi}\sigma_{\lambda'}$.
\end{enumerate}
\item\label{thc5}
For each of the roots $\lambda\in \Phi$, the $\F_p$-linear span of the image of $\sigma_{\lambda}$ in $\F_q$ is $\F_q$.

\item\label{thc7}
At each prime $v\in S$ such that $v\neq p$, there is a liftable local deformation condition $\mathcal{C}_v$ with tangent space $\mathcal{N}_v$ of dimension 
\[\dim  \mathcal{N}_v=h^0(\operatorname{G}_v, \g).\]\item\label{thc8}Tilouine's regularity conditions $(\operatorname{REG})$ and $(\operatorname{REG})^*$ are satisfied, i.e. 
\[H^0(\op{G}_p, \g/\mathfrak{b})=0\text{ and }H^0(\op{G}_p, (\g/\mathfrak{b})(\bar{\chi}))=0.\]
\end{enumerate} Let $\kappa$ be a fixed choice of a lift of the character $\bar{\kappa}$ such that $\kappa=\kappa_0\chi^k$, where $k$ is a positive integer divisible by $p(p-1)$ and $\kappa_0$ is the Teichm\"uller lift of $\bar{\kappa}$. Then $\exists$ a finite set of auxiliary primes $X$ disjoint from $S$ and a lift $\rho$ \[\begin{tikzpicture}[node distance = 2.0cm, auto]
      \node (GSX) {$\operatorname{G}_{\Q,S\cup X}$};
      \node (GS) [right of=GSX] {$\operatorname{G}_{\Q,S}$};
      \node (GL2) [right of=GS]{$\GSp_{2n}(\F_q).$};
      \node (GL2W) [above of= GL2]{$\GSp_{2n}(\text{W}(\F_q))$};
      \draw[->] (GSX) to node {} (GS);
      \draw[->] (GS) to node {$\bar{\rho}$} (GL2);
      \draw[->] (GL2W) to node {} (GL2);
      \draw[dashed,->] (GSX) to node {$\rho$} (GL2W);
      \end{tikzpicture}\] for which 
\begin{enumerate}
\item $\rho$ is irreducible,
\item $\rho$ is $p$-ordinary (in the sense of \cite[section 4.1]{patrikisexceptional}),
\item $\nu\circ \rho= \kappa$,
\item for $v\in S\backslash \{p\}$, the restriction to the decomposition group $\rho_{\restriction \operatorname{G}_v}\in \mathcal{C}_v$.
\end{enumerate}
\end{Th}

The lift $\rho$ is geometric in the sense of Fontaine and Mazur. For $\lambda\in \Phi$, setting $\lambda=-\lambda'$ in condition $\eqref{thc4}$, we have that $\sigma_{\lambda}^2\neq 1$. Note that the conditions also imply that $\sigma_{\lambda}\neq \bar{\chi},\bar{\chi}^{-1}$. This is reminiscent of the condition $\varphi^2\neq 1$ of Hamblen-Ramakrishna. In particular, condition $\eqref{thc4}$ implies that $p>2$. The requirement $p>2n$ is primarily made so that we may suitably work with the exponential map in various places.
\par It is a consequence of Tilouine's regularity conditions that the ordinary deformations of $\bar{\rho}_{\restriction \op{G}_p}$ constitute a liftable deformation condition $\mathcal{C}_p$ for which the tangent space $\mathcal{N}_p$ has dimension:
\[\dim \mathcal{N}_p=h^0(\op{G}_p,\g)+\dim \mathfrak{n}.\] For a discussion on the ordinary deformation condition and Tilioune's regularity conditions, the reader may refer to \cite[section 4]{patrikisexceptional}. With reference to condition $\eqref{thc7}$, the reader may consult \cite[sections 4.3 and 4.4]{patrikisexceptional} for examples of such deformation conditions.
\par When $n=1$, the unipotent group is abelian and examples of such Galois representations $\bar{\rho}$ are constructed via class field theory. The reader may, for instance, refer to \cite{ribetclassgroups}, \cite[section 7]{hamblenramakrishna} and \cite{ray} for more details. In section $\ref{examples}$, examples of Galois representations $\bar{\rho}:\op{G}_{\Q,\{p\}}\rightarrow \op{GSp}_4(\F_p)$ satisfying the conditions of Theorem $\ref{main}$ are constructed. It is shown that if $p\geq 23$ is a regular prime, there exists a Galois representation
  \[\bar{\rho}=\left( {\begin{array}{cccc}
   \bar{\chi}^3 &\ast & \ast & \ast \\
    & 1 & \ast & \ast \\
    & & \bar{\chi}^6 & \\
    & & \ast & \bar{\chi}^9
  \end{array} } \right):\op{G}_{\Q,\{p\}}\rightarrow \op{GSp}_4(\F_p)\] which satisfies the conditions of Theorem $\ref{main}$.
\subsection*{Acknowledgements}
The author is very grateful to his advisor Ravi Ramakrishna for introducing him to the fascinating subject of Galois deformations. He thanks Brian Hwang and Nicolas Templier for fruitful conversations. The author also appreciates the suggestions made by the anonymous referee which have led to significant improvement of the article.
\section{Notation}\label{notationsection}
\begin{itemize}
\item For an $\F_q$-vector space $M$, set $\dim M:=\dim_{\F_q} M$.
\item At every prime $v$, choose an embedding $\iota_v:\bar{\Q}\hookrightarrow\bar{\Q}_v$. The absolute Galois group $\operatorname{G}_v=\operatorname{Gal}(\bar{\Q}_v/\Q_v)$ is identified with the decomposition group of the prime dividing $v$ determined by $\iota_v$. 
\item 
 Let $e_{i,j}$ denote the $2n\times 2n$ square matrix with coefficients in $\F_q$ with $1$ in the $(i,j)$-th position and $0$ at all other positions. \item The space $\g$ is an $\F_q[\operatorname{G}_{\Q,S}]$-module with underlying space $\operatorname{sp}_{2n}(\F_q)$. The Galois action is prescribed by
 \[g\cdot X=\bar{\rho}(g) X \bar{\rho}(g)^{-1}\]where $g\in \G$ and $X\in \operatorname{sp}_{2n}(\F_q)$.
 \item The space of diagonal matrices in $\g$ is denoted by $\mathfrak{t}$. Let $H_1,\dots, H_n$ be the basis of $\mathfrak{t}$ defined by $H_i:=e_{i,i}-e_{n+i,n+i}$. Let $L_1, \dots, L_n\in \mathfrak{t}^*$ be the dual basis.
 \item Let $\Phi$ be the set of roots of $\operatorname{sp}_{2n}(\F_q)$ and $\lambda_1,\dots, \lambda_n\in \Phi$ be the simple roots defined as follows
 \[\lambda_i:=\begin{cases}
 L_i-L_{i+1}\text{ for }i<n\\
 2L_n \text{ for }i=n.
 \end{cases}\]Let $\Delta$ be the simple roots $\{\lambda_1,\dots, \lambda_n\}$.
 Let $\Phi^{+}$ and $\Phi^{-}$ denote the positive and negative roots respectively.
 \item For $\lambda\in \Phi$, let $(\g)_{\lambda}$ be the $\lambda$ root-subspace of $\g$. For $\lambda\in \Phi$, let $X_{\lambda}$ be a choice a root vector generating the one-dimensional space $(\g)_{\lambda}$. For instance when $n=2$, we may choose root vectors as follows,
\[
{\small X_{2L_1}}:={\tiny\left( {\begin{array}{cccc}
    0 &  & 1 &  \\
    & 0 &  &  \\
   &  &  0 &  \\
    &  & & 0
  \end{array} }\right)}, {\small X_{2L_2}}:={\tiny\left( {\begin{array}{cccc}
    0 &  &  &  \\
    & 0 &  & 1 \\
   &  &  0 &  \\
    &  & & 0
  \end{array} }\right)},\]\[{\small X_{L_1+L_2}}:={\tiny\left( {\begin{array}{cccc}
    0 &  &  & 1 \\
    & 0 & 1 &  \\
   &  &  0 &  \\
    &  & & 0
  \end{array} }\right)}\text{ and }{\small X_{L_1-L_2}}:={\tiny\left( {\begin{array}{cccc}
    0 & 1 &  &  \\
    & 0 &  &  \\
   &  &  0 &  \\
    &  & -1 & 0
  \end{array} }\right).}
\]
For $\lambda\in \Phi^-$, we may choose $X_{\lambda}$ to be the transpose of $X_{-\lambda}$.

\item
Let $\lambda$ be a root. There is a unique presentation of $\lambda=\sum \alpha_i \lambda_i$, where $\alpha_i$ are all non-negative or all non-positive. The height of $\lambda$ is defined by $\operatorname{ht}(\lambda):=\sum_i\alpha_i$. For instance, the root $2L_1=2\left(\sum_{i=1}^{n-1} \lambda_i\right)+\lambda_n$ has height equal to $2n-1$. Every other root has height less than $2n-1$.
\item 
For any integer $k$, let
$(\g)_k$ be the $\F_q[\G]$-submodule defined by
\[(\g)_k:=\bigoplus_{\substack{\alpha\in \Phi \\ \text{ht}\alpha\geq k}} (\g)_{\alpha}.\] Set $\mathfrak{b}:=(\g)_0$ and $\mathfrak{n}:=(\g)_1$.
\item
Associated with any root $\lambda$ is a Galois character denoted by $\sigma_{\lambda}:\G\rightarrow \F_q^{\times}$ obtained by composing $\bar{\rho}$ with the character induced on $\op{B}(\F_q)$ by the root $\lambda$. Denote by $\sigma_1$ the trivial character and set $\op{ht}(1)=0$. For $\lambda\in \Phi\cup\{1\}$, we have that 
\[g\cdot X_{\lambda}-\sigma_{\lambda}(g) X_{\lambda} \in (\g)_{\op{ht}(\lambda)+1}.\]
\item Let $Q\subseteq \g$ be an $\F_q[\G]$-submodule, the $\sigma_{2L_1}$-eigenspace of $Q$ is the $\F_q[\G]$-submodule defined by $Q_{\sigma_{2L_1}}:=Q\cap (\g)_{2L_1}$. Likewise, if $P\subseteq \g^*$ is an $\F_q[\G]$-submodule, the $\bar{\chi} \sigma_{2L_1}$-eigenspace $P_{\bar{\chi} \sigma_{2L_1}}$ is defined by
\[P_{\bar{\chi} \sigma_{2L_1}}:=\{v\in P\mid v(X)=0\text{ for }X\in (\g)_{-2n+2}\}.\]
\item
For $k\geq 1$, let $U_k\subset \op{B}(\F_q)$ be the exponential subgroup generated by $\operatorname{exp}((\g)_k)$. Note that the exponential map here is well-defined since $p>2n$. The group $U_1$ is the unipotent subgroup of $\op{B}(\F_q)$.
\item Throughout, $h^i$ will be an abbreviation for $\dim H^i$. For instance, $h^i(\op{G}_v, \g)$ is an abbreviation for $\dim H^i(\op{G}_v, \g)$.
\item Let $M$ be an $\F_q[\op{G}_S]$-module, let $\Sh^i_S(M)$ consist of cohomology classes $f\in H^i(\G, M)$ such that $f_{\restriction \op{G}_v}=0$ for all $v\in S$.
\end{itemize}
\section{The General Lifting Strategy}\label{section3}
\par Let $\bar{\varrho}$ be a Galois representation $\bar{\varrho}:\operatorname{G}_{\Q}\rightarrow \GL_2(\bar{\F}_p)$ which is irreducible, odd and unramified outside finitely many primes. Ramakrishna in \cite{RamFM} and \cite{RamLGR} showed that if $\bar{\varrho}$ satisfies additional conditions, it lifts to a continuous Galois representation $\varrho$ which is geometric in the sense of Fontaine and Mazur. In other words, $\varrho$ is odd, unramified outside finitely many primes and $\varrho_{\restriction \operatorname{G}_p}$ is de Rham. This geometric lifting theorem provided evidence for the weak version of Serre's conjecture before it was proved by Khare and Wintenberger. The geometric lifting construction was adapted to the reducible setting in \cite{hamblenramakrishna}. The main result of this manuscript is a higher dimensional generalization of the lifting theorem of Hamblen-Ramakrishna. The basic strategy involves successively lifting $\bar{\rho}$ to a characteristic zero irreducible geometric representation $\rho$ by successively lifting $\rho_m$ to $\rho_{m+1}$
\[\begin{tikzpicture}[node distance = 2.0 cm, auto]
      \node (GSX) at (0,0){$\operatorname{G}_{\Q,S\cup X}$};
      \node (GL2) at (5,0){$\GSp_{2n}(\F_q).$};
      \node (GL2Wn) at (3,2)[above of= GL2]{$\GSp_{2n}(\text{W}(\F_q)/p^{m})$};
      \node (GL2Wnplus1) at (5,4){$\GSp_{2n}(\text{W}(\F_q)/p^{m+1})$};
      \draw[->] (GSX) to node [swap]{$\bar{\rho}$} (GL2);
      \draw[->] (GL2Wn) to node {} (GL2);
      \draw[->] (GSX) to node [swap]{$\rho_m$} (GL2Wn);
      \draw[->] (GL2Wnplus1) to node {} (GL2Wn);
      \draw[dashed,->] (GSX) to node {$\rho_{m+1}$} (GL2Wnplus1);
      \end{tikzpicture}\] 
      
      \begin{Def} Let $\mathcal{C}$ be the category of coefficient rings over $\text{W}(\F_q)$ with residue field $\F_q$. The objects of this category consist of local $\text{W}(\F_q)$-algebras $(R,\mathfrak{m})$ for which
      \begin{itemize}
          \item $R$ is complete and Noetherian,
          \item $R/\mathfrak{m}$ is isomorphic to $\F_q$ as a $\text{W}(\F_q)$-algebra. The residual map \[\phi:R\rightarrow \F_q\]is the composite of the quotient map $R\rightarrow R/\mathfrak{m}$ with the unique isomorphism of $W(\F_q)$-algebras $R/\mathfrak{m}\xrightarrow{\sim}\F_q$.
      \end{itemize} A morphism $F:(R_1,\mathfrak{m}_1)\rightarrow (R_2,\mathfrak{m}_2)$ is a homorphism of local rings which is also a $\text{W}(\F_q)$-algebra homorphism. Recall that $\kappa$ is a fixed choice of lift of $\bar{\kappa}$. Let $\kappa_v$ denote the restriction of $\kappa$ to $\operatorname{G}_v$.
      \end{Def}
      \par Let $v$ be a prime and $R\in \mathcal{C}$. Denote by $\phi^*: \GSp_{2n}(R)\rightarrow \GSp_{2n}(\F_q)$ the group homomorphism induced by the residual homomorphism $\phi: R\rightarrow \F_q$. We say that $\rho_R:\operatorname{G}_v\rightarrow \GSp_{2n}(R)$ is an $R$-lift of $\bar{\rho}_{\restriction \operatorname{G}_v}$ if $\phi^*\circ \rho_R=\bar{\rho}_{\restriction \operatorname{G}_v}$, i.e. the following diagram commutes
     \[ \begin{tikzpicture}[node distance = 2.2 cm, auto]
            \node(G) at (0,0){$\operatorname{G}_{v}$};
             \node (A) at (3,0) {$\GSp_{2n}(\F_q)$.};
             \node (B) at (3,2) {$\GSp_{2n}(R)$};
      \draw[->] (G) to node [swap]{$\bar{\rho}_{\restriction \operatorname{G}_v}$} (A);
       \draw[->] (B) to node{$\phi^*$} (A);
      \draw[->] (G) to node {$\rho_R$} (B);
      \end{tikzpicture}\]Further, we shall require that the similitude character of $\rho_R$ coincides with the composite of $\kappa_v$ with the homomorphism  $W(\F_q)^{\times}\rightarrow R^{\times}$ induced by the structure map.
      \par Two lifts $\rho_R$ and $\rho_R'$ are said to be strictly-equivalent if there is
      \[A\in \text{ker}\lbrace \GSp_{2n}(R)\xrightarrow{\phi^*} \GSp_{2n}(\F_q)\rbrace \]
      such that 
      $\rho_R=A\rho_R' A^{-1}$. A deformation is a strict equivalence class of lifts. Let $\operatorname{Def}_v(R)$ be the set of $R$-deformations of $\bar{\rho}_{\restriction \operatorname{G}_v}$. The association $R\mapsto \operatorname{Def}_v(R)$ defines a covariant functor \[\operatorname{Def}_v:\mathcal{C}\rightarrow \operatorname{Sets}. \]
      The tangent space $\operatorname{Def}_v(\F_q[\epsilon]/(\epsilon^2))$ naturally acquires the structure of an $\F_q$-vector space and is isomorphic to $H^1(\operatorname{G}_v,\g)$. Under this association, a cohomology class $f$ is identified with the deformation $(\operatorname{Id}+\epsilon f)\bar{\rho}_{\restriction \operatorname{G}_v}$.
      \par For $m\in \Z_{\geq 2}$, the deformations $\operatorname{Def}_v(\text{W}(\F_q)/p^{m})$ are equipped with action of the cohomology group $H^1(\operatorname{G}_v, \g)$. For $\varrho_m\in \operatorname{Def}_v(\text{W}(\F_q)/p^{m})$ and $f\in H^1(\operatorname{G}_v, \g)$, the twist of $\varrho_m$ by $f$ is defined by the formula $(\op{Id}+p^{m-1}f)\varrho_m$. The set of deformations $\varrho_m$ of a fixed $\varrho_{m-1}\in \operatorname{Def}_v(\text{W}(\F_q)/p^{m-1})$ is either empty or in bijection with $H^1(\operatorname{G}_v,\g)$.
      \begin{Def}\label{defconditiondef} (see \cite{taylor}) We say that a sub-functor $\mathcal{C}_v$ of $\operatorname{Def}_v$ is a deformation condition if (1) to (3) below are satisfied. If condition (4) is satisfied, $\mathcal{C}_v$ is said to be liftable.
      \begin{enumerate}
          \item First, we require that $\mathcal{C}_v(\F_q)=\{\bar{\rho}_{\restriction \operatorname{G}_v}\}.$
          \item For $R_1$ and $R_2$ be $\mathcal{C}$, let $\rho_1\in \mathcal{C}_v(R_1)$ and $\rho_2\in \mathcal{C}_v(R_2)$. Let $I_1$ be an ideal in $R_1$ and $I_2$ an ideal in $R_2$ such that there is an isomorphism $\alpha:R_1/I_1\xrightarrow{\sim} R_2/I_2$ satisfying \[\alpha(\rho_1 \;\text{mod}{I_1})=\rho_2 \;\text{mod}{I_2}.\] Let $R_3$ be the fibred product \[R_3=\lbrace(r_1,r_2)\mid \alpha(r_1\text{mod} I_1)=r_2 \text{mod} I_2\rbrace\] and $\rho_3$ the $R_3$-deformation induced from $\rho_1$ and $\rho_2$. Then $\rho_3$ satisfies $\mathcal{C}_v(R_3)$.
          \item Let $R\in \mathcal{C}$ with maximal ideal $\mathfrak{m}_R$. If $\rho\in \operatorname{Def}_v(R)$ is such that $\rho\mod{\mathfrak{m}_R^n}$ satisfies $\mathcal{C}_v$ for all $n\in \Z_{\geq 1}$, then  $\rho$ also satisfies $\mathcal{C}_v$.
          \item Let $R\in \mathcal{C}$ and $I$ an ideal such that $I.\mathfrak{m}_R=0$. For $\rho\in \mathcal{C}_v(R/I)$, there exists $\tilde{\rho}\in \mathcal{C}_v(R)$ such that $\rho=\tilde{\rho}\mod{I}$.
      \end{enumerate}
      \end{Def}
      Let $\mathcal{C}_v$ be a local deformation condition at the prime $v$. The tangent space $\mathcal{N}_v$ consists of $f\in H^1(\operatorname{G}_v, \g)$, such that $(\op{Id}+\epsilon f) \bar{\rho}_{\restriction \op{G}_v}\in \mathcal{C}_v(\F_q[\epsilon]/(\epsilon^2))$. The action of $\mathcal{N}_v$ on $\operatorname{Def}_v(\text{W}(\F_q)/p^m)$ stabilizes $\mathcal{C}_v(\text{W}(\F_q)/p^m)$. In other words, if $\varrho_m\in \mathcal{C}_v(\text{W}(\F_q)/p^m)$ and $f\in \mathcal{N}_v$, then 
      \[(\op{Id}+p^{m-1}f)\varrho_{m}\in \mathcal{C}_v(\text{W}(\F_q)/p^m). \]It is assumed that each prime $v\in S\backslash \{p\}$ is equipped with a liftable local deformation condition $\mathcal{C}_v$ such that 
\[\dim  \mathcal{N}_v=h^0(\operatorname{G}_v, \g).\] The reader may consult \cite[sections 4.3 and 4.4]{patrikisexceptional} for examples of such deformation conditions. The deformation condition $\mathcal{C}_p$ is the ordinary deformation condition. Since we have assumed that Tilouine's regularity conditions are satisfied (cf. \cite[section 4.1]{patrikisexceptional}), $\mathcal{C}_p$ is liftable and the tangent space $\mathcal{N}_p$ has dimension equal to
\[\dim \mathcal{N}_p=h^0(\operatorname{G}_p, \g)+\dim \mathfrak{n},\]see \cite[Proposition 4.4]{patrikisexceptional}. We allow the successive lifts $\rho_m$ to be ramified at a set of primes $S\cup X$. Each auxiliary prime $v\in X$ is equipped with a liftable subfunctor $\mathcal{C}_v$ of $\operatorname{Def}_v$. These auxiliary primes are referred to as trivial primes and were introduced by Hamblen and Ramakrishna in the two-dimensional setting \cite[section 4]{hamblenramakrishna}. These are primes $v\equiv 1\mod{p}$, not contained in $S$, at which $\bar{\rho}_{\restriction \op{G}_v}$ the trivial representation and $v\not\equiv 1\mod{p^2}$. We use a higher dimensional generalization of the deformation functor $\mathcal{C}_v$ at a trivial prime $v$, due to Fakhruddin, Khare and Patrikis \cite[Definition 3.1]{FKP1}. At each trivial prime $v$ there is a subspace $\mathcal{N}_v$ of $H^1(\operatorname{G}_v, \g)$ of dimension $h^0(\operatorname{G}_v, \g)$ which behaves like a versal tangent space. For $m\geq 3$, the action of $\mathcal{N}_v$ on $\operatorname{Def}(\text{W}(\F_q)/p^m)$ stabilizes $\mathcal{C}_v(\text{W}(\F_q)/p^m)$. This is proved in the $\op{GL}_2$-case by Hamblen-Ramakrishna, see \cite[Corollory 25, 29]{hamblenramakrishna}. For more general groups, we refer to Fakhruddin-Khare-Patrikis \cite[Lemma 3.6,3.10]{FKP1} for the precise statement. However, this is not the case for $m=2$.
\par Let $X$ be a finite set of trivial primes disjoint from $S$. For $v\in S\cup X$, set $\mathcal{N}_v^{\perp}\subseteq H^1(\operatorname{G}_v,\g^*)$ to be the orthogonal complement of $\mathcal{N}_v$ with respect to the non-degenerate Tate pairing 
\[H^1(\operatorname{G}_v, \g)\times H^1(\operatorname{G}_v, \g^*)\rightarrow H^2(\operatorname{G}_v, \F_q(\bar{\chi}))\xrightarrow{\sim}\F_q.\] Set $\mathcal{N}_{\infty}=0$ and $\mathcal{N}_{\infty}^{\perp}=0$. The Selmer-condition $\mathcal{N}$ is the tuple $\{\mathcal{N}_v\}_{v\in S\cup X\cup \{\infty\}}$ and the dual Selmer condition $\mathcal{N}^{\perp}$ is $\{\mathcal{N}_v^{\perp}\}_{v\in S\cup X\cup\{\infty\}}$. Attached to $\mathcal{N}$ and $\mathcal{N}^{\perp}$ are the Selmer and dual-Selmer groups defined as follows:
      \[H^1_{\mathcal{N}}(\op{G}_{\Q,S\cup X}, \g):=\text{ker}\left\{ H^1(\operatorname{G}_{\Q, S\cup X}, \g)\xrightarrow{\operatorname{res}_{S\cup X}} \bigoplus_{v\in S\cup X} \frac{H^1(\operatorname{G}_v, \g)}{\mathcal{N}_v}\right\}\]
      and
      \[H^1_{\mathcal{N}^{\perp}}(\op{G}_{\Q,S\cup X}, \g^*):=\text{ker}\left\{ H^1(\operatorname{G}_{\Q, S\cup X}, \g^{*})\xrightarrow{\op{res}_{S\cup X}'} \bigoplus_{v\in S\cup X} \frac{H^1(\operatorname{G}_v, \g^*)}{\mathcal{N}_v^{\perp}}\right\}\]
      respectively. The following formula is due to Wiles (see \cite[Theorem 8.7.9]{NW}):
      \begin{equation}\label{wilesformula}\begin{split}h^1_{\mathcal{N}}(\op{G}_{\Q,S\cup X}, \g)-h^1_{\mathcal{N}^{\perp}}(\operatorname{G}_{\Q, S\cup X}, \g^{*})&=h^0(\op{G}_{\Q}, \g)-h^0(\op{G}_{\Q},\g^*)\\ &+\sum_{v\in S\cup X\cup \{\infty\}} \left(\dim \mathcal{N}_v-h^0(\op{G}_v, \g)\right).\\\end{split}\end{equation} Since $\bar{\rho}$ is odd, one has that $h^0(\op{G}_{\infty}, \g)=\dim \mathfrak{n}$. It follows from the above formula that the dimensions of the Selmer group and dual Selmer group coincide, i.e.
      \[h^1_{\mathcal{N}}(\op{G}_{\Q,S\cup X}, \g)=h^1_{\mathcal{N}^{\perp}}(\operatorname{G}_{\Q, S\cup X}, \g^{*}).\]The Selmer and dual Selmer groups fit into a long exact sequence called the Poitou-Tate sequence. We only point out that the cokernel of $\op{res}_{S\cup X}$ injects into $H^1_{\mathcal{N}^{\perp}}(\op{S}_{S\cup X}, \g^*)^{\vee}$. In particular, if the Selmer group is zero, then so is the dual Selmer group, in which case the restriction map $\operatorname{res}_{S\cup X}$ is an isomorphism. Since the spaces $\mathcal{N}_v$ at a trivial prime $v$ stabilize lifts only past mod $p^3$, it becomes necessary to produce a mod $p^3$ lift $\rho_3$ of $\bar{\rho}$ before applying the general lifting-strategy. All deformations $\rho_m$ discussed in this paper will have similitude character equal to $\kappa\mod{p^m}$.
      \par The three main steps are as follows:
      \begin{enumerate}
          \item first it is shown that there is a finite set of trivial primes $X_1$ disjoint from $S$ such that the representation $\bar{\rho}$ lifts to a mod $p^2$ representation $\rho_2$ which is unramified outside $S\cup X_1$.
          \item It is shown in \cite[section 5]{hamblenramakrishna} that there is a finite set of trivial primes $X_2\supset X_1$ disjoint from $S$ and a mod $p^3$ lift $\rho_3$ of $\rho_2$ which satisfies the following conditions
          \begin{itemize}
              \item $\rho_3$ is irreducible, i.e. does not contain a free rank one Galois stable $\text{W}(\F_q)/p^3$-submodule.
              \item It is unramified outside $S\cup X_2$.
              \item The lift $\rho_3$ is also arranged to satisfy conditions $\mathcal{C}_v$ at each prime $v\in S\cup X_2$.
          \end{itemize}
          This strategy for getting past mod-$p^2$ is based on the methods developed by Khare, Larsen and Ramakrishna in \cite{KLR}.
          \item At this stage, all that remains to be shown is that the set of primes $X_2$ may be further enlarged to a set of trivial primes $X$ which is disjoint from $S$ such that the Selmer group $H^1_{\mathcal{N}}(\operatorname{G}_{\Q,S\cup X}, \g)$ is equal to zero.
      \end{enumerate}
      \par The rest of the argument warrants some explanation. Since the Selmer group is zero, the map $\op{res}_{S\cup X}$ is an isomorphism. Suppose for $m\geq 3$ and $\rho_m$ is a mod $p^m$ lift of $\rho_3$ which is unramified outside $S\cup X$ and satisfies the conditions $\mathcal{C}_v$ at each prime $v\in S\cup X$. We show that $\rho_m$ may be lifted to $\rho_{m+1}$ which satisfies the same conditions. Since the dual Selmer group is zero, so is $\Sh^1_{S\cup X}(\g^*)$, and it follows from global-duality that $\Sh^2_{S\cup X}(\g)$ is zero. Since local condition $\mathcal{C}_v$ is liftable, there are no local obstructions to lifting ${\rho_m}$. The cohomological obstruction to lifting $\rho_{m}$ to $\rho_{m+1}$ is a class in $\Sh^2_{S\cup X}(\g)$ and hence is zero. As a result, $\rho_m$ does lift one more step to $\rho_{m+1}$. In order to complete the inductive argument, it is shown that $\rho_{m+1}$ can be chosen to satisfy the conditions $\mathcal{C}_v$. After picking a suitable global cohomology class $z\in H^1(\operatorname{G}_{\Q, S\cup X}, \g)$ and replacing $\rho_{m+1}$ by its twist $(\operatorname{Id}+p^{m}z)\rho_{m+1}$, this may be arranged. At each prime $v\in S\cup X$, there is a cohomology class $z_v\in H^1(\operatorname{G}_v, \g)$ such that the twist $(\operatorname{Id}+p^mz_v){\rho_{m+1}}_{\restriction \operatorname{G}_v}$ satisfies $\mathcal{C}_v$. Since we assume that $m\geq 3$, we have that $\mathcal{N}_v$ stabilizes $\mathcal{C}_v$. For $v\in S\cup X$, the elements $z_v$ are defined modulo $\mathcal{N}_v$. Since $\operatorname{res}_{S\cup X}$ is an isomorphism, the tuple
      $(z_v)\in \bigoplus_{v\in S\cup X} H^1(\operatorname{G}_v, \g)/\mathcal{N}_v$ arises from a unique global cohomology class $z$ which is unramified outside $S\cup X$. After replacing $\rho_{m+1}$ by $(\op{Id}+p^m z) \rho_{m+1}$, it satisfies the conditions $\mathcal{C}_v$ at each prime $v\in S\cup X$. This completes the inductive lifting argument.
\section{Preliminaries}
\par In this section, we prove a number of Galois theoretic results which will be applied in later sections. Let $M$ be a finite abelian group with $\op{G}_{\Q}$-action and $E$ be a number field. Denote by $E(M)$ the extension of $E$ \textit{cut out} by $M$. In other words, it is the Galois extension of $E$ which is fixed by the kernel of the action of $\op{G}_{E}$ on $M$. Let $M_1,\dots, M_k$ be finite abelian groups on which $\op{G}_{\Q}$ acts. Denote by $E(M_1,\dots, M_k)$ the composite of the fields $E(M_1)\cdots E(M_k)$. Let $K:=\Q(\bar{\rho}, \mu_p)$ and $L:=\Q(\bar{\rho})$ and set $\op{G}':=\op{Gal}(K/\Q)$ and $\op{G}:=\op{Gal}(L/\Q)$. Let $F$ be the subfield $\Q(\varphi_1,\dots, \varphi_n, \bar{\kappa})$ of $L$, where we recall that the characters $\varphi_1,\dots, \varphi_n$ are as in $\eqref{introducingbarrho}$. Denote by $N':=\operatorname{Gal}(K/F(\mu_p))$ and $N:=\operatorname{Gal}(L/F)$. The groups $\op{G},\op{G}', N$ and $N'$ are depicted in the following field diagram
\begin{equation*}
\begin{tikzpicture}[node distance = 1.5cm, auto]
      \node(Q) {$\Q.$};
      \node (L) [above of =Q]{$F$};
      \node (E) [above of=L, right of=L] {$F(\mu_p)$};
      \node (F) [above of=L, left of =L] {$\Q(\bar{\rho})$};
      \node (K) [above of=E, left of=E] {$\Q(\bar{\rho},\mu_p)$};
      \draw[-] (Q) to node {} (L);
      \draw[-] (L) to node {} (E);
      \draw[-] (L) to node {\scriptsize$N$} (F);
      \draw[-] (F) to node {} (K);
      \draw[-] (E) to node [swap]{\scriptsize$N'$} (K);
      \end{tikzpicture}
      \end{equation*}
      Condition $\eqref{thc3}$ of Theorem $\ref{main}$ asserts that the image of $\bar{\rho}$ contains $U_1(\F_q)$. Therefore $N$ may be identified with $\bar{\rho}(N)=U_1(\F_q)$. In particular the abelianization $N^{ab}$ may be identified with $U_1(\F_q)/U_2(\F_q)$. Since $N$ is a $p$-group and $[F(\mu_p):F]$ is coprime to $p$, it follows that $\Q(\bar{\rho})$ and $F(\mu_p)$ are linearly disjoint over $F$. It follows that $N$ is canonically isomorphic to $N'$. The inclusion of $\mathcal{T}$ into $\op{B}$ is a section of the quotient map $\op{B}\rightarrow \mathcal{T}$. This induces a semi-direct product decomposition $\op{B}=U_1\rtimes \mathcal{T}$. Let $\mathcal{T}'$ be the intersection of the image of $\bar{\rho}$ with $\mathcal{T}$. The group $\op{G}$ may be identified with the image of $\bar{\rho}$. It is easy to see that $\op{G}$ has a semi-direct product decomposition $\op{G}\simeq \bar{\rho}(\op{G})=U_1(\F_q)\rtimes \mathcal{T}'$.\begin{Lemma}\label{l1}
Suppose $0<|k|\leq 2n-1$, there is an isomorphism of $\F_q[\G]$-modules
\[(\g)_k/(\g)_{k+1}\simeq \bigoplus_{ht \lambda=k} \F_q(\sigma_{\lambda}).\]On the other hand, \[(\g)_0/(\g)_1=\mathfrak{b}/\mathfrak{n}\simeq \mathfrak{t}.\]
\end{Lemma}
\begin{proof}
Let $\lambda$ be of height $k$. Let $X\in (\g)_{\lambda}$, we observe that
\[\bar{\rho}(g)\cdot X \cdot \bar{\rho}(g)^{-1}\equiv \sigma_{\lambda}(g) X \mod{(\g)_{k+1}}.\] Likewise, for $X\in \mathfrak{b}$, the conjugation action on $X$ modulo $\mathfrak{n}$ is trivial.
\end{proof}

\par Let $\zeta$ be a non-zero element of $\F_q(\bar{\chi})$. For $i=1,\dots, n$, set $\delta_{i,j}=\zeta$ if $i=j$ and $0$ otherwise. Likewise, for roots $\lambda$ and $\gamma$, set $\delta_{\lambda, \gamma}$ to equal $\zeta$ if $\lambda=\gamma$ and $0$ otherwise. Denote by $X_{\lambda}^*$ and $H_i^*$ the elements of $\g^*$ which are defined by the following relations: \begin{equation}\label{XHdual}\begin{split}& X_{\lambda}^*(X_{\gamma})=\delta_{\lambda,\gamma}\text{ and }X_{\lambda}^*(H_i)=0,\\
      & H_i^*(X_{\lambda})=0 \text{ and }H_i^*(H_j)=\delta_{i,j}.\end{split}\end{equation} The element $H_i^*\in \g^*$ should not be confused with $L_i\in \mathfrak{t}^*$. Let $(\g^*)_{\bar{\chi}}$ be the span of $H_1^*,\dots, H_n^*$ and $(\g^*)_{\bar{\chi}\sigma_{\lambda}}$ the span of $X_{-\lambda}^*$. Let $P$ be a Galois-stable subgroup of $\g^*$. Associated to $P$ are its eigenspaces for the action of $\bar{\rho}^{-1}(\mathcal{T})$. For $\lambda\in \Phi\cup \{1\}$, set $P_{\bar{\chi}\sigma_{\lambda}}$ to be the intersection of $P$ with $(\g^*)_{\bar{\chi}\sigma_{\lambda}}$. Likewise, associate to a Galois stable subgroup $Q\subseteq \g$, an eigenspace $Q_{\sigma_{\lambda}}$. Define $Q_1$ to be the intersection $Q\cap \mathfrak{t}$. For $\lambda \in \Phi$, denote by $Q_{\sigma_{\lambda}}$ the intersection $Q\cap (\g)_{\lambda}$.
      \par The representation $\bar{\rho}$ factors through $\op{G}$. Let $\mathbb{T}$ be the subgroup of $\op{G}'$ consisting of $g$ such that $\bar{\rho}(g)\in \mathcal{T}$. For $\lambda\in \Phi\cup \{1\}$, $\mathbb{T}$ acts on $P_{\bar{\chi}\sigma_{\lambda}}$ by the character $\bar{\chi}\sigma_{\lambda}$ and on $Q_{\sigma_{\lambda}}$ by the character $\sigma_{\lambda}$. Since the characters $\sigma_{\lambda}$ are assumed to be distinct, it is easy to see that
      \[P_{\bar{\chi}\sigma_{\lambda}}=\{p\in P| t\cdot p=\bar{\chi}\sigma_{\lambda}(t)p\text{ for }t\in \mathbb{T}\}\]
      \[Q_{\sigma_{\lambda}}=\{q\in Q| t\cdot q=\sigma_{\lambda}(t)q\text{ for }t\in \mathbb{T}\}.\] The order of $\mathbb{T}$ is coprime to $p$, hence Maschke's theorem asserts that any finite dimensional $\F_p[\op{G}']$-module $M$ decomposes into a direct sum $M=\bigoplus_{\tau} M_{\tau}$, where $\tau$ is a character of $\mathbb{T}$ and $M_{\tau}$ is the $\tau$-eigenspace $M_{\tau}:=\{m\in M| g\cdot m=\tau(g) m\}$. Thus, we have the next Lemma, which follows from the discussion above.
          \begin{Lemma}\label{Pdecomposition} Let $P\subseteq \g^*$ and $Q\subseteq \g$ be Galois-stable subgroups.
          \begin{enumerate}
              \item As a $\mathbb{T}$-module, $P$ decomposes into a direct sum of eigenspaces: \[P=\bigoplus_{\lambda\in \Phi\cup\{1\}} P_{\bar{\chi}\sigma_{\lambda}}.\]
              \item As a $\mathbb{T}$-module, $Q$ decomposes into a direct sum of eigenspaces:
    \[Q=\bigoplus_{\lambda\in \Phi\cup \{1\}} Q_{\sigma_{\lambda}}.\]
          \end{enumerate}
      \end{Lemma}

     Set the height of the formal symbol "$1$" to be equal to zero. Fix a total order on $\Phi\cup \{1\}$ such that $\op{ht}(\lambda)\leq \op{ht}(\gamma)$ if $\lambda\leq \gamma$. 
\begin{Lemma}\label{mainin}
\begin{enumerate} Let $P$ be a non-zero Galois stable subgroup of $\g^*$. Let $Q$ be a non-zero Galois stable subgroup of $\g$. The following statements hold:
    \item \label{43c1} $P_{\bar{\chi}\sigma_{2L_1}}$ is equal to $(\g^*)_{\bar{\chi}\sigma_{2L_1}}$,
    \item \label{43c2} $Q_{\sigma_{2L_1}}$ is equal to $(\g)_{\sigma_{2L_1}}$.
\end{enumerate}
\end{Lemma}
\begin{proof}
It follows from Lemma $\ref{Pdecomposition}$ that $P$ decomposes into $\bigoplus_{\lambda\in \Phi\cup \{1\}} P_{\bar{\chi}\sigma_{\lambda}}$. By condition $\eqref{thc5}$ of Theorem $\ref{main}$, the image of $\sigma_{2L_1}$ spans $\F_q$. Since $\bar{\chi}$, takes values in $\F_p^{\times}$, the same is true for the image of $\bar{\chi}\sigma_{2L_1}$. Therefore, if $P_{\bar{\chi}\sigma_{2L_1}}$ is not zero, then $P_{\bar{\chi}\sigma_{2L_1}}=(\g^*)_{\bar{\chi}\sigma_{2L_1}}$. Suppose by way of contradiction that $P_{\bar{\chi}\sigma_{2L_1}}=0$. Then we may choose $\gamma\in \Phi\cup \{1\}$ such that:
\begin{itemize}
    \item $P_{\bar{\chi}\sigma_{\gamma}}\neq 0$,
    \item $P_{\bar{\chi}\sigma_{\lambda}}=0$ for all $\lambda>\gamma$.
\end{itemize} By assumption, $\gamma$ is not the maximal root $2L_1$. There exists $\gamma_1\in \Phi$ such that the difference $\mu:=\gamma_1-\gamma$ is in $ \Phi^{+}$. This can be shown by considering all possibilities for $\gamma$:
\begin{enumerate}
\item $\gamma=1$, then let $\mu=\gamma_1$ be any positive root,
    \item $\gamma=2L_i$ for $i> 1$, then $\mu=L_{i-1}-L_i$ and $\gamma_1=L_{i-1}+L_i$,
    \item $\gamma=-2L_i$ for $i>1$, then $\mu=L_{i-1}+L_i$ and $\gamma_1=L_{i-1}-L_i$,
    \item $\gamma=-2L_1$, then $\mu=L_{1}+L_2$ and $\gamma_1=-L_{1}+L_2$,
    \item $\gamma=L_i+L_j$ for $i<j$, then $\mu=L_i-L_j$ and $\gamma_1=2L_i$,
    \item $\gamma=-L_i-L_j$ for $i<j$, then $\mu=L_i-L_j$ and $\gamma_1=-2L_j$,
    \item $\gamma=L_i-L_j$ for $i\neq j$, then $\mu=L_i+L_j$ and $\gamma_1=2L_i$.
\end{enumerate}
By condition $\eqref{thc3}$ of Theorem $\ref{main}$, the root subgroup $U_{\mu}$ is contained in the image of $\bar{\rho}$. Let $g\in \op{G}_{\Q}$ be such that $\bar{\rho}(g)\neq \op{Id}$ and $\bar{\rho}(g)\in U_{\mu}$. Let $p\in P_{\bar{\chi}\sigma_{\gamma}}$ be a non-zero element. We show that the projection of $g\cdot p$ to $P_{\bar{\chi} \sigma_{\gamma_1}}$ is non-zero. Express $\bar{\rho}(g)$ as $e^{-X}:=\sum_{i=0}^{2n-1} \frac{(-X)^i}{i!}$, where $X\in (\g)_{\mu}$. For $Y\in (\g)_k$, the following identity is well known (see \cite[Exercise 3.9.14]{hall}): 
\[\begin{split}g^{-1}\cdot Y=& e^X Y e^{-X}= e^{\op{ad}_X}(Y)= \sum_{i=0}^{2n-1} \frac{(\op{ad}_X)^i(Y)}{i!}\\ =& Y+[X,Y]+\frac{1}{2!}[X,[X,Y]]+\frac{1}{3!}[X,[X,[X,Y]]]+\dots.\end{split}\] We note here that the above formula makes sense since it is assumed that $p>2n$. Note that $\mu$ is a positive root and hence,  
\[g^{-1}\cdot Y-Y=[X,Y]\mod{(\g)_{\op{ht}(\mu)+k+1}}.\]Note that since $-\gamma_1+\mu=-\gamma$ is a root, \[[(\g)_{\mu},(\g)_{-\gamma_1}]=(\g)_{-\gamma}\] (cf. \cite[p. 39]{humphreys}). Letting $Y$ run through an appropriate basis of $\g$, it follows from the above identity that $g\cdot p-p$ can be expressed as a sum $a+b$ where $a\neq 0$ is in $P_{\bar{\chi}\sigma_{\gamma_1}}$ and $b\in \bigoplus _{\lambda>\gamma_1}P_{\bar{\chi}\sigma_{\lambda}}$. In particular, this shows that the projection of $g\cdot p$ to $P_{\bar{\chi}\sigma_{\gamma_1}}$ is non-zero.
\par Since $\gamma_1=\mu+\gamma$ and $\mu\in \Phi^+$, the height of $\gamma_1$ is strictly larger than the height of $\gamma$. As a result, $\gamma_1>\gamma$. Therefore, the subgroup $P_{\bar{\chi}\sigma_{\gamma_1}}=0$. This contradiction shows that $\gamma=2L_1$ and $P_{\bar{\chi}\sigma_{2L_1}}\neq 0$. This concludes part $\eqref{43c1}$. The proof of part $\eqref{43c2}$ is similar and is left to the reader.
\end{proof}
For $\lambda\in \Phi\cup \{1\}$, set \[N_{\lambda}=\begin{cases}
1\text{ if }\lambda\in \Delta,\\
0\text{ otherwise.}
\end{cases}\]
\begin{Lemma}\label{l2}
Let $\lambda\in \Phi\cup \{1\}$ and $\sigma_{\lambda}$ the associated character. The following assertions are satisfied:
\begin{enumerate}
\item\label{l2c1}
$\dim \Hom( N, \F_q(\sigma_{\lambda}))^{\op{G}/N}=N_{\lambda}$.
\item\label{l2c2}
$\dim \Hom(N', \F_q(\sigma_{\lambda})^*)^{\op{G}'/N'}=0$.
\item\label{l2c3} For $k\neq 1$, \[ H^1(\op{G}, (\g)_k/(\g)_{k+1})=0.\]
On the other hand,
\[h^1 (\op{G}, (\g)_1/(\g)_2)= \dim \mathfrak{t}.\]
\item\label{l2c4} For all $k$,
 $h^1(\op{G}', ((\g)_k/(\g)_{k+1})^*)=0$.
\end{enumerate}
\end{Lemma}
\begin{proof}
By condition $\ref{thc3}$ of Theorem $\ref{main}$, $N$ may be identified with $U_1(\F_q)$. The abelianization $N^{ab}$ is \[U_1/U_2(\F_q)\simeq \bigoplus_{\gamma\in \Delta} \F_q(\sigma_{\gamma}).\]
By condition $\eqref{thc5}$ of Theorem $\ref{main}$, any $\op{G}/N$ equivariant map $F:\F_q(\sigma_{\gamma})\rightarrow \F_q(\sigma_{\lambda})$ is determined by the image of any nonzero element, hence \[\dim \op{Hom}(\F_q(\sigma_{\gamma}),\F_q(\sigma_{\lambda}))^{\op{G}/N}\leq 1.\] We have that \[F(\sigma_{\gamma}(g_1)\sigma_{\gamma}(g_2))=\sigma_{\lambda}(g_1)\sigma_{\lambda}(g_2)F(1)=F(\sigma_{\gamma}(g_1))F(\sigma_{\gamma}(g_2))F(1).\] Since the image of $\sigma_{\lambda}$ spans $\F_q$, it follows that $F$ is a scalar multiple of an element of $\op{Gal}(\F_q/\F_p)$. By assumption, if $\lambda\neq \gamma$, the characters $\sigma_{\lambda}$ and $\sigma_{\gamma}$ are not equal up to a twist of $\op{Gal}(\F_q/\F_p)$. Therefore,
\[\Hom( \F_q(\sigma_{\gamma}), \F_q(\sigma_{\lambda}))^{\op{G}/N}=\begin{cases} \F_q \mbox{ if $\sigma_{\lambda}=\sigma_{\gamma}$,}\\
0 \mbox{ otherwise,}
\end{cases}\]and part $\eqref{l2c1}$ follows this.

Observe that $N'$ is isomorphic to $N$ and $\op{G}'/N'$ is the Galois group $\op{Gal}(\Q(\{\varphi_i\}, \bar{\kappa},\bar{\chi})/\Q)$. By condition $\eqref{thc4}$, the characters $\sigma_{\gamma}$ and $\sigma_{\lambda}^*=\bar{\chi}\sigma_{-\lambda}$ are not equivalent up to a twist of $\op{Gal}(\F_q/\F_p)$. Part $\eqref{l2c2}$ follows via the same reasoning as part $\eqref{l2c1}$.
\par The order of $\op{G}/N$ is coprime to $p$. By inflation-restriction and part $(\ref{l2c1})$
\begin{equation*}
\begin{split}
& \dim  H^1(\op{G}, (\g)_k/(\g)_{k+1})\\
= & \dim  \Hom(N, (\g)_k/(\g)_{k+1})^{\op{G}/N}\\
= & \sum_{\lambda\in\Phi, ht(\lambda)=k} \dim  \Hom(N, \F_q(\sigma_{\lambda}))^{\op{G}/N}\\
=  & \sum_{\lambda\in \Phi, ht(\lambda)=k} N_{\lambda}.\\
\end{split}
\end{equation*}

It follows that if $k\neq 1$, 
\[ H^1(\op{G}, (\g)_k/(\g)_{k+1})=0\]and that
\[h^1(\op{G}, (\g)_1/(\g)_{2})=\# \Delta =n=\dim \mathfrak{t}.\] This concludes part $\eqref{l2c3}$.
\par The order of $\op{G}'/N'$ is coprime to $p$. Therefore, by inflation-restriction,
\begin{equation*}
\begin{split}
& \dim  H^1(\op{G}', (\g)_k/(\g)_{k+1})\\
= & \dim  \Hom(N', (\g)_k/(\g)_{k+1})^{\op{G}'/N'}\\
= & \sum_{\lambda\in\Phi, ht(\lambda)=k} \dim  \Hom(N', \F_q(\sigma_{\lambda})^*)^{\op{G}'/N'}\\
=& 0.
\end{split}
\end{equation*} This concludes the proof of part $\eqref{l2c4}$.
\end{proof}
\begin{Def}\label{perpdef}
Let $(\g)_k^{\perp}\subset \g^*$ be the subspace of $\g^*$ consisting of $f\in \g^*$ for which $f_{\restriction (\g)_k}=0$.
\end{Def}
\begin{Remark}
For $k>-2n+1$, the submodule $(\g)_k^{\perp}\neq 0$ and by Lemma $\ref{mainin}$, \[(\g)_{k,\bar{\chi}\sigma_{2L_1}}^{\perp}\simeq(\g)_{\bar{\chi}\sigma_{2L_1}}^*.\] 
\end{Remark}
\begin{Lemma}\label{l3}
Let $k$ be an integer,
\begin{enumerate}
\item\label{l3c1}
$H^1(\op{G},\g/(\g)_k)=0$ and $H^1(\op{G}',\g/(\g)_k)=0$,
\item\label{l3c2}
$H^1(\op{G}',(\g)_k^{\perp})=0$.
\end{enumerate}
\end{Lemma}
\begin{proof}
We begin with the proof of part $\eqref{l3c1}$. Consider the case when $k\leq 1$. By part $\eqref{l2c3}$ of Lemma $\ref{l2}$, for $i\leq 1$, 
\begin{equation*}
H^1(\op{G}, (\g)_{i-1}/(\g)_{i})=0.
\end{equation*}
and hence there is an injection
\begin{equation*}H^1(\op{G}, \g/(\g)_i)\hookrightarrow H^1(\op{G}, \g/(\g)_{i-1}).\end{equation*}
We deduce that $H^1(\op{G}, \g/(\g)_k)=0$.
\par Next consider the case $k>1$. Associated to 
\begin{equation*}
0\rightarrow (\g)_1/(\g)_k\rightarrow \g/(\g)_k\rightarrow \g/(\g)_1\rightarrow 0 \end{equation*}
is the long exact sequence in cohomology. It follows from $\ref{mainin}$ that any non-zero submodule of $\g$ has a non-trivial $\sigma_{2L_1}$ eigenspace for the $\mathbb{T}$-action. As a consequence, $H^0(\operatorname{G}_{\Q}, \g)=0$. Since Lemma $\ref{l2}$ asserts that \[H^1(\op{G}, (\g)_k/(\g)_{k+1})=0,\] we have that $H^0(\op{G}, \g/(\g)_{k+1})$ surjects onto $H^0(\op{G}, \g/(\g)_{k})$ for $k>1$. As a result, $H^0(\op{G}, \g)$ surjects onto $H^0(\op{G}, \g/(\g)_{k})$, and therefore, \[H^0(\op{G}, \g/(\g)_{k})=0.\]
\par Since it has been shown that $H^1(\op{G}, \g/(\g)_1)=0$, we have a short exact sequence
\[0\rightarrow  H^0(\op{G}, \g/(\g)_1)\rightarrow H^1(\op{G}, (\g)_1/(\g)_k)\rightarrow H^1(\op{G}, \g/(\g)_k)\rightarrow 0.\]It suffices to show that
\[\dim H^0(\op{G}, \g/(\g)_1)\geq   \dim H^1(\op{G}, (\g)_1/(\g)_k).\]
Condition $\eqref{thc4}$ of Theorem $\ref{main}$ stipulates that for $\lambda \in \Phi$, $\sigma_{\lambda}$ is not equal to $\sigma_1=1$. Therefore for $i\leq 0$,
\begin{equation*}
H^0(\op{G}, (\g)_{i-1}/(\g)_{i})=\bigoplus_{ \text{ht} \gamma=i-1} H^0(\op{G}, \F_q(\sigma_{\gamma}))=0.
\end{equation*}
For $i\leq 0$, we deduce that
\begin{equation*}
H^0(\op{G}, (\g)_i/(\g)_1)\xrightarrow{\sim} H^0(\op{G}, (\g)_{i-1}/(\g)_1).
\end{equation*}
Composing these isomorphisms we have that
\begin{equation*}
(\g)_0/(\g)_1=H^0(\op{G},(\g)_0/(\g)_1)\xrightarrow{\sim} H^0(\op{G},\g/(\g)_1).
\end{equation*}
We have deduced that
\begin{equation*}
h^0(\op{G}, \g/(\g)_1)=\dim (\g)_0/(\g)_1=\dim \mathfrak{t}.
\end{equation*}
By Lemma $\ref{l2}$ part $\eqref{l2c3}$,
\begin{equation*}
h^1(\op{G}, (\g)_1/(\g)_2)=\dim \mathfrak{t}=h^0(\op{G}, \g/(\g)_1).
\end{equation*}
By Lemma $\ref{l2}$, for $i\geq 2$, we have that \begin{equation*}H^1(\op{G}, (\g)_i/(\g)_{i+1})=0.\end{equation*} and it follows that $H^1(\op{G}, (\g)_2/(\g)_k)=0$.
Hence it follows that
\begin{equation}\label{deltaiso}h^1(\op{G},(\g)_1/(\g)_k)\leq h^1(\op{G}, (\g)_1/(\g)_2).\end{equation}We conclude that
\[h^1(\op{G},(\g)_1/(\g)_k)\leq h^0(\op{G}, \g/ (\g)_1).\]Therefore we conclude that $H^1(\op{G}, \g/(\g)_k)=0$.
\par Since $[L:K]$ is coprime to $p$, from a direct application of inflation-restriction it follows that $H^1(\op{G}',\g/(\g)_k)=0$.
We have proved part $(\ref{l3c1})$.
\par Consider the short exact sequence \[0\rightarrow (\g)_{i-1}^{\perp}\rightarrow (\g)_i^{\perp} \rightarrow ((\g)_{i-1}/(\g)_{i})^*\rightarrow 0\] and the associated sequence in cohomology. By Lemma $\ref{l2}$, 
\begin{equation*}H^1(\op{G}', ((\g)_{i-1}/(\g)_{i})^*)=0
\end{equation*}
from which we deduce that 
\begin{equation*}
H^1(\op{G}',(\g)_{i-1}^{\perp})\rightarrow H^1(\op{G}', (\g)_i^{\perp})
\end{equation*}
is a surjection for all $i$. As \[(\g)_{-2n+1}^{\perp}\simeq (\g/(\g)_{-2n+1})^*=0\] we deduce that $H^1(\op{G}', (\g)_k^{\perp})=0$.
\end{proof}
\par For $\psi$ in $H^1(\operatorname{G}_{\Q}, \g^*)$, the restriction $\psi_{\restriction \op{G}_K}:\op{G}_K\rightarrow \g^*$ is a homomorphism since the action of $\op{G}_K$ on $\g^*$ is trivial. Let $K_{\psi}\supseteq K$ be the extension \textit{cut out} by $\psi$, i.e. $K_{\psi}$ is the smallest extension of $K$ which is fixed by the kernel of $\psi_{\restriction \op{G}_K}$. Identify $\op{Gal}(K_{\psi}/K)$ with $\psi(\op{G}_K)\subseteq \g^*$ and let $J_{\psi}\subseteq K_{\psi}$ be the subfield for which $\op{Gal}(K_{\psi}/J_{\psi})\simeq \psi(\op{G}_K)_{\bar{\chi}\sigma_{2L_1}}$. Likewise for $f$ in $H^1(\operatorname{G}_{\Q}, \g)$ denote by $L_f$, the extension of $L$ cut out by $f$. Set $K_f$ to be the composite of $L_f$ with $K$. Since $p\nmid [K:L]$, we have that $\op{Gal}(K_f/K)\simeq \op{Gal}(L_f/L)$.
\begin{Lemma}\label{l4}Let $\mathcal{J}\supseteq S$ be a finite set of primes. Let $f\in H^1(\op{G}_{\Q,\mathcal{J}}, \g)$ and $\psi\in H^1(\op{G}_{\Q,\mathcal{J}}, \g)$ be a non-zero cohomology classes. Then the following assertions are satisfied:
\begin{enumerate}
\item\label{l4p1} $L_{f}\supsetneq  L$ (equivalently, $K_f\supsetneq K$),
\item\label{l4p2}
$K_{\psi}\supsetneq J_{\psi}$, in particular, $K_{\psi}\supsetneq K$.
\end{enumerate}
\begin{proof}
\par For part $\eqref{l4p2}$, recall that Lemma $\ref{l3}$ asserts that
$H^1(\op{G}', \g^*)=0$. Therefore, the restriction $\psi_{\restriction \op{G}_K}$ is not zero. This shows that $K_{\psi}\supsetneq K$. That $K_{\psi}$ strictly contains $J_{\psi}$ is a direct consequence of Lemma $\ref{mainin}$. Part $\eqref{l4p1}$ also follows from Lemma $\ref{l3}$.
\end{proof}
\end{Lemma}

\begin{Lemma}Let $P\subseteq \g^*$ (resp. $Q\subseteq \g$) be a nonzero Galois-stable subgroup and $\iota_P:P\rightarrow \g^*$ (resp. $\iota_Q:Q\rightarrow \g$) denote the inclusion. Then, \label{y1}
\begin{enumerate}
    \item\label{48c1} 
$\Hom_{\F_p}(P,\g^*)^{\op{G}'}=\F_q\cdot \iota_P$,
\item\label{48c2} 
$\Hom_{\F_p}(Q,\g^*)^{G}=\F_q\cdot \iota_Q$.
\end{enumerate}
\end{Lemma}
\begin{proof}
We prove part $\eqref{48c1}$, the proof of $\eqref{48c2}$ is similar. Let $\Phi'=\Phi\cup \{1\}$ and $m:=\# \Phi'$. Enumerate $\Phi'=\{\gamma_1,\dots, \gamma_m\}$, so that $\gamma_i>\gamma_j$ if $i<j$. The root $\gamma_1$ is the highest root $2L_1$ and $\gamma_m$ is $-2L_1$. Let $W_i$ be the $\op{G}'$-submodule of $\g^*$ defined by
\[W_i:= \left(\bigoplus_{j\leq  i} (\g^*)_{\bar{\chi}\sigma_{\gamma_j}}\right).\] Setting $P_i=P\cap W_i$, Lemma $\ref{Pdecomposition}$ asserts that 
\[P=\bigoplus_{\lambda\in \Phi'} P_{\bar{\chi}\sigma_{\lambda}}\text{ and } P_i= \left(\bigoplus_{j\leq  i} P_{\bar{\chi}\sigma_{\gamma_j}}\right).\]Let $\p\in \Hom(P,\g^*)^{\op{G}'}$, we show that there exists $\beta\in \F_q$ such that $\varphi=\beta\cdot \iota_P$.
\par First, it is shown that the restriction of $\varphi$ to $P_1$ is equal to $\beta\cdot  \iota_{P_1}$, where $\iota_{P_1}$ is the inclusion of $P_1$ in $\g^*$ and $\beta\in \F_q$.  Note that $P_1=(P)_{\bar{\chi}\sigma_{2L_1}}$ is a $\op{G}'$-submodule. Since $P$ is assumed to be non-zero, by Lemma $\ref{mainin}$, we have that $P_1$ is equal to $(\g^*)_{\bar{\chi}\sigma_{2L_1}}$, a one-dimensional $\F_q$-vector space. Without loss of generality, consider the case when $\varphi(P_1)$ is non-zero. By Lemma $\ref{mainin}$, we have that $\varphi(P_1)$ must contain $(\g^*)_{\bar{\chi}\sigma_{2L_1}}=P_1$. From the inequalities \[\dim_{\F_p} P_1\leq \dim_{\F_p} \varphi(P_1)\leq \dim_{\F_p} P_1,\] conclude that $\varphi(P_1)=P_1$ if $\varphi(P_1)\neq 0$. We have shown that $\varphi_{\restriction P_1}$ is a $\op{G}'$-endomorphism of $P_1$ and that $P_1=(\g^*)_{\bar{\chi}\sigma_{2L_1}}$. Note that $\varphi$ is not a priori $\F_q$-linear, however since the image of $\sigma_{2L_1}$ spans $\F_q$ over $\F_p$, it follows that \[\op{Hom}((\g^*)_{\bar{\chi}\sigma_{2L_1}}, (\g^*)_{\bar{\chi}\sigma_{2L_1}})^{\op{G}'}\simeq \F_q.\] Consequently, there exists $\beta\in \F_q$ be such that $\varphi_{\restriction P_1}=\beta \cdot \iota_{P_1}$. \par 
The next step is to show that the map $\varphi- \beta \cdot \iota_P: P/P_1\rightarrow \g^*$ is equal to zero. This will conclude the proof. We show that 
\[\op{Hom}(P/P_1, \g^*)^{\op{G}'}=0.\]For $i\geq 2$, from the short exact sequence
\[0\rightarrow  P_i/P_{i-1}\rightarrow P/P_{i-1}\rightarrow P/P_i\rightarrow 0,\] we have
\[0\rightarrow \op{Hom}(P/P_i, \g^*)^{\op{G}'}\rightarrow \op{Hom}(P/P_{i-1}, \g^*)^{\op{G}'}\rightarrow \op{Hom}(P_i/P_{i-1}, \g^*)^{\op{G}'} .\] It suffices to show that 
\[\op{Hom}(P_i/P_{i-1}, \g^*)^{\op{G}'}=0\]for $i\geq 2$. Let $\phi\in \op{Hom}(P_i/P_{i-1}, \g^*)^{\op{G}'}$ and set $R:=\phi(P_i/P_{i-1})\subset \g^*$. Note that $\op{G}'$ does act on $P_i/P_{i-1}$ via the character $\bar{\chi} \sigma_{\gamma_i}$. Consider two cases, first, consider the case when $\gamma_i$ is the formal symbol "1", i.e. $\sigma_{\gamma_i}=1$. Since $\bar{\chi}\sigma_{\gamma_i}=\bar{\chi}$ takes values in $\F_p$ and $\varphi$ is $\F_p$-linear, it follows that $\op{G}'$ acts on $R$ via the character $\bar{\chi}$. On the other hand, if $R\neq 0$, Lemma $\ref{mainin}$ asserts that it contains the $\bar{\chi} \sigma_{2L_1}$-eigenspace of $\g^*$. The assumption $\sigma_{2L_1}\neq \sigma_{-2L_1}$ implies that $\sigma_{2L_1}^2\neq 1$, and in particular, $\sigma_{2L_1}\neq 1$. Therefore, $R$ is equal to zero and thus it has been shown that \[\op{Hom}(P_i/P_{i-1}, \g^*)^{\op{G}'}=0\] when $\gamma_i$ is equal to "1". Consider the case when $\gamma_i\in \Phi$. In this case, $W_i/W_{i-1}$ is a one-dimensional $\F_q$ vector space and $P_i/P_{i-1}$ is an $\F_p$-subspace on which $\op{G}'$ acts via $\bar{\chi}\sigma_{\gamma_i}$. Assume without loss of generality that $P_i/P_{i-1}\neq 0$. Since the $\F_p$-span of the image of $\bar{\chi}\sigma_{\gamma_i}$ is $\F_q$, it follows that $P_i/P_{i-1}=W_i/W_{i-1}\simeq \F_q(\bar{\chi}\sigma_{\gamma_i})$. If $R$ is non-zero, it follows from Lemma $\ref{mainin}$, that $R$ contains $(\g^*)_{\bar{\chi}\sigma_{\gamma_1}}$. The dimension of $P_i/P_{i-1}$ over $\F_p$ is equal to that of $(\g^*)_{\bar{\chi}\sigma_{\gamma_1}}$ as they are both one-dimensional over $\F_q$. Since
\[\dim_{\F_p} R\leq \dim_{\F_p} P_i/P_{i-1},\] it follows that $R$ is equal to $(\g^*)_{\bar{\chi}\sigma_{\gamma_1}}$. Therefore, we have that \[\op{Hom}(P_i/P_{i-1}, R)^{\op{G}'}\simeq \op{Hom}(\F_q(\bar{\chi}\sigma_{\gamma_i}), \F_q(\bar{\chi}\sigma_{\gamma_1}))^{\op{G}'}=0\]since $\sigma_{\gamma_i}$ and $\sigma_{\gamma_1}$ are not twist equivalent. It has thus been shown that $\phi=0$. This concludes the proof of the Lemma.

\end{proof}
\begin{Cor}\label{Coradd}
The following statements hold:
\begin{enumerate}
    \item\label{coradd1}
    let $P_1$ and $P_2$ be Galois-stable subgroups of $\g^*$ such that there is an isomorphism $\phi:P_1\xrightarrow{\sim} P_2$ of Galois modules. Then $P_1=P_2$ and $\phi$ is multiplication by a scalar.
    \item \label{coradd2}
    Let $Q_1$ and $Q_2$ be Galois-stable subgroups of $\g$ such that there is an isomorphism $\phi:Q_1\xrightarrow{\sim} Q_2$ of Galois modules. Then $Q_1=Q_2$ and $\phi$ is multiplication by a scalar.
\end{enumerate}
\end{Cor}
\begin{proof}
We prove part $\eqref{coradd1}$, part $\eqref{coradd2}$ is identical. Let $\iota_{P_i}:P_i\hookrightarrow \g^*$ be the inclusion. By Proposition $\ref{y1}$, the two inclusions $\iota_{P_1}$ and $\iota_{P_2}\circ \phi$ are the same upto a scalar. The assertion follows.
\end{proof}
Let $Q$ be a $\op{G}$-submodule of $\g$, by Lemma $\ref{Pdecomposition}$, the projection of $Q$ to $(\g)_{-2L_1}$ equals $Q_{\sigma_{-2L_1}}$. For convenience of notation, let $Q_{-2L_1}$ denote $Q_{\sigma_{-2L_1}}$.
\begin{Lemma}\label{fullrankLemma}
Let $Q$ be a Galois-stable submodule of $\g$ for which $ Q_{-2L_1}\neq 0$, then $Q=\g$.
\end{Lemma}
\begin{proof}
Let $P:=\{\gamma\in \g^*\mid \gamma(x)=0 \text{ for }x\in Q\}$. The assumption on $Q$ implies that $P_{\bar{\chi}\sigma_{2L_1}}\neq (\g^*)_{\bar{\chi}\sigma_{2L_1}}$. Since the image of $\bar{\chi}\sigma_{2L_1}$ spans $\F_q$, $P_{\bar{\chi}\sigma_{2L_1}}\neq (\g^*)_{\bar{\chi}\sigma_{2L_1}}$ implies that  $P_{\bar{\chi}\sigma_{2L_1}}=0$ By Lemma $\ref{mainin}$, $P=0$, and therefore, $Q=\g$.
\end{proof}

\begin{Lemma}\label{22Dec5} The following statements hold:
\begin{enumerate}
    \item\label{411c1} the fields $K=\Q(\bar{\rho},\mu_p)$ and $\Q(\mu_{p^2})$ are linearly disjoint over $\Q(\mu_p)$.
    \item\label{411c2} Let $\mathcal{J}\supseteq S$ be a finite set of prime numbers, $\psi_1,\dots, \psi_t\in H^1(\operatorname{G}_{\Q,\mathcal{J}},\g^*)$ and set $K_j:=K_{\psi_j}$ for $j=1,\dots, t$. Then the composite $K_1\cdots K_t$ and $\Q(\mu_{p^2})$ are linearly disjoint over $\Q(\mu_p)$.
\end{enumerate}
\end{Lemma}
\begin{proof}
Suppose by way of contradiction that $\Q(\mu_{p^2})\subseteq K$. Set $V:=\op{Gal}(K/F(\mu_{p^2}))$ and $\mathcal{A}:=\op{G}'/N'=\op{Gal}(F(\mu_p)/\Q)$. For $n\in N'^{ab}$ and $g\in \mathcal{A}$, let $\tilde{n}$ and $\tilde{g}$ be lifts of $n$ and $g$ to $N'$ and $\op{G}'$ respectively. The action of $\mathcal{A}$ on $N'^{ab}$ is induced by conjugation, defined by $g\cdot n:= \tilde{g}\tilde{n}\tilde{g}^{-1}\mod{[N',N']}$. The groups $N$ and $N'$ are isomorphic and the image of $\bar{\rho}$ is assumed to contain $U_1(\F_q)$ (condition $\ref{thc3}$ of Theorem $\ref{main}$). The quotient $N'/V=\op{Gal}(F(\mu_{p^2})/F(\mu_p))\simeq \F_p$. Let $\pi: \op{N}'^{ab}\rightarrow \F_p$ denote the map induced by the mod-$V$ quotient. Being the composite of Galois extensions, $F(\mu_{p^2})$ is Galois over $\Q$. As a result, $\pi$ is $\mathcal{A}$-equivariant. Furthermore, since $F(\mu_{p^2})$ is an abelian extension of $\Q$, the $\mathcal{A}$-action on $N'/V$ is trivial. On the other hand, as an $\mathcal{A}$-module, $N'^{ab}\simeq \bigoplus_{\lambda\in \Delta} \F_q(\sigma_{\lambda})$. It follows from condition $\ref{thc4}$ of Theorem $\ref{main}$ that $\sigma_{\lambda}\neq 1$ for $\lambda\in \Phi$. As a result, 
\[\op{Hom}(N'^{ab}, \F_p)^{\op{G}'}\simeq \bigoplus_{\lambda\in \Delta}\op{Hom}(\F_q(\sigma_{\lambda}), \F_p)^{\op{G}'}=0.\] This is a contradiction which concludes the proof of the first part.
\par 
Set $\mathcal{K}_j$ to be $K_1\dots K_j$ and $\mathcal{K}_0:=K$. Setting $E:=\Q(\mu_{p^2})$, it suffices to show that $\mathcal{K}_j\cap E=\mathcal{K}_{j-1}\cap E$. We begin with the case $j=1$. For $\psi\in H^1(\op{G}_{\Q,\mathcal{J}}, \g^*)$, regard $\operatorname{Gal}(K_{\psi}/K)$ as an $\F_q[\op{G}']$-module, where the Galois action is induced via conjugation. The $\op{G}'$-module $P_1:=\operatorname{Gal}(K_{1}/K)$ is identified with $\psi_1(\op{G}_K)$. Let $Q_1\subseteq P_1$ be the $\op{G}'$-stable subgroup defined by $Q_1:=\op{Gal}(K_1/(K_1\cap E)\cdot K)$. The action of $\op{G}'$ on $P_1/Q_1=\op{Gal}((K_1\cap E)\cdot K/K)$ is trivial. By Lemma $\ref{Pdecomposition}$, the quotient $P_1/Q_1$ decomposes into subgroups 
\[P_1/Q_1=\bigoplus_{\lambda\in \Phi\cup\{1\}} (P_1)_{\bar{\chi}\sigma_{\lambda}}/(Q_1)_{\bar{\chi}\sigma_{\lambda}}.\]
The characters $\sigma_{\lambda}\neq \bar{\chi}^{-1}$ and hence $P_1=Q_1$. We have thus shown that $K_1\cap E= K\cap E$.
\par Let $P_j$ be defined by $P_j:=\operatorname{Gal}(\mathcal{K}_j/\mathcal{K}_{j-1})$. The $\op{G}'$-module $P_j$ is isomorphic to \[ \operatorname{Gal}(K_j/K_j\cap \mathcal{K}_{j-1})\subseteq  \psi_j(\op{G}_K)\subseteq \g^*.\] Let $Q_j$ be the $\op{G}'$-stable subgroup $\op{Gal}(\mathcal{K}_j/(\mathcal{K}_j\cap E)\cdot \mathcal{K}_{j-1})$ and note that the $\op{G}'$ action on $P_j/Q_j$ is trivial. Invoking the same argument as in the case when $j=1$, we have that $P_j=Q_j$ and hence $\mathcal{K}_j\cap E=\mathcal{K}_{j-1}\cap E$. This completes the proof. 
\end{proof}

\begin{Def}

\begin{enumerate}
    \item Let $M_1$ and $M_2$ be $\F_p[\op{G}']$-modules. We say that $M_1$ is unrelated to $M_2$ if for every $\F_p[\op{G}']$-submodule $N$ of $M_1$, 
\[\op{Hom}(N,M_2)^{\op{G}'}=0.\]

\item Let $E$ be a finite extension of $K$ such that $E$ is Galois over $\Q$ and $\op{Gal}(E/K)$ is an $\F_p$-vector space. Let $M$ be an $\F_p[\op{G}']$-module. We say that $E$ is unrelated to $M$ if $\op{Gal}(E/K)$ is $\op{G}'$-unrelated to $M$. Here, the $\op{G}'$-action on $\op{Gal}(E/K)$ is induced via conjugation (let $x\in \op{Gal}(E/K)$ and $g\in \op{G}'$, pick a lift $\tilde{g}$ of $g$, set $g\cdot x:=\tilde{g}x\tilde{g}^{-1}$).
\end{enumerate}

\end{Def}
\begin{Prop}\label{414}
Let $\mathcal{J}\supseteq S$ be a finite set of prime numbers and \[\theta_0,\dots, \theta_t\in H^1(\op{G}_{\Q,\mathcal{J}},\g^*)\] be linearly independent over $\F_q$. Set $K_i:=K_{\theta_i}$ and let $\mathbb{L}_1,\dots, \mathbb{L}_k$ be a (possibly empty) set of Galois extensions of $\Q$. Assume that $\mathbb{L}_i$ contains $K$ and $\op{Gal}(\mathbb{L}_i/K)$ is an $\F_p$-vector space for $i=1,\dots, k$. Suppose that $\mathbb{L}_i$ is unrelated to $\g^*$ for $i=1,\dots, k$. Denote by $\mathcal{L}$ the composite $\mathbb{L}_1\cdots \mathbb{L}_k$. If the set $\{\mathbb{L}_1,\dots, \mathbb{L}_k\}$ is empty, set $\mathcal{L}=K$. The field $K_0$ is not contained in the composite of the fields $K_1\cdots K_t \cdot \mathcal{L}$.
\end{Prop}
\begin{proof}
Let $\mathcal{K}$ denote the composite of the fields $K_1,\dots, K_t$. If $K_0$ is contained in $\mathcal{K}\cdot \mathcal{L}$, then $\theta_0,\dots, \theta_t\in H^1(\op{Gal}(\mathcal{K}\cdot \mathcal{L}/\Q),\g^*)$ and hence
\[h^1(\op{Gal}(\mathcal{K}\cdot \mathcal{L}/\Q),\g^*)\geq t+1.\] Hence it suffices to show that 
\[h^1(\op{Gal}(\mathcal{K}\cdot \mathcal{L}/\Q),\g^*)\leq t.\]First we show that 
\[h^1(\op{Gal}( \mathcal{L}/\Q),\g^*)=0.\] Denote by $\mathcal{L}_i$ the composite of the fields $\mathbb{L}_1\cdots \mathbb{L}_i$ and set $\mathcal{L}_0:=K$. Note that $\op{Gal}(\mathcal{L}_i/\mathcal{L}_{i-1})$ is isomorphic to $\op{Gal}(\mathbb{L}_i/\mathbb{L}_i\cap \mathcal{L}_{i-1})$, which is an $\F_p[\op{G}']$-submodule of $\op{Gal}(\mathbb{L}_i/K)$. Since $\mathbb{L}_i$ is unrelated to $\g^*$,
\[\Hom(\op{Gal}(\mathcal{L}_i/\mathcal{L}_{i-1}),\g^*)^{\op{G}'}=0.\]Hence the inflation map
\[H^1(\op{Gal}(\mathcal{L}_{i-1}/\Q),\g^*)\xrightarrow{\op{inf}} H^1(\op{Gal}(\mathcal{L}_{i}/\Q),\g^*)\]is an isomorphism. We deduce that $H^1(\op{Gal}(\mathcal{L}/\Q),\g^*)$ is isomorphic to $H^1(\op{G}',\g^*)$ and hence, is zero.
\par Let $\mathcal{K}_i$ denote the composite $K_1\cdots K_i$ and $\mathcal{K}_0$ denote $K$. Note that $\op{Gal}(\mathcal{K}_i\cdot \mathcal{L}/\mathcal{K}_{i-1}\cdot\mathcal{L})$ is an $\F_p[\op{G}']$-submodule of $\op{Gal}(K_i/K)$, and hence, of $\g^*$. Lemma $\ref{y1}$ asserts that
\[\dim \Hom(\op{Gal}(\mathcal{K}_i\cdot \mathcal{L}/\mathcal{K}_{i-1}\cdot\mathcal{L}),\g^*)^{\op{G}'}\leq 1.\]Therefore, by inflation-restriction,
\[h^1(\op{Gal}(\mathcal{K}_i\cdot \mathcal{L}/\Q),\g^*)\leq h^1(\op{Gal}(\mathcal{K}_{i-1}\cdot\mathcal{L}/\Q),\g^*) +1.\]Consequently, we deduce that
$h^1(\op{Gal}(\mathcal{K}\cdot \mathcal{L}/\Q),\g^*)\leq t$ and the proof is complete.
\end{proof}
\begin{Lemma}\label{415} Let $\mathcal{J}\supseteq S$ be a finite set of primes.
 \begin{enumerate}
 
     \item\label{415c1} Let $M$ be a nontrivial quotient of $\g^*$ and $\eta\in H^1(\op{G}_{\Q,\mathcal{J}},M)$ be non-zero. Let $K_{\eta}$ be the field extension of $K$ cut out by $\eta$. The field $K_{\eta}$ is unrelated to $\g^*$.
     \item \label{415c2} The field $K(\mu_{p^2})$ is unrelated to $\g^*$.
    \item\label{415c3} Let $f\in H^1(\op{G}_{\Q,\mathcal{J}}, \g)$, then the extension $K_f$ is unrelated to $\g^*$.
    \item \label{415c4} Suppose that we are given a lift $\zeta_2:\op{G}_{\Q, \mathcal{J}}\rightarrow \op{GSp}_{2n}(\text{W}(\F_q)/p^2)$ of $\bar{\rho}$ with similitude character $\kappa\mod{p^2}$. The field extension $K(\zeta_2)$ cut out by the kernel of $\zeta_2$, is unrelated to $\g^*$.
 \end{enumerate}
\end{Lemma}
\begin{proof}
\par For part $\eqref{415c1}$, it suffices to show that $M$ is unrelated to $\g^*$. Since $M$ is a non-trivial quotient of $\g^*$, it follows from Lemma $\ref{mainin}$ that the $\bar{\chi}\sigma_{2L_1}$-eigenspace of $M$ is zero. Let $N\subseteq M$ be a $\op{G}'$-submodule and $f:N\rightarrow \g^*$ be a homomorphism. The $\bar{\chi}\sigma_{2L_1}$-eigenspace of $f(N)$ is zero, hence by Lemma $\ref{mainin}$, the map $f=0$.
\par Since $\Q(\mu_{p^2})$ is an abelian extension of $\Q$, the $G'$-action on $\op{Gal}(K(\mu_{p^2})/K)$ is trivial. On the other hand, $\g^*$ has no trivial $\mathbb{T}$-eigenspace. It follows that $K(\mu_{p^2})$ is unrelated to $\g^*$ and part $\eqref{415c2}$ follows.
\par 
Let $Q$ be a $\op{G}'$-submodule of $\g$, Lemma $\ref{Pdecomposition}$ asserts that $Q=\bigoplus_{\lambda\in \Phi\cup \{1\}} Q_{\sigma_{\lambda}}$. On the other hand, $\g^*=\bigoplus_{\gamma\in \Phi\cup \{1\}} (\g^*)_{\bar{\chi}\sigma_{\lambda}}$. It follows from condition $\eqref{thc4}$ of Theorem $\ref{main}$ that 
\[\op{Hom}(Q_{\sigma_{\lambda}}, (\g^*)_{\bar{\chi}\sigma_{\gamma}})^{\mathbb{T}}=0.\] As a result, $\op{Hom}(Q,\g^*)^{\op{G}'}=0$ and hence part $\eqref{415c3}$ follows.
\par For part \eqref{415c4}, identify $\g$ with the kernel of the mod-$p$ reduction map \[\op{Sp}_{2n}(\text{W}(\F_q)/p^2)\rightarrow \op{Sp}_{2n}(\F_q),\]by identifying $X\in \g$ with $\op{Id}+pX$. Recall that $\kappa=\kappa_0\chi^k$, where $k$ is a positive integer which is divisible by $p(p-1)$. Setting $\kappa_2:=\kappa\mod{p^2}$, we see that the restriction $\kappa_{2\restriction \op{G}_K}$ is trivial. Therefore, $\op{Gal}(K(\zeta_2)/K)$ may be identified with a Galois submodule of $\g$. Here, $g\in \op{Gal}(K(\zeta_2)/K)$ is identified with 
\[\zeta_2(g)\in \op{ker}\left(\op{Sp}_{2n}(\text{W}(\F_q)/p^2)\rightarrow \op{Sp}_{2n}(\F_q)\right)\simeq \g.\] The same reasoning as in the previous case shows that $K(\zeta_2)$ is indeed unrelated to $\g^*$.
\end{proof}
\begin{Lemma}\label{lemma416}
Let $L_1,\dots, L_k$ and $K_1,\dots , K_l$ be Galois extensions of $\Q$ which contain $K$. Assume that:
\begin{itemize}
    \item $\op{Gal}(L_i/K)$ and $\op{Gal}(K_i/K)$ are finite dimensional $\F_p$-vector spaces.
    \item As a $\op{G}'$-module, $\op{Gal}(L_i/K)$ is isomorphic to a subquotient of $\g$ for $i=1,\dots, k$.
   \item As a $\op{G}'$-module, $\op{Gal}(K_i/K)$ is isomorphic to a subquotient of $\g^*$ for $i=1,\dots, l$.
\end{itemize} Then the composite $L_1\cdots L_k$ is linearly disjoint from $K_1,\dots, K_l$.
\end{Lemma}
\begin{proof}
The order of $\mathbb{T}$ is coprime to $p$, hence Maschke's theorem asserts that any finite dimensional $\F_p[\op{G}']$-module $M$ decomposes into a direct sum 
\[M=\oplus_{\tau} M_{\tau}.\]Here, $\tau$ is a character of $\mathbb{T}$ and $M_{\tau}$ is the $\tau$-eigenspace 
\[M_{\tau}:=\{m\in M| g\cdot m=\tau(g) m\}.\]The action of $\op{G}'$ on $\op{Gal}(L_i/K)$ and $\op{Gal}(K_i/K)$ is induced by conjugation. By assumption, $\op{Gal}(L_i/K)$ is isomorphic to a subquotient of $\g$, i.e. there exist $\op{G}'$-submodules $Q_1\subseteq Q_2$ of $\g$ such that $\op{Gal}(L_i/K)\simeq Q_2/Q_1$. By Lemma $\ref{Pdecomposition}$, the module $Q_i$ decomposes into $\mathbb{T}$-eigenspaces
\[Q_i=\bigoplus_{\lambda\in \Phi\cup \{1\}}(Q_i)_{\sigma_{\lambda}}\]for $i=1,2$. Therefore, the quotient $\op{Gal}(L_i/K)$ decomposes into \[\op{Gal}(L_i/K)=\bigoplus_{\lambda\in \Phi\cup \{1\}}(\op{Gal}(L_i/K))_{\sigma_{\lambda}}\] where $(\op{Gal}(L_i/K))_{\sigma_{\lambda}}:=(Q_2)_{\sigma_{\lambda}}/(Q_1)_{\sigma_{\lambda}}$ is the $\sigma_{\lambda}$-eigenspace for the action of $\mathbb{T}$ on $\op{Gal}(L_i/K)$. Likewise, $\op{Gal}(K_i/K)$ decomposes into \[\op{Gal}(K_i/K)=\bigoplus_{\lambda\in \Phi\cup \{1\}}(\op{Gal}(K_i/K))_{\bar{\chi}\sigma_{\lambda}}.\]
\par Let $\mathcal{L}$ be the composite $L_1\cdots L_k$ and $\mathcal{K}$ be the composite $K_1\cdots K_l$. Letting $\mathcal{L}_i$ be the composite $L_1\cdots L_i$, filter $\mathcal{L}$ by
 \[\mathcal{L}\supseteq \mathcal{L}_{k-1}\cdots \supseteq \mathcal{L}_1\supseteq K.\] The Galois group \[\op{Gal}(\mathcal{L}_i/\mathcal{L}_{i-1})\simeq \op{Gal}(L_i/L_i\cap \mathcal{L}_{i-1})\] is a $\op{G}'$-submodule of $\op{Gal}(L_i/K)$. Hence the characters for the action of $\mathbb{T}$ on $\op{Gal}(\mathcal{L}_i/\mathcal{L}_{i-1})$ are each of the form $\sigma_{\lambda}$. Similar reasoning shows that the characters for the action of $\mathbb{T}$ on $\op{Gal}(\mathcal{K}/K)$ are each of the form $\bar{\chi}\sigma_{\lambda}$. Set $E=\mathcal{K}\cap \mathcal{L}$ and $M=\op{Gal}(E/K)$. Being a quotient of $\op{Gal}(\mathcal{L}/K)$, $M$ decomposes into eigenspaces for the action of the torus 
\[M=\bigoplus_{\lambda\in \Phi\cup \{1\}} M_{\sigma_{\lambda}}.\]Since $M$ is a quotient of $\op{Gal}(\mathcal{K}/K)$, \[M=\bigoplus_{\gamma\in \Phi\cup \{1\}} M_{\bar{\chi}\sigma_{\gamma}}.\] It is assumed that the image of $\sigma_{\lambda}$ spans $\F_q$ and that $\sigma_{\lambda}$ is not a $\op{Gal}(\F_q/\F_p)$ twist of $\bar{\chi}\sigma_{\gamma}$. Hence, it follows that
\[\op{Hom}(\F_q(\sigma_{\lambda}),\F_q(\bar{\chi}\sigma_{\gamma}))^{\mathbb{T}}=0.\]Therefore, $\Hom(M,M)^{\op{G}'}=0$ and in particular, the identity map is zero. This implies that $\mathcal{K}\cap \mathcal{L}=K$.
\end{proof}
\section{Deformation conditions at Auxiliary Primes}

We introduce the auxiliary primes $v$ and the liftable deformation problem $\mathcal{C}_v$ at $v$.
\begin{Def} 
A prime number $v$ is a trivial prime if the following splitting conditions are satisfied:
\begin{itemize}
\item  $\operatorname{G}_v\subseteq \ker\bar{\rho}$,

\item $v\equiv 1 \mod{p}$ and $v \not\equiv 1 \mod{p^2}$.\end{itemize} 
\end{Def}
 In other words, a prime number $v$ is trivial if it splits in $\Q(\bar{\rho},\mu_p)$ and does not split in $\Q(\mu_{p^2})$. By Lemma $\ref{22Dec5}$, $\Q(\bar{\rho},\mu_p)$ does not contain $\Q(\mu_{p^2})$. This is a Chebotarev condition, i.e. defined by a finite union of sets that are defined by applying the Chebotarev density theorem. Therefore, the set of trivial primes has positive Dirichlet density, in particular, it is infinite.\par Let $v$ be a trivial prime. The deformations of the trivial representation $\bar{\rho}_{\restriction \operatorname{G}_v}$ are tamely ramified. The Galois group of the maximal pro-p extension of $\Q_{v}$ is generated by a Frobenius $\sigma_v$ and a generator of tame pro-$p$ inertia $\tau_v$. These satisfy the relation 
$\sigma_v\tau_v\sigma_v^{-1}=\tau_l^{v}$. We define the deformation functor $\mathcal{C}_v$. The functor $\mathcal{C}_v$ will be liftable, however, it will not be a deformation condition. Let $\alpha$ be a root which shall be specified later. The root-subgroup $\text{U}_{\alpha}\subset \GSp_{2n}$ is the subgroup generated by the image of the root-subspace $(\op{sp}_{2n})_{\alpha}$ under the exponential map. We let $\ZU$ be the subgroup of $\GSp_{2n}$ consisting of elements which commute with $\text{U}_{\alpha}$.

\begin{Def}\label{ramtriv}\cite[Definition 3.1]{FKP1}
 Let $\mathcal{D}_v^{\alpha}$ consist of the deformation classes of lifts such that some representative $\varrho$ satisfies:
 \begin{enumerate}
     \item $\varrho(\sigma_v)\in \mathcal{T}\cdot \ZU$ and $\varrho(\tau_v)\in \text{U}_{\alpha}$,
     \item under the composite 
     \[\mathcal{T}\cdot \ZU\rightarrow \mathcal{T}/(\mathcal{T}\cap \ZU)\xrightarrow{\alpha} \text{GL}_1\] $\varrho(\sigma_v)$ maps to $v$.
 \end{enumerate}
 \begin{Remark}
 When $n=1$ and $\alpha$ is the positive root of $\operatorname{sl}_2$, the deformation functor $\mathcal{D}_v^{\alpha}$ consists of $\varrho$ such that there exists $x$ and $y$ such that \[\varrho(\sigma_v)=c\mtx{v}{x}{0}{1}\text{ and } \varrho(\tau_v)=\mtx{1}{y}{0}{1}.\]Here $c$ is equal to $(\kappa(\sigma_v)/v)^{\frac{1}{2}}$.
 \end{Remark}
\end{Def}
We shall denote by the kernel of $\alpha$ restricted to $\mathfrak{t}$ by $\mathfrak{t}_{\alpha}$. Since the action of $\operatorname{G}_v$ on $\g$ is trivial, 
\[H^1(\operatorname{G}_v, \g)=\text{Hom}(\operatorname{G}_v, \g).\] Let $\mathcal{P}_v^{\alpha}$ be the subspace of $H^1(\operatorname{G}_v, \g)$ consisting of $\phi$ such that
\[\phi(\sigma_v)\in  \mathfrak{t}_{\alpha}+\text{Cent}((\g)_{\alpha})\]
\[\phi(\tau_v)\in (\g)_{\alpha}.\]
Let $\Phi^{\alpha}$ denote the subset of roots $\beta\in \Phi$ such that $[(\g)_{\alpha}, (\g)_{\beta}]\neq 0$. Recall that $X_{\alpha}$ is a choice of root vector for $\alpha$.

\begin{Def}\label{defconditions}
\begin{enumerate}
    \item Let $v$ be a trivial prime which is unramified mod $p^2$ in our lifting argument. Set $\alpha=2L_1$ and $\mathcal{C}_v=\mathcal{C}_v^{nr}$ consist of deformations with a representative \[\varrho'=(\operatorname{Id}+X_{-\alpha})\varrho (\operatorname{Id}+X_{-\alpha})^{-1}\] where $\varrho$ is a representative for a deformation in $\mathcal{D}_v^{\alpha}$ which satisfies further conditions. In accordance with \cite[Definition 3.5]{FKP1}, we assume that the mod-$p^2$ reduction $\varrho_2:=\varrho\mod{p^2}$ satisfies the following conditions:
    \begin{enumerate}
        \item $\varrho_2$ is unramified, with $\varrho_2(\sigma_v)\in \mathcal{T}(\text{W}(\F_q)/p^2)$,
        \item for all $\beta\in \Phi^{\alpha}$, 
        \[\beta(\varrho_2(\sigma_v))\neq 1\mod{p^2}.\]
    \end{enumerate} Let $\mathcal{S}_{v}^{\alpha}$ consist of $\phi\in H^1(\op{G}_v, \g)$ such that $\phi(\sigma_v)\in \bigoplus_{\beta\in \Phi^{\alpha}} (\g)_{\beta}$ and $\phi(\tau_v)=0$. Let $\mathcal{N}_v$ be specified by \[\mathcal{N}_v=\mathcal{N}_v^{nr}:=(\operatorname{Id}+X_{-\alpha})(\mathcal{P}_v^{\alpha}+\mathcal{S}_v^{\alpha})(\operatorname{Id}+X_{-\alpha})^{-1}.\]
    \item \label{defconditions2}Let $v$ be a trivial prime which will be ramified mod $p^2$ in our lifting argument. Let $\alpha=-2L_1$ and $\mathcal{C}_v=\mathcal{C}_v^{ram}$ consist of deformations in $\mathcal{D}_v^{\alpha}$ with representative $\varrho$ satisfying some additional conditions, which we specify. In accordance with \cite[Definition 3.9]{FKP1}, assume that the mod-$p^2$ reduction $\varrho_2$ satisfies the following conditions:
    \begin{enumerate}
        \item $\varrho_2(\tau_v)\in u_{\alpha}(py)$ where $y\in \text{W}(\F_q)^{\times}$, and $u_{\alpha}: (\g)_{\alpha}\rightarrow \op{GSp}_{2n}$ is the root group homomorphism over $\op{W}(\F_q)$.
        \item For all $\beta\in \Phi^{\alpha}$, 
        \[\beta(\varrho_2(\sigma_v))\neq 1\mod{p^2}.\]
    \end{enumerate}Let $\mathcal{S}_v^{\alpha}$ denote the space of cohomology classes specified in the proof of \cite[Lemma 3.10]{FKP1}. Let $\mathcal{N}_v=\mathcal{N}_v^{ram}$ be defined by \[\mathcal{N}_v:=\mathcal{P}_v^{\alpha}+\mathcal{S}_v^{\alpha}.\] 
\end{enumerate}
\end{Def}
The following gives us a criterion for an element $f\in H^1(\op{G}_v, \g)$ to not be contained in $\mathcal{N}_v^{nr}$. This criterion is used in the proof of Proposition $\ref{lastchebotarev}$.
\begin{Lemma}\label{lemma55}
Let $v$ be a trivial prime and $\mathcal{C}_v=\mathcal{C}_v^{nr}$. Let $f\in \mathcal{N}_v$, express $f(\sigma_v)=\sum_{\lambda\in \Phi} {a_{\lambda}} X_{\lambda}+\sum_{i=1}^n a_i H_i$. Write $X_{-2L_1}=c e_{n+1,1}$ and $X_{2L_1}=d e_{1,n+1}$. We have that $a_{2L_1}= -(cd)^{-1} a_{1}$.
\end{Lemma}
\begin{proof}
Set $g:=(\op{Id} + X_{-2L_1})^{-1} f(\op{Id} + X_{-2L_1})$ and express $g(\sigma_v)=\sum_{\lambda\in \Phi} {b_{\lambda}} X_{\lambda}+\sum_{i=1}^n b_i H_i$. Note that for $\phi\in \mathcal{P}_v^{2L_1}$, \[\phi(\sigma_v)\in  \mathfrak{t}_{2L_1}+\text{Cent}((\g)_{2L_1})\] and hence has zero $H_1$-component. For $\phi\in \mathcal{S}_v^{2L_1}$, we have that
\[\phi(\sigma_v)\in \bigoplus_{\beta\in \Phi^{2L_1}} (\g)_{\beta}.\] We deduce that the $H_1$-component $b_1$ is equal to zero. We show that the $H_1$-component of $g(\sigma_v)$ is equal to $a_1+cd a_{2L_1}$ from the relation $g(\sigma_v)=(\op{Id}+X_{-2L_1})^{-1}f(\sigma_v) (\op{Id}+X_{-2L_1})$. Note that $X_{-2L_1}^2=0$ and thus, $(\op{Id}+X_{-2L_1})^{-1}=(\op{Id}-X_{-2L_1})$. One has that
\[\begin{split}g(\sigma_v)=&(\op{Id}-X_{-2L_1})f(\sigma_v) (\op{Id}+X_{-2L_1})\\=&(\op{Id}-c e_{n+1,1})f(\sigma_v) (\op{Id}+c e_{n+1,1})\\
=&f(\sigma_v)+ c[f(\sigma_v),e_{n+1,1}]-c^2e_{n+1,1}f(\sigma_v) e_{n+1,1}.\end{split}\]
Note that 
\[c^2e_{n+1,1}f(\sigma_v) e_{n+1,1}=a_{2L_1}c^2  e_{n+1,1}X_{2L_1} e_{n+1,1}= a_{2L_1}c^2d  e_{n+1,1}= a_{2L_1}cd X_{-2L_1},\]and thus does not contribute to the $H_1$-component of $g(\sigma_v)$. The contribution of \[c[f(\sigma_v), e_{n+1,1}]=\sum_{\lambda\in \Phi} {ca_{\lambda}} [X_{\lambda},e_{n+1,1}]+\sum_{i=1}^n ca_i [H_i,e_{n+1,1}]\] to the $H_1$-component of $g(\sigma_v)$ is from the term
\[ca_{2L_1} [X_{2L_1},e_{n+1,1}]=ca_{2L_1} [d e_{1,n+1},e_{n+1,1}]=cd a_{2L_1} H_1.\] Thus, we have shown that $b_1=a_1+cd a_{2L_1}$. Since, $b_1=0$, we deduce that $a_1= -cd a_{2L_1}$.
\end{proof}
\begin{Lemma}\cite[Lemma 3.2, 3.6,3.10]{FKP1}
Let $v$ be a trivial prime (for which either $\mathcal{C}_v=\mathcal{C}_v^{ram}$ or $\mathcal{C}_v^{nr}$ is the chosen deformation condition) and $X\in \mathcal{N}_v$,
\begin{enumerate}
    \item $\dim \mathcal{N}_v=\dim\g=h^0(\operatorname{G}_v, \g)$.
    \item Let $m\geq 3$ and $\rho_m\in \mathcal{C}_v(\text{W}(\F_q)/p^{m})$, then
\[(\operatorname{Id}+p^{n-1}X)\rho_m\in \mathcal{C}_v(\text{W}(\F_q)/p^{m}).\]
\item The deformation functor $\mathcal{C}_v$ is liftable.
\end{enumerate}
\end{Lemma}
\par Prior to lifting $\bar{\rho}$ to characteristic zero, we show that $\bar{\rho}$ lifts to $\rho_2$ after increasing the set of ramification from $S$ to $S\cup X_1$. One may choose a continuous lift $\tau$ of $\bar{\rho}$ as depicted
\[ \begin{tikzpicture}[node distance = 2.6 cm, auto]
            \node(G) at (0,0) {$\operatorname{G}_{\Q,S\cup X_1}$};
             \node (A) at (3,0) {$\GSp_{2n}(\F_q)$};
             \node (B) at (3,2){$\GSp_{2n}(\text{W}(\F_q)/p^2)$};
      \draw[->] (G) to node [swap]{$\bar{\rho}$} (A);
       \draw[->] (B) to node{} (A);
      \draw[->] (G) to node {$\tau$} (B);
      \end{tikzpicture}\]such that the composite $\nu\circ \tau=\psi \mod{p^2}$.
      The obstruction class \[\mathcal{O}(\bar{\rho})_{\restriction S\cup X_1}\in H^2(\op{G}_{S\cup X_1}, \g)\] is represented by the $2$-cocycle
      \[(g_1,g_2)\mapsto \tau(g_1 g_2)\tau(g_2)^{-1}\tau(g_1)^{-1}.\]
      
      The residual representation $\bar{\rho}$ lifts to a representation $\rho_2$ ramified only at primes in $S\cup X_1$ if and only if this obstruction is zero. For $v\in S$, the local representation $\bar{\rho}_{\restriction \operatorname{G}_v}$ satisfies $\mathcal{C}_v$ which is a liftable deformation condition (by assumption) and thus lifts to mod $p^2$. The residual representation $\bar{\rho}$ is unramified at each prime $v\in X_1$ and thus it is easy to see that $\bar{\rho}_{\restriction \operatorname{G}_v}$ lifts to mod $p^2$ for $v\in  X_1$. As a consequence, $\mathcal{O}(\bar{\rho})_{\restriction S\cup X_1}$ is contained in $\Sh^2_{S\cup X_1}(\g)$. We will show that a set of finitely many trivial primes $X_1$ can be chosen so that \[\Sh^2_{S\cup X_1}(\g)=0.\]For such a choice of $X_1$, there is a deformation $\rho_2$
      \[ \begin{tikzpicture}[node distance = 2.8 cm, auto]
            \node(G) at (0,0) {$\operatorname{G}_{\Q,S\cup X_1}$};
             \node (A) at (3,0) {$\GSp_{2n}(\F_q)$.};
             \node (B) at (3,2) {$\GSp_{2n}(\text{W}(\F_q)/p^2)$};
      \draw[->] (G) to node [swap]{$\bar{\rho}$} (A);
       \draw[->] (B) to node{} (A);
      \draw[->] (G) to node {$\rho_2$} (B);
      \end{tikzpicture}\]
     
\begin{Prop}\label{Shavanishing}
Let $\mathscr{M}$ denote the finite set of $\operatorname{G}_{\Q}$-modules defined by \[\begin{split}\mathscr{M}:=&\{(\g)/(\g)_k, \mid -2n+1\leq k \leq 2n\}\\&\cup \{(\g)_k^{\perp}, \mid -2n+1\leq k \leq 2n\}.\\
\end{split}\]
There is a finite set $T\supset S$ such that $T\backslash S$ consists of only trivial primes such that for all $M\in \mathscr{M}$,
\begin{equation}\label{equationTminusS}
\ker \{H^1(\op{G}_{\Q,T},M)\rightarrow \bigoplus_{w\in T\backslash S} H^1(\op{G}_w, M)\}=0
\end{equation}
and so in particular,
\begin{equation*}
\Sh_T^1(M)=0.
\end{equation*}
\end{Prop}
\begin{proof}
We show that $T$ can be chosen for which 
\begin{equation*}
\Sh_T^1(\g^*)=0,
\end{equation*}
the argument for any $M\in \mathscr{M}$ is identical. For $0\neq \psi\in H^1(\G, \g^*)$, let $K_{\psi}\supset \Q(\g^*)$ be the field extension cut out by $\psi$. By Lemma $\ref{l4}$, the extension $K_{\psi}$ is not equal to $\Q(\g^*)$. The extension $K(\mu_{p^2})$ is linearly disjoint with $K_{\psi}$ over $K$. By Lemma $\ref{22Dec5}$, $K(\mu_{p^2})$ is not contained in $K$ and $K(\mu_{p^2})\cap K_{\psi}=K$. As a result, there is a nonempty Chebotarev class of primes which split in $K$ and are non-split in $K_{\psi}$ and $K(\mu_{p^2})$. If $v$ is such a prime, it must be a trivial prime since it splits in $K$ and is non-split in $\Q(\mu_{p^2})$. On the other hand, since $v$ is non-split in $K_{\psi}$, deduce that $\psi_{\restriction \op{G}_v}\neq 0$. We may therefore choose a finite set of primes $T$ such that 
\begin{itemize}
\item $T$ is finite,
\item $T\backslash S$ consists of only trivial primes,
\item $\ker \{H^1(\op{G}_{\Q,T},\g^*)\rightarrow \bigoplus_{w\in T\backslash S} H^1(\op{G}_w, \g^*)\}=0$.
\end{itemize}
\end{proof}
The set of trivial primes $X_1$ is taken to be $T\backslash S$.
 
\section{Lifting to mod $p^3$}

\par By Proposition $\ref{Shavanishing}$, there is a finite set of primes $T$ containing $S$ such that $T\backslash S$ consists of trivial primes and $\Sh_{ T}^1(\g^*)=0$. Let $X_1$ be the set of trivial primes $T\backslash S$. At each prime $v\in X_1$, let $\mathcal{C}_v$ be the liftable deformation problem $\mathcal{C}_v^{nr}$. By global duality, $\Sh_{T}^2(\g)=0$ and thus the cohomological obstruction to lifting $\bar{\rho}$ to a representation $\zeta_2$
\begin{equation}\label{zeta2} \begin{tikzpicture}[node distance = 2.2 cm, auto]
            \node(G) at (0,0){$\operatorname{G}_{\Q,T}$};
             \node (A) at (3,0) {$\GSp_{2n}(\F_q)$};
             \node (B) at (3,2) {$\GSp_{2n}(\text{W}(\F_q)/p^2)$};
      \draw[->] (G) to node [swap]{$\bar{\rho}$} (A);
       \draw[->] (B) to node{} (A);
      \draw[->] (G) to node {$\zeta_2$} (B);
      \end{tikzpicture}\end{equation}
      vanishes. Here, $\zeta_2$ is stipulated to have similitude character $\kappa\mod{p^2}$. Let $v\in T$, recall that the set of $W(\F_q)/p^2$ lifts of $\bar{\rho}_{\restriction \op{G}_v}$ is an $H^1(\op{G}_v, \g)$-torsor. Therefore there exists $z_v\in H^1(\operatorname{G}_v, \g)$ such that the twist $(\operatorname{Id}+z_v p) {\zeta_2}_{\restriction \operatorname{G}_v}$ satisfies $\mathcal{C}_v$. Further, for $v\in X_1$, the class $z_v$ may is chosen so that this twist is unramified. We show that there is a set $W$ of at most two trivial primes such that on increasing the set $T$ to $Z=T\cup W$ there exists a global cohomology class $h\in H^1(\op{G}_{\Q, Z}, \g)$ such that
      \begin{itemize}
     
          \item $h_{\restriction \operatorname{G}_v}=z_v$ for $v\in T$,
          \item $(1+ph)\zeta_{2}|_{\op{G}_v}\in \mathcal{C}_v^{ram}$ for $v\in W$.
      \end{itemize} Further, letting $\rho_2$ be the twist $\rho_2=(\operatorname{Id}+ph)\zeta_2$, each local representation ${\rho_2}_{\restriction \operatorname{G}_v}$ satisfies $\mathcal{C}_v$ for $v\in Z$. As a consequence, the obstruction class $\mathcal{O}(\rho_2)$ is in $ \Sh_{Z}^2(\g)$. Since $Z$ contains $T$, the group $\Sh_{Z}^2(\g)$ is zero. As a result, $\rho_2$ must lift to $W(\F)/p^3$. Assume that there is no such class $h$ for a set $W$ such that $\# W\leq 1$. It is shown that there is a pair of trivial primes $v_1,v_2\notin T$ such that $W$ can be chosen to be equal to $\{v_1,v_2\}$. The set of trivial primes $X_2$ is then chosen to be $Z\backslash S$. For $v\in W$, choose $\mathcal{C}_v$ to be equal to $\mathcal{C}_v^{ram}$. In what follows, a Chebotarev class refers to a nonempty collection of primes defined by the application of the Chebotarev density theorem. Note that a Chebotarev class has positive Dirichlet density, and is in particular, infinite.
    \begin{Prop}\label{P1}
Let $T$ be as in Proposition $\ref{Shavanishing}$ and $\psi$ be a nonzero element in $H^1(\operatorname{G}_{\Q,T}, \g^*)$ and let $W\subset H^1(\operatorname{G}_{\Q,T}, \g^*)$ be a subspace not containing $\psi$. Then, there exists a Chebotarev class of trivial primes $v$ such that \[\begin{split} &{\psi}_{\restriction \operatorname{G}_v}\neq 0\\
&{\beta}_{\restriction \operatorname{G}_v}=0 \text{ for all } \beta\in W.
\end{split}\]Moreover we may choose $v$ so that $v$ does not split completely in the $\bar{\chi} \sigma_{2L_1}$-eigenspace of $\operatorname{Gal}(K_{\psi}/K)$ when viewed as a Galois submodule of $\g^*$. 
\end{Prop}
\begin{proof}
Let $\{\psi_1,\dots, \psi_m\}$ be a basis of $W$. Since $\psi$ is not contained in the span of $W$, the classes $\psi,\psi_1,\dots, \psi_m$ are linearly independent. Extend $\psi_1,\dots, \psi_m$ to $\psi_1,\dots, \psi_r$, so that $\psi,\psi_1,\dots, \psi_r$ is a basis of $H^1(\op{G}_{\Q,T}, \g^*)$. Let $\widetilde{W}$ be the span of $\{\psi_1,\dots, \psi_r\}$. It suffices to prove the statement for $\widetilde{W}$ in place of $W$, since $W$ is contained in $\widetilde{W}$. Let $\mathfrak{F}$ denote the composite $K_{\psi_1} \cdots  K_{\psi_r}$. Set $P:=\operatorname{Gal}(K_{\psi}/K)$ and recall that $J_{\psi}\subset K_{\psi}$ is the field fixed by $P_{\bar{\chi} \sigma_{2L_1}}$. Lemma $\ref{l4}$ asserts that $J_{\psi}\neq  K_{\psi}$. We will show that $\mathfrak{F}\cap K_{\psi}\subseteq J_{\psi}$. First, we show how the result follows from this.
\par Set $\mathfrak{L}:=\mathfrak{F}\cdot K_{\psi}=K_{\psi_1}\cdots K_{\psi_r}\cdot K_{\psi}$. We consider the following field diagram,
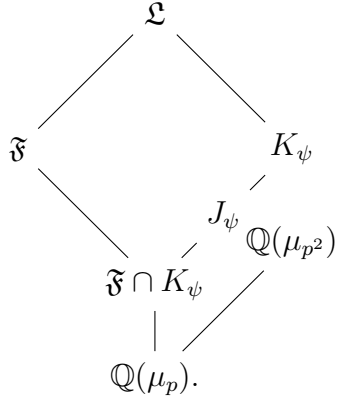
\begin{equation*}
\begin{tikzpicture}[node distance = 1.8cm, auto]
     \node (Qmu) {$\Q(\mu_p).$};
      \node (FK) [above of=Qmu, node distance= 1.25cm] {$\mathfrak{F}\cap K_{\psi}$};
      \node (J) [above of=FK, right of= FK, node distance= 0.9 cm] {$J_{\psi}$};
      \node (Kpsi) [above of=J, right of= J, node distance= 0.9 cm] {$K_{\psi}$};
      \node (F) [above of=FK, left of= FK] {\small $\mathfrak{F}$};
      \node(P) [above of= Qmu, right of= Qmu] {$\Q(\mu_{p^2})$};
      \node(FdotK)[above of= F, right of= F, node distance= 1.8cm] {$\mathfrak{L}$};
      \draw[-] (Qmu) to node {} (FK);
      \draw[-] (Qmu) to node {} (P);
      \draw[-] (FK) to node {} (J);
      \draw[-] (FK) to node {} (F);
      \draw[-] (F) to node {} (FdotK);
      \draw[-] (J) to node {} (Kpsi);
      \draw[-] (Kpsi) to node {} (FdotK);
      \end{tikzpicture}
      \end{equation*}
By Lemma $\ref{22Dec5}$, the intersection $ K\cap \Q(\mu_{p^2})=\Q(\mu_p) $. In fact, Lemma $\ref{22Dec5}$ asserts that $\mathfrak{F}\cap\Q(\mu_{p^2})=\Q(\mu_p)$. Therefore there is a prime $v$ which is
\begin{enumerate}
    \item split in $\op{Gal}(\mathfrak{F}/\Q)$,
    \item nonsplit in $\op{Gal}(\Q(\mu_{p^2})/\Q(\mu_p))$,
    \item nonsplit in $\op{Gal}(K_{\psi}/J_{\psi})$.
\end{enumerate} Since $K=\Q(\bar{\rho},\mu_p)$ is contained in $\mathfrak{F}$, the prime $v$ is a trivial prime. Since $v$ splits in $\op{Gal}(\mathfrak{F}/\Q)$, we have that $\psi_{i\restriction \op{G}_v}=0$ for $i=1,\dots, r$. Since $v$ does not split in $\op{Gal}(K_{\psi}/K)$, we have that $\psi_{\restriction \op{G}_v}\neq 0$.
\par We begin by showing that $K_{\psi}$ is not contained in $\mathfrak{F}$. This is equivalent to the assertion that $\mathfrak{L}$ is not equal to $\mathfrak{F}$. Each of the classes $\psi, \psi_1,\dots, \psi_r$ is in the image of the inflation map 
\[H^1(\operatorname{Gal}(\mathfrak{L}/\Q), \g^*)\xrightarrow{\op{inf}}H^1(\operatorname{G}_{\Q,T}, \g^*),\]and hence the above map is an isomorphism. It follows that
\begin{equation*}
h^1(\operatorname{Gal}(\mathfrak{L}/\Q), \g^*)=h^1(\operatorname{G}_{\Q,T}, \g^*)\geq r+1.\end{equation*}
It suffices to show that
$h^1(\operatorname{Gal}(\mathfrak{F}/\Q), \g^*)\leq r$. We show by induction on $i$ that
\begin{equation*}
h^1(\operatorname{Gal}( K_{\psi_1}\cdots K_{\psi_i}/\Q), \g^*)\leq i.\end{equation*} Lemma $\ref{l3}$ asserts that $H^1(\op{G}',\g^*)=0$ and hence by inflation-restriction,
\begin{equation*}
H^1(\operatorname{Gal}(K_{\psi_1}/\Q), \g^*)\simeq \op{Hom}(P_1, \g^*)^{\op{G}'}.
\end{equation*}
Lemma $\ref{y1}$ asserts that
\[\dim \op{Hom}(P_1,\g^*)^{\op{G}'}\leq 1\]and hence the case $i=1$ follows.

For the induction step, set $\mathfrak{F}_i:=K_{\psi_1}\cdots K_{\psi_{i}}$ and \[P_{i}:=\op{Gal}(\mathfrak{F}_{i}/\mathfrak{F}_{i-1})\simeq \op{Gal}(K_{\psi_{i}}/K_{\psi_{i}}\cap \mathfrak{F}_{i-1}).\] Lemma $\ref{y1}$ asserts that \[\dim \op{Hom}(P_{i},\g^*)^{\op{G}'}\leq 1\]
from which we see from inflation-restriction
\begin{equation*}
h^1(\operatorname{Gal}(\mathfrak{F}_i/\Q),\g^*)
    \leq h^1(\operatorname{Gal}(\mathfrak{F}_{i-1}/\Q),\g^*)+1.
\end{equation*}
 We conclude that $\mathfrak{L}\neq \mathfrak{F}$ and thus we have deduced that $K_{\psi}\cap \mathfrak{F}\neq K_{\psi}$. Set $Q:=\op{Gal}(K_{\psi}/K_{\psi}\cap \mathfrak{F})$, by Lemma $\ref{mainin}$, 
 \begin{equation*}
 Q_{\bar{\chi}\sigma_{2L_1}}\simeq (\g^*)_{\bar{\chi}\sigma_{2L_1}}\simeq P_{\bar{\chi}\sigma_{2L_1}}.
 \end{equation*}
We deduce that $K_{\psi}\cap \mathfrak{F}$ is contained in $J_{\psi}$. This completes the proof.
\end{proof}
\begin{Def}
Let $\mathcal{J}$ be a set of trivial primes that contains the set $S$ and $v\notin \mathcal{J}$ be a trivial prime. Denote by $\Psi_{\mathcal{J}}^k$ and $\Psi_{\mathcal{J},v}^k$ the maps defined by
\begin{equation*}
\Psi_{\mathcal{J}}^k:H^1(\operatorname{G}_{\Q,\mathcal{J}},(\g)_k)\xrightarrow{res_{\mathcal{J}}}\bigoplus_{w\in \mathcal{J}} H^1(\op{G}_w, (\g)_k)
\end{equation*}
and
\begin{equation*}
\Psi_{\mathcal{J},v}^k: H^1(\operatorname{G}_{\Q,\mathcal{J}\cup\{v\}},(\g)_k)\xrightarrow{res_\mathcal{J}}\bigoplus_{w\in \mathcal{J}} H^1(\op{G}_w, (\g)_k).
\end{equation*}
Let $\tau_v$ be a generator of the maximal pro-$p$ quotient of the tame inertia at $v$, denote by
\begin{equation*}
\pi_{\mathcal{J},v}^k: H^1(\operatorname{G}_{\Q,\mathcal{J}\cup\{v\}}, (\g)_k)\rightarrow (\g)_k
\end{equation*}
the evaluation map defined by
\[\pi_{\mathcal{J},v}^k(f):=f(\tau_v).\]
\end{Def}
\begin{Lemma}\label{lemmaDec26}
Let $T$ be a set of primes as in Proposition $\ref{Shavanishing}$ that contains the set $S$ and $k$ an integer. Suppose $v\notin T$ is a trivial prime with the property that for all $\beta\in H^1(\op{G}_{\Q,T},(\g)_k^*)$, the restriction $\beta_{\restriction \operatorname{G}_v}=0$.  The following are exact:
\begin{equation}\label{shortexact1}
0\rightarrow \op{ker}\Psi_{T}^k\xrightarrow{inf} \op{ker}\Psi_{T,v}^k\xrightarrow{\pi_v^k} (\g)_k\rightarrow 0,
\end{equation} 
\begin{equation}\label{shortexact2}
0\rightarrow H^1(\op{G}_{\Q,T},(\g)_k)\xrightarrow{inf} H^1(\op{G}_{\Q,T\cup\{v\}},(\g)_k)\xrightarrow{\pi_v^k} (\g)_k\rightarrow 0.
\end{equation} 
Further, the image of $\Psi_T$ is equal to the image of $\Psi_{T,v}$.
\end{Lemma}

\begin{proof}
Clearly the composite of the maps is zero and $\eqref{shortexact1}$ is exact in the middle. Denote by $\op{res}_v$ the restriction map: \[\op{res}_v:H^1(\op{G}_{\Q, T\cup\{v\}},(\g)_k^*)\rightarrow H^1(\op{G}_v,(\g)_k^*).\]By assumption, $H^1(\op{G}_{\Q,T},(\g)_k^*)$ and $\op{ker}\op{res}_v$ are equal. By the local Euler characteristic formula and local duality, \[h^1(\op{G}_v,(\g)_k)-h^0(\op{G}_v,(\g)_k)\]\[=h^2(\op{G}_v,(\g)_k)=h^0(\op{G}_v,(\g)_k^*)=\dim (\g)_k.\] By Wiles' Formula $\eqref{wilesformula}$, 
\[\begin{split}\dim \op{ker}\Psi_{T,v}^k=&\dim \op{ker}\Psi_{T}^k+\dim \op{ker}\op{res}_v-h^1(\op{G}_{\Q,T},(\g)_k^*)\\+&h^1(\op{G}_v,(\g)_k^*)-h^0(\op{G}_v,(\g)_k^*)\\
=&\dim \op{ker}\Psi_{T}^k+\dim (\g)_k\\
\end{split}\]and the exactness of $\eqref{shortexact1}$ follows. The exactness of $\eqref{shortexact2}$ follows by the same arguments. Therefore, 
\[\begin{split}\dim \op{im} \Psi_{T,v}=&h^1(\op{G}_{\Q,T\cup \{v\}}, (\g)_k)-\dim \op{ker} \Psi_{T,v}\\=&h^1(\op{G}_{\Q,T}, (\g)_k)-\dim \op{ker} \Psi_{T}=\dim \op{im} \Psi_{T}.\\\end{split}\]
\end{proof}
Let $M$ be an $\F_q[\operatorname{G}_{w}]$-module which is a finite dimensional $\F_q$-vector space. The cup product induces the map
\[H^1(\operatorname{G}_w, M)\times H^1(\operatorname{G}_w, M^*)\rightarrow H^2(\operatorname{G}_w, \F_q(\bar{\chi}))\xrightarrow{\sim} \F_q\] taking $f_1\in H^1(\operatorname{G}_w, M)$ and $f_2\in H^1(\operatorname{G}_w, M^*)$ to $\op{inv}_w(f_1\cup f_2)\in \F_q$. Define the non-degenerate pairing 
\[\left(\bigoplus_{w\in T} H^1(\op{G}_w, \g)\right)\times \left(\bigoplus_{w\in T} H^1(\op{G}_w, \g^*)\right) \rightarrow \F_q\]defined by
$a\cup b=\sum_{w\in T} \op{inv}_w (a_w\cup b_w)$. Denote by $\op{Ann}((z_w)_{w\in T})$ the annihilator of the tuple $(z_w)_{w\in T}$. Recall that we assume that $(z_w)_{w\in T}$ does not arise from a global class unramified outside $T$. In particular, the tuple $(z_w)_{w\in T}$ is not zero, and as a result, $\op{Ann}((z_w)_{w\in T})$ is a codimension one subspace of $\bigoplus_{w\in T} H^1(\op{G}_w, \g^*)$. Let $\Psi_T$ and $\Psi_T^*$ denote the restriction maps
\begin{equation*}
\begin{split}&\Psi_{T}:H^1(\operatorname{G}_{\Q,T},\g)\rightarrow\bigoplus_{w\in T} H^1(\op{G}_w, \g),\\&\Psi_{T}^*:H^1(\operatorname{G}_{\Q,T},\g^*)\rightarrow\bigoplus_{w\in T} H^1(\op{G}_w, \g^*).
\end{split}
\end{equation*}
From the exactness of the Poitou-Tate sequence \cite[Theorem 8.6.14]{NW}, it follows that the images of $\Psi_T$ and $\Psi_T^*$ are exact annihilators of one another. Since it is assumed that $(z_w)_{w\in T}$ is not in the image $\Psi_T^*$, it follows that the image of $\Psi_T$ is not contained in $\op{Ann}((z_w)_{w\in T})$. As a result, ${\Psi_T^*}^{-1}(\op{Ann}(z_w)_{w\in T})$ has codimension one in $H^1(\op{G}_{\Q,T},\g^*)$. Set $(\g)_{-2L_1}$ for the $\F_q$ span of the root vector $X_{-2L_1}$.
\begin{Prop}\label{P2}
Let $T$ be as in Proposition $\ref{Shavanishing}$. There exists a Chebotarev class $\mathfrak{l}$ of trivial primes $v$ such that
\begin{enumerate}
\item\label{22Decc1} $\beta_{\restriction \operatorname{G}_v}=0$ for all $\beta \in H^1(\operatorname{G}_{\Q,T}, (\g)_d^*)$ for $d\geq -2n+2$,
\item\label{22Decc2} there exists an $\F_q$ basis $\{\psi, \psi_1,\dots, \psi_r\}$ of $H^1(\operatorname{G}_{\Q,T}, \g^*)$ such that 
\begin{itemize}
\item\label{12}
$\{\psi_1,\dots, \psi_r\}$ is a basis of ${\Psi_T^*}^{-1}(Ann(z_{w})_{w\in T})$ 
\item
$\psi_{\restriction \operatorname{G}_v}\neq 0$ and ${\psi_j}_{\restriction \operatorname{G}_v}=0$ for all $j\geq 1$. 
\end{itemize}
\end{enumerate}
Furthermore, there is, for each $v\in \mathfrak{l}$, an element $h^{(v)}\in H^1(\operatorname{G}_{T\cup\{v\}}, \g)$ such that 
\[h^{(v)}|_{\op{G}_w}=z_w\] for all $w\in T$ and 
\begin{equation}\label{equation64}
h^{(v)}(\tau_v)\in (\g)_{-2L_1}\backslash \{0\}.\end{equation}
\end{Prop}
\begin{proof}
First, we analyze condition $\eqref{22Decc1}$.
Recall that $(\g)_d^*$ is the quotient of $\g^*$ by the Galois stable subspace $(\g)_d^{\perp}$, see Definition $\ref{perpdef}$. Its $\mathbb{T}$-eigenspaces consist of $(\g)_{d,\bar{\chi}\sigma_{\lambda}^{-1}}^*$, where $\lambda$ ranges through $\Phi\cup \{1\}$ with $\op{ht}(\lambda)\geq d$. Condition $\eqref{thc4}$ of Theorem $\ref{main}$ asserts that $\bar{\chi}\sigma_{-\lambda}\neq \sigma_{1}$, i.e., $\sigma_{\lambda}\neq \bar{\chi}$. Therefore, $(\g)_d^*$ contains no trivial eigenspace. Hence, the splitting conditions imposed by $\eqref{22Decc1}$ are independent of the non-splitting condition in $\Q(\mu_{p^2})$ imposed by the fact that trivial primes are not $1\mod{p^2}$. On the other hand, by Proposition $\ref{P1}$, condition $\eqref{22Decc2}$ can be satisfied by a Chebotarev class of trivial primes.
\par Next, we show that conditions $\eqref{22Decc1}$ and $\eqref{22Decc2}$ can be satisfied simultaneously. To show this, note that the condition requiring $\psi_{\restriction{\op{G}_v}}\neq 0$ is a non-splitting condition of $v$ in $\op{Gal}(K_{\psi}/K)$. By Lemma $\ref{l4}$, the $\bar{\chi}\sigma_{2L_1}$-eigenspace for the $\mathbb{T}$-action on $\op{Gal}(K_{\psi}/K)$ is nontrivial. We shall require that $v$ does not split in the $\bar{\chi}\sigma_{2L_1}$-eigenspace of $\op{Gal}(K_{\psi}/K)$. On the other hand, \[(\g^*)_d=\bigoplus_{ \op{ht}(\lambda)\geq d}(\g^*)_{d,\bar{\chi}\sigma_{\lambda}^{-1}},\]does not contain the $\bar{\chi}\sigma_{2L_1}$-eigenspace of $\g^*$. Note that the character $\bar{\chi}\sigma_{2L_1}$ is not twist equivalent to any of the characters occurring in the $\mathbb{T}$-eigenspace decomposition of $(\g^*)_d$. As a result, it follows via an argument identical to that in proof of Lemma $\ref{lemma416}$, that the non-splitting condition of $v$ in $\op{Gal}(K_{\psi}/K)_{\bar{\chi}\sigma_{2L_1}}$ may be simultaneously satisfied along with the rest of the splitting conditions.

\par Let $v$ be a trivial prime which satisfies conditions $\eqref{22Decc1}$ and $\eqref{22Decc2}$ and moreover is non-split in $\op{Gal}(K_{\psi}/K)_{\bar{\chi}\sigma_{2L_1}}$. Let $d\geq -2n+2$, Lemma $\ref{lemmaDec26}$ asserts that the image of
\begin{equation*}
\Psi_T^d:H^1(\operatorname{G}_{\Q,T\cup\{v\}}, (\g)_d)\rightarrow \bigoplus_{w\in T} H^1(\op{G}_w, (\g)_d)
\end{equation*}
is the same as the image of
\begin{equation*}
\Psi_{T,v}^d:H^1(\operatorname{G}_{\Q,T}, (\g)_d)\rightarrow \bigoplus_{w\in T} H^1(\op{G}_w, (\g)_d).
\end{equation*}
For a trivial prime $v$ for which condition $\eqref{22Decc2}$ is satisfied, it follows from an application of Wiles' formula $\eqref{wilesformula}$ that the image of the map 
\begin{equation*}
\Psi_{T,v}:H^1(\operatorname{G}_{\Q,T\cup\{v\}}, \g)\rightarrow \bigoplus_{w\in T} H^1(\op{G}_w, \g)
\end{equation*}
is greater than that of the map
\begin{equation*}
\Psi_T:H^1(\operatorname{G}_{\Q,T}, \g)\rightarrow \bigoplus_{w\in T} H^1(\op{G}_w, \g).
\end{equation*} We next deduce the existence of $h^{(v)}\in H^1(\operatorname{G}_{\Q,T\cup\{v\}}, \g)$ satisfying the specified properties. Since the image of $\Psi_{T,v}$ is greater than the image of $\Psi_T$, there is a class $g$ in $ H^1(\operatorname{G}_{\Q,T\cup\{v\}}, \g)$ such that $\Psi_{T,v}(g)\notin \text{Image}(\Psi_T)$. Let \[W_1:=\text{Image}(\Psi_T)+\F_q\cdot \Psi_{T,v}(g)\]
and 
\[W_2:=\text{Image}(\Psi_T)+\F_q\cdot (z_w)_{w\in T}.\]
The argument in \cite[Proposition 34]{hamblenramakrishna} applies verbatim to imply that $W_1=W_2$ and so we deduce the existence of $h^{(v)}\in H^1(\operatorname{G}_{\Q,T\cup\{v\}}, \g)$ for which \[h^{(v)}_{\restriction \op{G}_w}={z_w}_{\restriction \op{G}_w}\] for all $w\in T$. As we have observed,
\[\text{Image}(\Psi_{T}^{-2n+2})=\text{Image}(\Psi_{T,v}^{-2n+2})\] since $h^{(v)}\notin \text{Image}(\Psi_T)$ it follows that $h^{(v)}(\tau_v)$ is not contained in $(\g)_{-2n+2}$. Invoking Lemma $\ref{lemmaDec26}$, we deduce that on adding a suitable linear combination of elements to $h^{(v)}$ from $\ker \Psi_{T,v}^d$ for $d> -2n+1$, we modify the class $h^{(v)}$ so that \[h^{(v)}(\tau_v)\in (\g)_{-2L_1}\backslash \{0\}\] as required. 
\end{proof}
For $d\in \Z$, the natural inclusion $(\g)_d^{\perp}\hookrightarrow \g^*$ induces a natural map of cohomology groups $H^1(\op{G}_{\Q}, (\g)_d^{\perp})\rightarrow H^1(\op{G}_{\Q}, \g^*)$.
\begin{Lemma}
Let $\mathfrak{l}$ be the Chebotarev class of trivial primes in the Proposition $\ref{P2}$. Let $\{X_{\lambda}^*\}_{\lambda\in \Phi}$ and $\{H_1^*,\dots, H_n^*\}$ be as in $\eqref{XHdual}$. There exists an $\F_q$-independent set \[\{\eta_{\lambda}^{(v)}\mid \lambda \in \Phi\}\cup\{\ \eta_{1}^{(v)}, \dots,\eta_{n}^{(v)}\}\] contained in $H^1(\operatorname{G}_{\Q,T\cup\{v\}}, \g^*)$, satisfying the following properties:
\begin{enumerate}
\item
$\eta_{\lambda}^{(v)}$ is in the image of the natural map \[H^1(\operatorname{G}_{\Q,T\cup\{v\}}, (\g)_{h+1}^{\perp})\rightarrow H^1(\operatorname{G}_{\Q,T\cup\{v\}}, \g^*),\] where $h=\op{ht}(\lambda)$.
\item For $i=1,\dots, n$, the cohomology class $\eta_{i}^{(v)}$ 
is in the image of the natural map \[H^1(\operatorname{G}_{\Q,T\cup\{v\}}, (\g)_{1}^{\perp})\rightarrow H^1(\operatorname{G}_{\Q,T\cup\{v\}}, \g^*).\]
\item 
For $\lambda \in \Phi$, we have that $\eta_{\lambda}^{(v)}(\tau_v)= X_{\lambda}^*$.
\item 
For $i=1,\dots, n$, we have that $\eta_{i}^{(v)}(\tau_v)=H_i^*$.
\item 
The images of the elements $\eta_{\lambda}^{(v)}$ are a basis for the cokernel of the inflation map
\begin{equation*}
H^1(\operatorname{G}_{\Q,T}, \g^*)\rightarrow H^1(\operatorname{G}_{\Q,T\cup\{v\}}, \g^*).\end{equation*}
\end{enumerate}
\end{Lemma}
\begin{proof}
 The dual to $(\g)_k^{\perp}$ is $\g/(\g)_k$. Proposition $\ref{Shavanishing}$ asserts that \[\Sh_T^1((\g)_k^{\perp*})=0\] for all $k\in \Z$. Wiles' formula \eqref{wilesformula} asserts that
\[\begin{split}&h^1(\operatorname{G}_{\Q,T\cup\{v\}},(\g)_k^{\perp})-\dim \Sh_{T\cup\{v\}}^1((\g)_k^{\perp*})\\=& h^1(\operatorname{G}_{\Q,T},(\g)_k^{\perp})-\dim \Sh_{T}^1((\g)_k^{\perp*})\\+&h^1(\op{G}_v, (\g)_k^{\perp})-h^0(\op{G}_v, (\g)_k^{\perp}).\end{split}.\] Also, by Proposition $\ref{Shavanishing}$, we have that \[\Sh_T^1(\g/(\g)_k)=0.\] On applying the local Euler characteristic formula and Tate duality we have that 
\[h^1(\op{G}_v,(\g)_k)^{\perp}-h^0(\op{G}_v,(\g)_k)^{\perp})=h^0(\op{G}_v,(\g)_k^{\perp*})=\dim (\g)_k^{\perp}.\]For the last equality, note that $\bar{\chi}_{\restriction \op{G}_v}=1$ since $v\equiv 1\mod{p}$ and that the action on $(\g)_k)^{\perp}$ is trivial. It follows that \[h^1(\operatorname{G}_{\Q,T\cup\{v\}},(\g)_k^{\perp})=h^1(\operatorname{G}_{\Q,T},(\g)_k^{\perp})+\dim (\g)_k^{\perp}\] and the evaluation map at $\tau_v$
\[H^1(\operatorname{G}_{\Q,T\cup\{v\}}, (\g)_k^{\perp})\rightarrow (\g)_k^{\perp}\] induces a short exact sequence
\[0\rightarrow H^1(\operatorname{G}_{\Q,T}, (\g)_k^{\perp})\rightarrow H^1(\operatorname{G}_{\Q,T\cup\{v\}},(\g)_k^{\perp})\rightarrow (\g)_k^{\perp}\rightarrow 0.\] The assertion of the Lemma follows.
\end{proof}
Let $v$ a trivial prime in the Chebotarev class $\mathfrak{l}$ of Proposition $\ref{P2}$. For $\lambda\in \Phi$, denote by $K_{\lambda}^{(v)}:=K_{\eta_{\lambda}^{(v)}}$ and for $i=1,\dots, n$, set $K_i^{(v)}:=K_{\eta_{i}^{(v)}}$. Let $J_i^{(v)}\subsetneq K_i^{(v)}$ and $J_{\lambda}^{(v)}\subsetneq K_{\lambda}^{(v)}$ denote $J_{\eta_{i}^{(v)}}$ and $J_{\eta_{\lambda}^{(v)}}$ respectively. If $E=K_i^{(v)}$ (resp. $K_{\lambda}^{(v)}$), denote by $J_E$ the sub-extension $J_{i}^{(v)}$ (resp. $J_{\lambda}^{(v)}$). Let $\mathcal{F}^{(v)}$ denote the collection of fields consisting of $K_i^{(v)}$ for $i=1,\dots, n$ and $K_{\lambda}^{(v)}$ for $\lambda\in \Phi$. The Chebotarev class $\mathfrak{l}$ from Proposition $\ref{P2}$ is defined by Chebotarev classes in a collection of fields $\mathcal{F}_{\mathfrak{l}}$. More specifically, $\mathcal{F}_{\mathfrak{l}}$ is the collection of fields:
\begin{itemize}
    \item $K_{\psi}, K_{\psi_1},\dots, K_{\psi_r}$ from Proposition $\ref{P2}$,
    \item $K_{\beta}$ as $\beta$ runs through all cohomology classes $H^1(\op{G}_{\Q,T}, (\g^*)_d)$, where $d\geq -2n+2$,
    \item $K(\zeta_2)$ (with $\zeta_2$ defined at the start of the section),
    \item $K(\mu_{p^2})$.
\end{itemize}For $v\in \mathfrak{l}$ from Proposition $\ref{P2}$, recall that $L_{h^{(v)}}$ is the field extension of $L$ cut out by \[h^{(v)}_{\restriction \op{G}_L}:\op{G}_L\rightarrow \g.\] Associate to a set of trivial primes $A=\{v_1,\dots, v_k\}$ in $\mathfrak{l}$, \[\mathcal{F}_A:=\cup_{i=1}^k \mathcal{F}^{(v_i)}\text{, and } \mathcal{L}_A:=\{L_{h^{(v_1)}},\dots, L_{h^{(v_k)}}\}.\]
\begin{Lemma}\label{eigenspaceeta}
Let $A=\{v_1,\dots, v_k\}\subset \mathfrak{l}$.
\begin{enumerate}
    \item\label{66c1} Let $F_1$ be a field in the collection $\mathcal{F}_A$ and $F_2$ be the composite of all the other fields in $ \mathcal{F}_A\cup\mathcal{L}_A\cup \mathcal{F}_{\mathfrak{l}}$. Then $F_1$ is not contained in $F_2$. Moreover, the intersection $F_1\cap F_2$ is contained in $J_{F_1}$.
    \item\label{66c2} Let $M_1$ be a field in the collection $\mathcal{L}_A$ and $M_2$ denote the composite of all the other fields in $\mathcal{F}_A\cup \mathcal{L}_A\cup \mathcal{F}_{\mathfrak{l}}$. The intersection $M_1\cap M_2=L$.
\end{enumerate}
\end{Lemma}
\begin{proof}
Part $\eqref{66c1}$ is obtained from an application of Proposition $\ref{414}$, as we now explain. In accordance with the statement of Proposition $\ref{414}$, we define a sequence of linearly independent classes \[\theta_0,\dots, \theta_t\in H^1(\op{G}_{\Q,T\cup A},\g^*),\] and a sequence of fields $\mathbb{L}_1, \dots, \mathbb{L}_b$, each of which is unrelated to $\g^*$.
\par Consider the classes $\eta_i^{(v_j)}$ and $\eta_{\lambda}^{(v_j)}$ as $i=1,\dots, n$, $\lambda\in \Phi$ and $j=1,\dots, k$. Enumerate these classes by $\theta_0,\dots, \theta_a$, so that $\theta_0$ is the cohomology class specified in the description of $F_1$, i.e. $F_1=K_{\theta_0}$. The number $a$ is equal to $(nk\# \Phi)-1$ and $\mathcal{F}_A=\{K_{\theta_0}, K_{\theta_1},\dots, K_{\theta_a}\}$. Let $\theta_{a+1},\dots, \theta_{t}$ be a basis of $H^1(\op{G}_{\Q,T},\g^*)$. Recall that for $\lambda \in \Phi$, we have that $\eta_{\lambda}^{(v_j)}(\tau_{v_j})= X_{\lambda}^*$, and for $i=1,\dots, n$, we have that $\eta_{i}^{(v_j)}(\tau_{v_j})=H_i^*$. The classes $\theta_{a+1},\dots, \theta_{t}$ are unramified at each of the primes $v_j\in A$. It is thus, easy to see that $\theta_0,\dots, \theta_{t}$ are linearly independent. Let $\mathbb{L}_1,\dots, \mathbb{L}_l$ be an enumeration for the fields $K_{\beta}$, as $\beta$ runs through all cohomology classes $H^1(\op{G}_{\Q,T}, (\g^*)_d)$ for $d\geq -2n+2$. Let $\mathbb{L}_{l+1}$ be the field $K(\zeta_2)$ and $\mathbb{L}_{l+2}$ the field $K(\mu_{p^2})$. The collection of fields $\mathcal{F}_{\mathfrak{l}}$ consists of $K_{\theta_{i}}$ for $i=a+1,\dots, t$ and $\mathbb{L}_i$ for $i=1,\dots, l+2$. Next, we have to account for the fields in $\mathcal{L}_A$. Let $\mathbb{L}_{l+3},\dots, \mathbb{L}_b$ be an enumeration of the fields $L_{h^{(v_1)}}, \dots, L_{h^{(v_k)}}$. Thus, the collection of fields $\mathcal{L}_A$ is $\{\mathbb{L}_{l+3},\dots, \mathbb{L}_b\}$. In order to apply Proposition $\ref{414}$, it suffices to show that each of the fields $\mathbb{L}_1,\dots, \mathbb{L}_{b}$ is unrelated to $\g^*$. Note that:
\begin{itemize}
    \item by Lemma $\ref{415}$, part $\eqref{415c1}$, each of the fields $\mathbb{L}_{1},\dots, \mathbb{L}_l$ is unrelated to $\g^*$,
    \item by part $\eqref{415c2}$, $\mathbb{L}_{l+2}$ is unrelated to $\g^*$,
    \item by part $\eqref{415c3}$, $\mathbb{L}_{l+3},\dots, \mathbb{L}_{b}$ are unrelated to $\g^*$,
    \item and by part $\eqref{415c4}$, $\mathbb{L}_{l+1}$ is unrelated to $\g^*$.
\end{itemize}

 By Proposition $\ref{414}$, $F_1$ is not contained in $F_2$ and it follows from Lemma $\ref{mainin}$ that $F_1\cap F_2\subseteq J_{F_1}$.
\par Assume without loss of generality that $M_1=L_{h^{(v_1)}}$. Recall that by \eqref{equation64}, we have that \[h^{(v_1)}(\tau_{v_1})\in (\g)_{-2L_1}\backslash\{0\}.\]As a result, $v_1$ is ramified in the $\sigma_{-2L_1}$-eigenspace of $\op{Gal}(M_1/L)$. On the other hand, $v_1$ is unramified in each of the field extensions in $\mathcal{F}_{\mathfrak{l}}$ and $\mathcal{L}_A\backslash \{M_1\}$. Since the classes $\eta_i^{(v_j)}$ and $\eta_{\lambda}^{(v_j)}$ are valued in $\g^*$, there is no $\sigma_{-2L_1}$-eigenspace for the action of $\mathbb{T}$ on $\op{Gal}(K_{\theta_i}/K)$ for $K_{\theta_i}\in \mathcal{F}_A$. As a result, $v_1$ is unramified in the $\sigma_{-2L_1}$-eigenspace of $\op{Gal}(M_2/L)$. Therefore, $M_1\not\subseteq M_2$. Identify $Q:=\operatorname{Gal}(M_1/M_1\cap M_2)$ with a subgroup of $h^{(v_1)}(\op{G}_L)\subseteq \g$. By Lemma $\ref{fullrankLemma}$ it suffices to show that $Q_{-2L_1} \neq 0$. Since $v_1$ is unramified in $M_2$, the image of $\tau_{v_1}$ in $\op{Gal}(M_1/L)$ lies in $\op{Gal}(M_1/M_1\cap M_2)$. Since $h^{(v_1)}(\tau_{v_1})$ is in $(\g)_{-2L_1}\backslash\{0\}$, we deduce that $Q_{-2L_1} \neq 0$. The assertion $\eqref{66c2}$ follows.
\end{proof}
\begin{Lemma}
Let $v\in \mathfrak{l}$ and $h^{(v)}$ be as in Proposition $\ref{P2}$. Then the $\operatorname{Gal}(L_{h^{(v)}}/L)\simeq \g$.
\end{Lemma}\label{hvLemma}
\begin{proof}
Let $Q:=\operatorname{Gal}(L_{h^{(v)}}/L)\subseteq \g$. Since $Q_{-2L_1} \neq 0$, the assertion follows from Lemma $\ref{fullrankLemma}$.
\end{proof}
\begin{Prop}\label{lifttorho3} For a pair $(v_1,v_2)$ of trivial primes in $\mathfrak{l}$ in Proposition $\ref{P2}$ set $h=-h^{(v_1)}+2h^{(v_2)}$ and $\rho_2:=(I+ph)\zeta_2$. There is a pair $(v_1,v_2)$ such that $\rho_{2\restriction \op{G}_w}\in \mathcal{C}_w$ for all $w\in T$ and $\rho_{2\restriction \op{G}_{v_i}}\in \mathcal{C}_{v_i}^{ram}$ for $i=1,2$.
\end{Prop}
\begin{proof}
For $i=1,2,$ we set $\mathcal{C}_{v_i}:=\mathcal{C}_{v_i}^{ram}$. Note that $h_{\restriction \op{G}_w}=z_w$
for all $w\in T$ and hence $\rho_{2\restriction \op{G}_w}\in \mathcal{C}_w$ for all $w\in T$. This is not the case at the primes $v_1$ and $v_2$. We show that one may indeed find a pair $(v_1,v_2)\in \mathfrak{l}\times \mathfrak{l}$ so that $(I+pz_{v_i})\zeta_2\in \mathcal{C}_{v_i}^{ram}$ for $i=1,2$. Consider for $v\in \mathfrak{l}$, the pair of elements $(\zeta_2(\sigma_v),h^{(v)}(\sigma_v))$, and let $A=(A_1,A_2)$ be the pair of matrices which occurs most frequently, that is, with maximal upper density. The choice of $A$ is not necessarily unique. Let $\mathfrak{l}_1=\{v\in \mathfrak{l}\mid \zeta_2(\sigma_v)=A_1,h^{(v)}(\sigma_v)=A_2\}$. Since there are finitely many choices for $A$, the set of primes $\mathfrak{l}_1$ has positive upper-density. Since $h(\tau_{v_i})\in(\g)_{-2L_1}$ and $\zeta_2$ is unramified at $v_i$, we have that \[(\operatorname{Id}+ph(\tau_{v_i}))\zeta_2(\tau_{v_i})=(\operatorname{Id}+ph(\tau_{v_i}))\in \op{U}_{-2L_1}.\] Furthermore, since $h(\tau_{v_i})\neq 0$, the additional condition on $(\operatorname{Id}+ph(\tau_{v_i}))\zeta_2(\tau_{v_i})$ (see Definition $\ref{defconditions}$ part $\eqref{defconditions2}$) is satisfied. Since $\zeta_2(\sigma_v)$ is fixed throughout $\mathfrak{l}_1$, there are (not necessarily unique) matrices $C_i$ such that if $h(\sigma_{v_i})=C_i$, we will have 
$(\operatorname{Id}+ph){\zeta_2}_{\restriction \operatorname{G}_{v_i}} \in \mathcal{C}_{v_i}$ for $i=1,2$. The values $h^{(v_i)}(\sigma_{v_j})$ are represented in the table below:
\begin{center}
\begin{tabular}{c|c|c } 
  & $\sigma_{v_1}$ & $\sigma_{v_2}$ \\ [0.5 ex]
 \hline
 $h^{(v_1)}$ & $A_2$ & $R$ \\
 \hline
 $h^{(v_2)}$ & $E$ & $A_2$. \\ 
\end{tabular}
\end{center}
We need $E=(A_2+C_1)/2$ and $R=2A_2-C_2$. Note that for an arbitrary pair $(v_1,v_2)\in \mathfrak{l}_1\times \mathfrak{l}_1$, this need not be the case. What follows is a recipe for producing a pair $(v_1,v_2)$ such that $E=(A_2+C_1)/2$ and $R=2A_2-C_2$.
\par For $v\in \mathfrak{l}_1$, let $\delta^{(v)}\in H^1(\operatorname{G}_v, \g^*)$ be the cohomology class given by $\delta^{(v)}(\sigma_v)=X_{-2L_1}^*$ and $\delta^{(v)}(\tau_v)=0$. Let $y$
be the element that occurs most frequently among the elements $\operatorname{inv}_v (\delta^{(v)} \cup  h^{(v)})$ among primes $v$ of $\mathfrak{l}_1$. Set \[\mathfrak{l}_2 =\{
v\in  \mathfrak{l}_1 \mid \operatorname{inv}_v (\delta^{(v)} \cup  h^{(v)})= y\},\] $\mathfrak{l}_2$ has positive
upper density. Suppose we first choose $v_1\in \mathfrak{l}_2$. Recall that $h^{(v_1)}(\tau_{v_1})\in (\g)_{-2L_1}$. By Lemma $\ref{fullrankLemma}$, the class $h^{(v_1)}$ has full rank, i.e. $h^{(v_1)}(\op{G}_K)=\g$. In particular, $2A_2-C_2$ is contained in $h^{(v_1)}(\op{G}_K)$. Choosing $v_2$ such that $h^{(v_1)}(\sigma_{v_2})=2A_2-C_2$ is a Chebotarev condition on the splitting of $v_2$ in $L_{h^{(v_1)}}$. We show that $h^{(v_2)}(\sigma_{v_1})$ is determined by how $v_2$ splits in the $\bar{\chi}\sigma_{2L_1}$-eigenspace each of the fields in $\mathcal{F}^{(v_1)}$. Since $h^{(v_2)}$ is unramified at $v_1$, the values $\eta_{\lambda}^{(v_1)}(\tau_{v_1})$ and $h^{(v_2)}(\sigma_{v_1})$ determine $(\eta_{\lambda}^{(v_1)}\cup h^{(v_2)})_{\restriction \op{G}_{v_1}}$. Express $h^{(v_2)}(\sigma_{v_1})=\sum_{\lambda} a_{\lambda} X_{\lambda}+\sum_{i=1}^n  a_{i} H_{i}$. As $\eta_{\lambda}^{(v_1)}(\tau_{v_1})= X_{\lambda}^*$, we see that $\op{inv}_{v_1}(\eta_{\lambda}^{(v_1)}\cup h^{(v_2)})$ determines $a_{\lambda}$. Likewise, $\op{inv}_{v_1}(\eta_{i}^{(v_1)}\cup h^{(v_2)})$ determines $a_{i}$. For $v\in \mathfrak{l}$ and $\lambda\in \Phi$, set $z_{\lambda}^{(v)}$ to be equal to $\op{inv}_v(\eta_{\lambda}^{(v)}\cup h^{(v)})$. The global reciprocity law asserts that
\[\sum_{w\in T\cup \{v_1,v_2\}}  \op{inv}_w(\eta_{\lambda}^{(v_1)}\cup h^{(v_2)})=0,\text{ and }\sum_{w\in T\cup \{v_1\}}  \op{inv}_w(\eta_{\lambda}^{(v_1)}\cup h^{(v_1)})=0.\]Since $h^{(v_2)}_{\restriction \op{G}_w}=z_w=h^{(v_1)}_{\restriction \op{G}_w}$ for $w\in T$, we deduce that
\begin{equation*}
\begin{split}
\op{inv}_{v_1}(\eta_{\lambda}^{(v_1)}\cup h^{(v_2)}) &=-\sum_{w\in T} \op{inv}_{w}(\eta_{\lambda}^{(v_1)}\cup h^{(v_2)})-\op{inv}_{v_2}(\eta_{\lambda}^{(v_1)}\cup h^{(v_2)})\\
&=-\sum_{w\in T} \op{inv}_{w}(\eta_{\lambda}^{(v_1)}\cup h^{(v_1)})-\op{inv}_{v_2}(\eta_{\lambda}^{(v_1)}\cup h^{(v_2)})\\
&=\op{inv}_{v_1}(\eta_{\lambda}^{(v_1)}\cup h^{(v_1)})-\op{inv}_{v_2}(\eta_{\lambda}^{(v_1)}\cup h^{(v_2)})\\
&=z_{\lambda}^{(v_1)}-\op{inv}_{v_2}(\eta_{\lambda}^{(v_1)}\cup h^{(v_2)}).
\end{split}
\end{equation*}
Since $z_{\lambda}^{(v_1)}$ depends on $v_1$ which is fixed, the variance of the right hand side of the equation comes from the term $\op{inv}_{v_2}(\eta_{\lambda}^{(v_1)}\cup h^{(v_2)})$. The specification of $h^{(v_2)}(\sigma_{v_1})$ amounts to the specification of $\op{inv}_{v_1}(\eta_{\lambda}^{(v_1)}\cup h^{(v_2)})$ for $\lambda\in \Phi$ and $\op{inv}_{v_1}(\eta_{i}^{(v_1)}\cup h^{(v_2)})$ for $i=1,\dots,n$. Set $u_{\lambda}$ to be $\eta_{\lambda}^{(v_1)}(\sigma_{v_2})(X_{-2L_1})$ for $\lambda \in \Phi$ and set $u_i$ to be $\eta_{i}^{(v_1)}(\sigma_{v_2})(X_{-2L_1})$ for $i=1,\dots, n$. Since $h^{(v_2)}(\tau_{v_2})$ is a multiple of $X_{-2L_1}$, we see that 
\[\begin{split}&\op{inv}_{v_2}(\eta_{\lambda}^{(v_1)}\cup h^{(v_2)})=\op{inv}_{v_2}(u_{\lambda}\delta^{(v_2)}\cup h^{(v_2)})\\
&\op{inv}_{v_2}(\eta_{i}^{(v_1)}\cup h^{(v_2)})=\op{inv}_{v_2}(u_{i}\delta^{(v_2)}\cup h^{(v_2)}).\\
\end{split}\] Moreover since $h^{(v_2)}(\tau_{v_2})$ is non-zero, we can choose $b\in \F_q$ such that $\op{inv}_{v_2}(b\delta^{(v_2)}\cup h^{(v_2)})$ takes on any desired value. Note that $\op{inv}_{v_2}(\delta^{(v_2)}\cup h^{(v_2)})$ is set to equal $y$ for all $v_2\in \mathfrak{l}_2$, i.e., does not depend on the choice of $v_2\in \mathfrak{l}_2$. As a result, for $v_1\in \mathfrak{l}_2$, there exist values $\{b_{\lambda}\}_{\lambda\in \Phi}$ and $\{b_i\}_{i=1,\dots, n}$ depending only on $v_1$ such that if $u_{\lambda}=b_{\lambda}$ for $\lambda \in \Phi$ and $u_{i}=b_{i}$ for $i=1,\dots, n$, then, \[h^{(v_2)}(\sigma_{v_1})=(A_2+C_1)/2.\] The condition requiring $h^{(v_2)}(\sigma_{v_1})=(A_2+C_1)/2$, is determined by $\eta_{\lambda}^{(v_1)}(\sigma_{v_2})(X_{-2L_1})$ for $\lambda \in \Phi$ and by $\eta_{i}^{(v_1)}(\sigma_{v_2})(X_{-2L_1})$ for $i=1,\dots, n$. Note that $v_2$ is unramified in $K_{\eta_{\lambda}^{(v_1)}}$ and the value of $\eta_{\lambda}^{(v_1)}(\sigma_{v_2})(X_{-2L_1})$ is determined by the projection of $\sigma_{v_2}$ to the $\bar{\chi}\sigma_{2L_1}$-eigenspace of $\op{Gal}(K_{\eta_{\lambda}^{(v_1)}}/K)$, when viewed as a $\mathbb{T}$-module. Recall that $J_{\eta_{\lambda}^{(v_1)}}$ is the subextension $K\subseteq J_{\eta_{\lambda}^{(v_1)}}\subsetneq K_{\eta_{\lambda}^{(v_1)}}$ such that \[\op{Gal}(K_{\eta_{\lambda}^{(v_1)}}/K)_{\bar{\chi}\sigma_{2L_1}}\simeq \op{Gal}(K_{\eta_{\lambda}^{(v_1)}}/J_{\eta_{\lambda}^{(v_1)}}).\] Since $v_2$ is a trivial prime, it is split in $K$. One may indeed insist that $v_2$ is split in $J_{\eta_{\lambda}^{(v_1)}}$ and $\sigma_{v_2}$ takes on the appropriate value in $\op{Gal}(K_{\eta_{\lambda}^{(v_1)}}/J_{\eta_{\lambda}^{(v_1)}})$ so that $u_{\lambda}=b_{\lambda}$. Hence, the condition $u_{\lambda}=b_{\lambda}$ is simply a condition on the $\bar{\chi}\sigma_{2L_1}$-eigenspace of $K_{\eta_{\lambda}^{(v_1)}}/K$. Likewise, the condition $u_{i}=b_{i}$ is a condition on the $\bar{\chi}\sigma_{2L_1}$-eigenspace of $K_{\eta_{i}^{(v_1)}}/K$.  To summarize, the condition requiring $h^{(v_2)}(\sigma_{v_1})=(A_2+C_1)/2$, is equivalent to Chebotarev conditions on the splitting of $v_2$ in the $\bar{\chi}\sigma_{2L_1}$-eigenspaces of the fields in $\mathcal{F}^{(v_1)}$, in the sense made precise in the preceding discussion.
\par Suppose that for the choice of $v_1\in \mathfrak{l}_2$, there is a $v_2\in \mathfrak{l}_2$ for which the required conditions are satisfied:
\begin{enumerate}
    \item the condition on the splitting of $v_2$ in $L_{h^{(v_1)}}$ which amounts to specifying $h^{(v_1)}(\sigma_{v_2})$,
    \item the condition on the splitting of $v_2$ in the fields $\mathcal{F}^{(v_1)}$ which amounts to specifying $h^{(v_2)}(\sigma_{v_1})$.
\end{enumerate}Then we are done. Note that $\mathfrak{l}_2$ is not a Chebotarev condition, it has only been observed that $\mathfrak{l}_2$ has positive upper density. Consider the case when there is no choice of $v_2\in \mathfrak{l}_2$ for which the above conditions are satisfied. Let $\mathfrak{l}_{v_1}$ be the subset of $\mathfrak{l}$ for which $(R, E)\neq (2A_2-C_2, \frac{(A_2 + C_1)}{2})$ for the choice of $v_1$. We have thus assumed that $\mathfrak{l}_2\subseteq \mathfrak{l}_{v_1}$, it follows that the upper density $\delta(\mathfrak{l}_2)$ is less than or equal to the upper density $\delta(\mathfrak{l}_{v_1})$. 
\par Set $\mathcal{E}^{(v_1)}$ to be the composite of the field $L_{h^{(v_1)}}$ with the fields in $\mathcal{F}^{(v_1)}$ and let $\mathfrak{F}_{\mathfrak{l}}$ be the composite of fields in $\mathcal{F}_{\mathfrak{l}}$. We show that there is an element $x\in\op{Gal}(\mathcal{E}^{(v_1)}\cdot \mathfrak{F}_{\mathfrak{l}}/K)$ such that if $v_2$ is trivial prime such that the Frobenius at $v_2$ maps to $x$, then $v_2\in \mathfrak{l}$ and the conditions on $v_2$ are satisfied. Said differently, if $\sigma_{v_2}=x$, then $v_2\in \mathfrak{l}\backslash \mathfrak{l}_{v_1}$. If $F_1$ is any of the fields in $\mathcal{F}^{(v_1)}$ and $F_2$ is the composite of the other fields in $\mathcal{F}^{(v_1)}\cup \mathcal{F}_{\mathfrak{l}}$, Lemma $\ref{eigenspaceeta}$ asserts that $F_1\cap F_2\subseteq J_{F_1}$. Lemma $\ref{eigenspaceeta}$ asserts that $L_{h^{(v_1)}}$ is linearly disjoint over $L$ from the composite of all fields in $\mathcal{F}^{(v_1)}\cup \mathcal{F}_{\mathfrak{l}}$. To construct such an element $x$, enumerate the fields in $\mathcal{F}^{(v_1)}=\{E_1,\dots, E_{k-1}\}$ and set $E_{k}:=F_{h^{(v_1)}}$. Set $E_0:=\mathfrak{F}_{\mathfrak{l}}$ and let $\mathcal{E}_j$ be the composite $E_0\cdots E_j$, note that $\mathcal{E}_{k}=\mathcal{E}^{(v_1)}\cdot \mathfrak{F}_{\mathfrak{l}}$. Consider the filtration
\[\mathcal{E}_{k}\supset \mathcal{E}_{k-1}\supset \dots \supset \mathcal{E}_1\supset \mathcal{E}_0\supset K.\] Let $x_0\in \op{Gal}(\mathcal{E}_0/K)$ be an element defining $\mathfrak{l}$. Note that $\op{Gal}(\mathcal{E}_1/\mathcal{E}_0)\simeq \op{Gal}(E_1/E_1\cap \mathcal{E}_0)$ and the intersection $E_1\cap \mathcal{E}_0$ is contained in $J_{E_1}$. The condition on $E_1/K$ is on the $\bar{\chi}\sigma_{2L_1}$-eigenspace $\op{Gal}(E_1/J_{E_1})$. Hence $x_0$ lifts to a suitable $x_1\in \op{Gal}(\mathcal{E}_1/K)$. Repeating the process, we see that $x_1$ lifts to $x_{k-1}\in \op{Gal}(\mathcal{E}_{k-1}/K)$ such that if $\sigma_{v_2}=x_{k-1}$, then $v_2\in \mathfrak{l}$ and $h^{(v_2)}(\sigma_{v_1})=(A_2+C_1)/2$. Since $E_k\cap \mathcal{E}_{k-1}=K$, it follows that $x_{k-1}$ can be lifted to $x_{k}\in \op{Gal}(\mathcal{E}_{k}/K)$ such that if $\sigma_{v_2}=x$, then all conditions on $v_2$ are satisfied.

As a result, $\delta(\mathfrak{l}\backslash \mathfrak{l}_{v_1})\geq \frac{1}{[\mathcal{E}^{(v_1)}\cdot \mathfrak{F}_{\mathfrak{l}}:K]}$, and hence,
\[\delta(\mathfrak{l}_{v_1})\leq \left(1-\frac{1}{[\mathcal{E}^{(v_1)}\cdot \mathfrak{F}_{\mathfrak{l}}:K]}\right).\]For $F\in \mathcal{F}^{(v_1)}$, the Galois group $\op{Gal}(F/K)$ may be identified with a Galois submodule of $\g^*$. Hence $[F:K]\leq q^{\dim(\g)}$ for $F\in \mathcal{F}^{(v_1)}$ is a uniform bound independent of $v_1$. Similar reasoning shows that $[L^{h^{(v_1)}}:L]\leq q^{\dim (\g)} $. Setting $N:=(\#\Phi +n+1)\cdot \dim \g$,  deduce that \[\delta(\mathfrak{l}_{v_1})\leq 1-q^{-N}[\mathfrak{F}_{\mathfrak{l}}:K]^{-1}.\]
\par Suppose that there is a sequence of $m$ primes $v_{1}^{(1)},\dots, v_{1}^{(m)}\in \mathfrak{l}_2$, such that it is not possible to find a second prime $v_2$ for any of the primes $v_1^{(j)}$. In other words, $\mathfrak{l}_2\subseteq \cap_{j=1}^m \mathfrak{l}_{v_1^{(j)}}$. We show that the density of $\cap_{j=1}^m \mathfrak{l}_{v_1^{(j)}}$ approaches zero as $m$ approaches infinity. Since the upper density of $\mathfrak{l}_2$ is positive, we will eventually find a pair $(v_1,v_2)$. For convenience of notation, set $w_j:=v_1^{(j)}$ and set $A=\{w_1,\dots, w_m\}$. Fix $1\leq j\leq m$ and enumerate the fields $\mathcal{F}^{(w_j)}=\{E_1,\dots, E_{k-1}\}$ and set $E_k=F_{h^{(w_j)}}$. Denote by $\mathfrak{E}_j:=\mathfrak{F}_{\mathfrak{l}}\cdot \mathcal{E}^{(w_1)}\cdots\mathcal{E}^{(w_j)}$ and let $C_j$ be the subset of $\op{Gal}(\mathfrak{E}_j/K)$ defining the set $\cap_{i=1}^j \mathfrak{l}_{w_i}$. This means that $v_2\in \cap_{i=1}^j \mathfrak{l}_{w_i}$ if and only if $\sigma_{v_2}\in C_j$. We show that any element $w\in \op{Gal}(\mathfrak{E}_{j-1}/K)$ lifts to an element  $\tilde{w}\in\op{Gal}(\mathfrak{E}_{j}/K)$ which is not in $C_j$. This is shown by filtering $\mathfrak{E}_j/\mathfrak{E}_{j-1}$ by 
\[\mathfrak{E}_{j}=\mathcal{E}_k\supset \mathcal{E}_{k-1}\supset \cdots \mathcal{E}_1\supset \mathcal{E}_0=\mathfrak{E}_{j-1},\]
where $\mathcal{E}_l:=\mathfrak{E}_{j-1}E_1\cdots E_l$. The argument is identical to that provided before.
\par
As a result, \[\# C_j\leq ([\mathfrak{E}_j:\mathfrak{E}_{j-1}] -1)\# C_{j-1}.\] Therefore,
\[\begin{split}\delta(\cap_{i=1}^j \mathfrak{l}_{w_i})=\frac{\# C_j}{[\mathfrak{E}_j:K]}\leq & \left(1-\frac{1}{[\mathfrak{E}_j:\mathfrak{E}_{j-1}]}\right)\frac{\# C_{j-1}}{[\mathfrak{E}_{j-1}:K]},\\
\leq & \left(1-\frac{1}{[\mathcal{E}^{(w_j)}:K]}\right)\frac{\# C_{j-1}}{[\mathfrak{E}_{j-1}:K]},\\
\leq & (1-q^{-N}) \delta(\cap_{i=1}^{j-1} \mathfrak{l}_{w_i}).
\end{split}\]Therefore, $\delta(\cap_{i=1}^m \mathfrak{l}_{w_i})\leq (1-q^{-N})^{m-1} (1-q^{-N}[\mathfrak{F}_{\mathfrak{l}}:K]^{-1})$. Since $\mathfrak{l}_2$ has positive upper density there is a large value of $m$ such that $\mathfrak{l}_2$ is not contained in $\cap_{i=1}^m \mathfrak{l}_{w_i}$. This shows that a pair $(v_1,v_2)$ satisfying the required conditions does exist.
\end{proof}

\begin{Prop}\label{bigimageprop}
Let $\rho_2$ be as in Proposition $\ref{lifttorho3}$. The image of $\rho_2$ is the principal congruence subgroup of $\op{GSp}_{2n}(\text{W}(\F_q)/p^2)$ of similitude character $1$.
\end{Prop}
\begin{proof}
Recall that the similitude character $\kappa$ is prescribed to equal $\kappa_0\chi^k$ where $\kappa_0$ is the Teichm\"uller lift of $\bar{\kappa}$ and $k$ is divisible by $p(p-1)$. Therefore, we have that $\kappa\equiv \kappa_0\mod{p^2}$, and as a result, elements in the principal congruence subgroup in the image of $\rho_2$ necessarily have similitude character $1$. Therefore, $\rho_2(\op{G}_L)$ may be identified with a subspace of $\g$. In greater detail, $\rho_2(g)$ is identified with $\frac{1}{p}(\rho_2(g)-\op{Id})$, for $g\in \op{G}_L=\op{ker}\bar{\rho}$. It may be checked that $\rho_2(\op{G}_L)$ is a $\op{G}$-submodule of $\g$ and that the natural $\op{G}$-action on $\op{Gal}(\Q(\rho_2)/L)$ (induced by conjugation) coincides with the $\op{G}$-action on $\rho_2(\op{G}_L)$ viewed as a submodule of $\g$. Recall that $\rho_2=(\op{Id}+ph)\zeta_2$, where $h$ is the cohomology class given by $-h^{(v_1)}+2h^{(v_2)}$. Since $\bar{\rho}$ is unramified at $v_1$, we have that $\tau_{v_1}\in \op{G}_L$. The cohomology class $h^{(v_2)}$ is unramified at $v_1$, as is $\zeta_2$. Therefore, we have that
\[\rho_2(\tau_{v_1})=(\op{Id}+ph(\tau_{v_1}))\zeta_2(\tau_{v_1})=(\op{Id}-ph^{(v_1)}(\tau_{v_1})).\]
Recall that by \eqref{equation64}, we have that \[h^{(v_1)}(\tau_{v_1})\in (\g)_{-2L_1}\backslash\{0\}.\]Therefore, $\rho_2(\op{G}_L)$ is identified with a Galois-submodule of $\g$ which contains an element with non-zero $-2L_1$-component. From Lemma $\ref{fullrankLemma}$, it is deduced that this module must be all of $\g$. This completes the proof.
\end{proof}
\section{Annihiliating the dual-Selmer Group}
Let $\rho_3:\operatorname{G}_{\Q,T\cup \{v_1,v_2\}}\rightarrow \GSp_{2n}(\text{W}(\F_q)/p^3)$ be the lift of $\bar{\rho}$ obtained from the application of Propositions $\ref{lifttorho3}$ and $\ref{bigimageprop}$. Recall that the Galois group $\op{Gal}(K(\rho_2)/K)$ is identified with $\g$. As a result, once it is shown that $\rho_3$ lifts to a characteristic zero representation $\rho$, it shall follow that $\rho$ is irreducible. In showing that $\rho_3$ can be lifted to characteristic zero, we enlarge the set of primes $Z=T\cup \{v_1,v_2\}$ to a finite set of primes $Y$ such that $X:=Y\backslash S$ consists only of trivial primes. For $i=1,2$ set $\mathcal{C}_{v_i}=\mathcal{C}_{v_i}^{ram}$ and for primes $v\in X\backslash \{v_1,v_2\}$, set $ \mathcal{C}_v=\mathcal{C}_v^{nr}$. We show that the dual-Selmer group $H^1_{\mathcal{N}^{\perp}}(\op{G}_{\Q,Y},\g^*)$ vanishes for a suitably chosen set of primes $Y$. For convenience of notation, denote by $\mathscr{W}$ the Galois submodule $(\g)_{-2n+2}$ of $\g$ spanned by root spaces $(\g)_{\beta}$ for $\beta\neq -2L_1$. 
\begin{Prop}\label{lastchebotarev}
Let $\rho_3:\operatorname{G}_{\Q,T\cup \{v_1,v_2\}}\rightarrow \GSp_{2n}(\text{W}(\F_q)/p^3)$ be the lift of $\bar{\rho}$ obtained from the application of Propositions $\ref{lifttorho3}$ and $\ref{bigimageprop}$. Let $Y$ be a finite set of primes which contains $Z=T\cup \{v_1,v_2\}$ such that $Y\backslash S$ consists of trivial primes. Suppose $f\in H^1_{\mathcal{N}}(\op{G}_{\Q,Y}, \g)$ and $\psi\in H^1_{\mathcal{N}^{\perp}}(\op{G}_{\Q,Y}, \g^*)$ are nonzero classes. Then there exists a prime $v\notin Y$ such that
\begin{enumerate}
\item \label{71one} $v$ is a trivial prime,
\item \label{71two}$\rho_{3\restriction \op{G}_v}$ satisfies $\mathcal{C}_v=\mathcal{C}_v^{nr}$, 
\item \label{71three} $f$ does not satisfy $\mathcal{N}_v=\mathcal{N}_v^{nr}$,
\item\label{71four}$\beta_{\restriction \operatorname{G}_v}=0$ for all $\beta\in H^1(\op{G}_{\Q,Y}, \mathscr{W}^*)$,
\item \label{71five} $\psi_{\restriction \operatorname{G}_v}\neq 0$ and one can extend $\{\psi\}$ to a basis $\psi_1=\psi,\psi_2,\dots, \psi_k$ of $H^1(\op{G}_{\Q,Y}, \g^*)$ such that $\psi_{i \restriction \operatorname{G}_v}=0$ for $i>1$.
\end{enumerate}
\end{Prop}
\begin{proof}
Each condition is a union of Chebotarev conditions on a number of finite extensions $J$ of $K$. Each of the extensions $J$ are Galois over $\Q$ with $\op{Gal}(J/K)$ an $\F_p$-vector space. Let $g\in \op{G}'$ and $x\in \op{Gal}(J/K)$, define, $g\cdot x:=\tilde{g} x \tilde{g}^{-1}$ where $\tilde{g}$ is a lift of $g$ to $\op{Gal}(J/\Q)$. This gives $\op{Gal}(J/K)$ the structure of an $\F_p[\op{G}']$-module. For each condition, we list the choices for $J$ below as well as characters for the $\mathbb{T}$-action on $\op{Gal}(J/K)$:
\begin{center}
\begin{tabular}{c|c|c } 
 \text{Condition} & $J$ & $\text{Eigenspaces of } \operatorname{Gal}(J/K)$ \\ [1 ex]
 \hline
 $(1)$ & $K(\mu_{p^2})$ & $1$ \\
 \hline
 $(2)$ & $K(\rho_2)$ & $1,\{\sigma_{\lambda}\}_{\lambda\in \Phi}$ \\
  \hline
 $(3)$ & $K_f$ & $1,\{\sigma_{\lambda}\}_{\lambda\in \Phi}$ \\
  \hline
 $(4)$ & $K_{\beta}$ \text{ for } $\beta \in H^1(\op{G}_{\Q,Y},\mathscr{W}^*)$& $\bar{\chi},\{\bar{\chi}\sigma_{\lambda}^{-1}|\lambda\neq -2L_1\}$\\
 \hline
 $(5)$ & $K_{\psi_i} $ & $\bar{\chi},\{\bar{\chi}\sigma_{\lambda}^{-1}\}$.\\
\end{tabular}
\end{center}
We show that these conditions may be simultaneously satisfied. First, we show that each of the conditions is a nonempty Chebotarev condition (or a union of finitely many Chebotarev conditions). It is clear that condition $\eqref{71one}$ and $\eqref{71two}$ are nonempty Chebotarev conditions. Lemma $\ref{lemma55}$ gives a criterion for the element $f$ to not be in the space $\mathcal{N}_v$. In accordance with Lemma $\ref{lemma55}$, write $X_{-2L_1}=c e_{n+1,1}$ and $X_{2L_1}=d e_{1,n+1}$. Since $f$ is non-zero, $f(\op{G}_L)$ is a non-zero Galois-stable submodule of $\g$. Hence, by Lemma $\ref{mainin}$, contains $(\g)_{\sigma_{2L_1}}$. Therefore, the image of $f_{\restriction \op{G}_L}$ contains an element 
\[f(g)=\sum_{\lambda\in \Phi} a_{\lambda} X_{\lambda} +\sum_{i=1}^n a_i H_i\]such that $a_{2L_1}\neq -(cd)^{-1} a_1$. As a result, condition $\eqref{71three}$ is a union of finitely many nonempty Chebotarev conditions. Condition $\eqref{71four}$ requires that the prime splits in the composite of the fields $K_{\beta}$. That condition $\eqref{71five}$ is a nonempty Chebotarev condition follows from Proposition $\ref{P2}$.

Next we examine the independence of these conditions. It follows from Lemma $\ref{lemma416}$ that the composite of the fields defining the first three conditions is linearly disjoint over $K$ from the composite of the fields defining the last two conditions. As a result, the conditions may be treated separately from the last two. It follows from Proposition $\ref{P2}$ that the conditions $\eqref{71four}$ and $\eqref{71five}$ are compatible with each other. Therefore, it remains to show that $\eqref{71one}$,$\eqref{71two}$ and $\eqref{71three}$ may be simultaneously satisfied. We begin with the independence of $\eqref{71one}$ and $\eqref{71two}$. Proposition $\ref{bigimageprop}$ asserts that $\operatorname{Gal}(K(\rho_2)/K)= \g$. Suppose that $Q$ is a proper $\op{G}'$-stable subgroup of $\g$. Lemma $\ref{Pdecomposition}$ asserts that $Q$ decomposes into $\mathbb{T}$-eigenspaces $Q=\bigoplus_{\lambda\in \Phi\cup \{1\}} Q_{\sigma_{\lambda}}$ and Lemma $\ref{fullrankLemma}$ asserts that the eigenspace $Q_{-2L_1}:=Q_{\sigma_{-2L_1}}$ must be trivial. Hence the quotient $\g/Q$ must have a non-zero $\sigma_{-2L_1}$-eigenspace. It follows that there is no proper Galois stable subgroup $Q$ of $\g$ such that $\g/Q$ is has trivial Galois action. Since $\op{G}'$ acts trivially
on $\operatorname{Gal}(K(\mu_{p^2} )/K)$ it follows that $K(\rho_2) \cap K(\mu_{p^2}) = K$. Thus conditions $\eqref{71one}$ and $\eqref{71two}$ are independent.
\par We show that the first three conditions may be simultaneously satisfied by considering the cases $K(\rho_2)\supseteq K_f$ and $K(\rho_2)\not \supseteq K_f$ separately. First consider the case when $K(\rho_2)\supseteq K_f$. Let $r:=\dim_{\F_p} f(\op{G}_K)$. Since $\op{Gal}(K(\rho_2)/K)\simeq \g$, if $r< \dim_{\F_p} \g$ the containment $K(\rho_2)\supset K_f$ is proper. Since $f$ is non-zero, Lemma $\ref{l4}$ asserts that $K_f\neq K$. Let $Q\subset \op{Gal}(K(\rho_2)/K)$ be the proper subgroup such that $\op{Gal}(K(\rho_2)/K)/Q\simeq \op{Gal}(K_f/K)$. Lemma $\ref{Pdecomposition}$ asserts that $Q$ decomposes into $\mathbb{T}$-eigenspaces $Q=\bigoplus_{\lambda\in \Phi\cup \{1\}} Q_{\sigma_{\lambda}}$ and Lemma $\ref{fullrankLemma}$ asserts that the eigenspace $Q_{-2L_1}:=Q_{\sigma_{-2L_1}}$ must be trivial. Hence the quotient $\op{Gal}(K_f/K)$ must have a non-zero $\sigma_{-2L_1}$-eigenspace. Identify $\op{Gal}(K_f/K)$ with $f(\op{G}_K)\subset \g$. Since $r<\dim_{\F_p} \g$, Lemma $\ref{fullrankLemma}$ asserts that $f(\op{G}_K)_{-2L_1}=0$, a contradiction. Hence, $K(\rho_2)\supseteq K_f$ forces equality $K(\rho_2)= K_f$. Let \[\alpha_1:=f_{\restriction \op{G}_K}:\op{Gal}(K_f/K)\xrightarrow{\sim} \g\]
and 
\[\alpha_2:=\rho_{2\restriction \op{G}_K}:\op{Gal}(K_f/K)\xrightarrow{\sim} \g. \] The composite $\alpha_1\alpha_2^{-1}$ is a $\op{G}'$-automorphism of $\g$. It follows from Corollary $\ref{Coradd}$ that $\alpha_1\alpha_2^{-1}$ is a scalar $a\in \F_q^{\times}$ and hence $\alpha_1=a\alpha_2$. Let $v$ satisfy $\eqref{71one}$, $\eqref{71two}$, $\eqref{71four}$ and $\eqref{71five}$ such that 
\[(\operatorname{Id}+X_{-2L_1})^{-1}\rho_2(\sigma_v)(\operatorname{Id}+X_{-2L_1})\in \mathcal{T}\]has non-trivial $H_1$ component. Since $v$ is a trivial prime, $\sigma_v$ lies in $\op{G}_K$. Identifying $\op{ker}\{\GSp(\text{W}(\F_q)/p^2)\rightarrow \GSp(\F_q)\}$ with $\g$, we view $\rho_2(\sigma_{v})$ as an element in $\g$. Since $f(\sigma_v)=a\rho_2(\sigma_v)$, we see that $(\operatorname{Id}+X_{-2L_1})^{-1}f(\sigma_v)(\operatorname{Id}+X_{-2L_1})$ has non-zero $H_1$ component and hence is not contained in $\mathfrak{t}_{2L_1}+\op{Cent}((\g)_{2L_1})$. As a result, $f$ is not in $(\operatorname{Id}+X_{-2L_1})\mathcal{P}_v^{2L_1}(\operatorname{Id}+X_{-2L_1})^{-1}$. It is easy to see that $f$ is not in $\mathcal{N}_v$ and hence $\eqref{71three}$ is also satisfied.

\par We consider the case when $K_f\not\subseteq K(\rho_2)$. Since \[[K(\rho_2):K]=\# \g \geq [K_f:K],\]$K(\rho_2)$ is not contained in $K_f$. It follows that $K(\rho_2)\supsetneq K(\rho_2)\cap K_f$ and $K_f\supsetneq K(\rho_2)\cap K_f$ and thus by Lemma $\ref{mainin}$, the images of \[\op{Gal}(K(\rho_2)/K(\rho_2)\cap K_f)\hookrightarrow  \g\text{ and  }\op{Gal}(K_f/K(\rho_2)\cap K_f)\hookrightarrow  \g\] contain $(\g)_{\sigma_{2L_1}}$. If $K_f\subseteq K(\rho_2,\mu_{p^2})$, then it follows that
\[\dim_{\F_p} (\op{Gal}(K(\rho_2,\mu_{p^2})/K)_{\sigma_{2L_1}})\geq 2\dim_{\F_p} (\g)_{\sigma_{2L_1}}=2[\F_q:\F_p].\]However, $\op{Gal}(K(\rho_2,\mu_{p^2})/K)_{\sigma_{2L_1}}$ may be identified with \[\op{Gal}(K(\rho_2)/K)_{\sigma_{2L_1}}\simeq (\g)_{\sigma_{2L_1}}\] since $K( \mu_{p^2})$ contributes to the trivial eigenspace.
Hence, $K_f \not \subseteq K(\rho_2, \mu_{p^2} )$. Let $v$ be a prime satisfying conditions $\eqref{71one}$, $\eqref{71two}$, $\eqref{71four}$ and $\eqref{71five}$. Lemma $\ref{mainin}$ asserts that the image of \[\op{Gal}(K_f/K_f\cap K(\rho_2,\mu_{p^2}))\hookrightarrow  \g\] contains $(\g)_{\sigma_{2L_1}}$ and thus, we have the freedom to stipulate that the $X_{2L_1}$-component of $f(\sigma_v)$ be anything we like. Lemma $\ref{lemma55}$ asserts that if $f\in \mathcal{N}_v$, an explicit relationship must be satisfied between the $X_{2L_1}$-component and the $H_1$-component of $f(\sigma_v)$. It follows that we may alter the $X_{2L_1}$-component of $f(\sigma_v)$ so that $f\notin \mathcal{N}_v$. Therefore all conditions may be satisfied and the proof is complete.
\end{proof}
\begin{Prop}
There is a finite set $Y\supseteq Z$ such that $Y\backslash S$ consists of trivial primes and $H^1_{\mathcal{N^{\perp}}}(\op{G}_{\Q,Y}, \g^*)=0$.
\end{Prop}
\begin{proof}
\par Let $Y$ be a finite set of primes containing $Z$ such that $Y\backslash S$ consists of trivial primes. If $H^1_{\mathcal{N}}(\op{G}_{\Q,Y}, \g)\neq 0$, we exhibit a trivial prime $v$ not contained in $Y$ such that 
\[h^1_{\mathcal{N}}(\op{G}_{\Q,Y\cup \{v\}}, \g)<h^1_{\mathcal{N}}(\op{G}_{\Q,Y}, \g).\] Therefore, a finite set of primes $Y$ may be chosen so that the Selmer group $ H^1_{\mathcal{N}}(\op{G}_{\Q,Y}, \g)$ is equal to zero. Since 
\[h^1_{\mathcal{N}}(\op{G}_{\Q,Y}, \g)=h^1_{\mathcal{N}^{\perp}}(\op{G}_{\Q,Y}, \g^*),\] the dual Selmer group does also vanish. \par Let $v\notin Y$ be trivial prime which satisfies the conditions of Proposition $\ref{lastchebotarev}$. Let $\mathcal{M}$ be the Selmer condition 
\[\mathcal{M}_w:=\begin{cases}\mathcal{N}_w\text{ if }w\in Y\\
H^1(\operatorname{G}_v, \g)\text{ if }w=v\\
H^1_{nr}(\op{G}_w, \g)\text{ if }w\notin Y\cup\{v\}.
\end{cases}\]Let $\psi$ be the non-zero class in $H^1_{\mathcal{N}^{\perp}}(\op{G}_{\Q,Y}, \g^*)$ as in Proposition $\ref{lastchebotarev}$. Note that $H^1_{\mathcal{N}^{\perp}}(\op{G}_{\Q,Y}, \g^*)$ contains $H^1_{\mathcal{M}^{\perp}}(\op{G}_{\Q,Y\cup \{v\}}, \g^*)$ and since $\psi_{\restriction \op{G}_v}\neq 0$, the element $\psi$ is not contained in $ H^1_{\mathcal{M}^{\perp}}(\op{G}_{\Q,Y\cup \{v\}}, \g^*)$. In particular, we have that
\begin{equation}\label{dimgreaterone}
h^1_{\mathcal{N}^{\perp}}(\op{G}_{\Q,Y}, \g^*)>h^1_{\mathcal{M}^{\perp}}(\op{G}_{\Q,Y\cup \{v\}}, \g^*).
\end{equation}
\par Consider the restriction  maps
\begin{equation*}
\Phi_1:H^1(\op{G}_{\Q,Y}, \mathscr{W})\rightarrow \bigoplus_{w\in Y}H^1(\op{G}_w, \mathscr{W})
\end{equation*}
and
\begin{equation*}
\Phi_2:H^1(\op{G}_{\Q,Y\cup\{v\}}, \mathscr{W})\rightarrow \bigoplus_{w\in Y} H^1(\op{G}_w, \mathscr{W}).
\end{equation*} We show that the maps $\Phi_1$ and $\Phi_2$ have the same image. By the Poitou-Tate sequence,
\[0\rightarrow \op{image}(\Phi_1)\rightarrow \bigoplus_{w\in Y} H^1(\op{G}_w, \mathscr{W})\rightarrow H^1(\op{G}_{\Q,Y}, \mathscr{W}^*)^{\vee},\] i.e.,
the image of $\Phi_1$ is the exact annihiliator of the image of the restriction map
\[H^1(\op{G}_{\Q,Y}, \mathscr{W}^*)\rightarrow \bigoplus_{w\in Y} H^1(\op{G}_w, \mathscr{W}^*).\]
Let $\mathfrak{M}$ be the Selmer condition 
\[\mathfrak{M}_w:=\begin{cases}0\text{ if }w\in Y\\
H^1(\operatorname{G}_v, \mathscr{W})\text{ if }w=v\\
H^1_{nr}(\op{G}_w, \mathscr{W})\text{ if }w\notin Y\cup\{v\},
\end{cases}\] with dual Selmer condition 
\[\mathfrak{M}_w^{\perp}=\begin{cases}H^1(\operatorname{G}_v, \mathscr{W}^*)\text{ if }w\in Y\\
0\text{ if }w=v\\
H^1_{nr}(\op{G}_w, \mathscr{W}^*)\text{ if }w\notin Y\cup\{v\}.
\end{cases}\]
By the Poitou-Tate sequence, the image of $\Phi_2$ is the exact annihilator of the restriction map 
\[H^1_{\mathfrak{M}^{\perp}}(\op{G}_{\Q,Y\cup\{v\}}, \mathscr{W}^*)\rightarrow \bigoplus_{w\in Y} H^1(\op{G}_w, \mathscr{W}^*).\] By Proposition $\ref{lastchebotarev}$ condition $\eqref{71four}$, 
\[H^1_{\mathfrak{M}^{\perp}}(\op{G}_{\Q,Y\cup\{v\}}, \mathscr{W}^*)=H^1(\op{G}_{\Q,Y}, \mathscr{W}^*)\] and therefore, the image of $\Phi_1$ is equal to the image of $\Phi_2$.
\par We deduce that
\begin{equation}\label{phi3phi4}
\begin{split}
\dim  \ker \Phi_2-\dim  \ker \Phi_1 &=h^1(\op{G}_{\Q,Y\cup \{v\}},\mathscr{W})-h^1(\op{G}_{\Q,Y},\mathscr{W})\\
&=h^1(\operatorname{G}_v, \mathscr{W})-h^0(\operatorname{G}_v, \mathscr{W})\\
&=h^1(\operatorname{G}_v, \mathscr{W})-h^1_{nr}(\operatorname{G}_v, \mathscr{W}).\\
\end{split}
\end{equation}
By $\ref{phi3phi4}$, we deduce that the sequence
\begin{equation}\label{lastequation}0\rightarrow \ker \Phi_1\rightarrow \ker\Phi_2\rightarrow \frac{H^1(\operatorname{G}_v, \mathscr{W})}{H^1_{nr}(\operatorname{G}_v, \mathscr{W})}\rightarrow 0\end{equation}
is a short exact sequence.

\par Define the maps
\begin{equation*}
\Phi_3: H^1(\op{G}_{\Q,Y\cup \{v\}},\g) \rightarrow \bigoplus_{w\in Y} \frac{H^1(\op{G}_w,\g)}{\mathcal{N}_w}
\end{equation*}
and 
\begin{equation*}
\Phi_4: H^1(\op{G}_{\Q, Y},\g) \rightarrow \bigoplus_{w\in Y} \frac{H^1(\op{G}_w,\g)}{\mathcal{N}_w}.
\end{equation*} From the Cassels-Poitou-Tate long exact sequence and the vanishing of $\Sh^2_Y(\g)$, we deduce that the following sequences are exact
\begin{equation*}
 H^1(\op{G}_{\Q,Y\cup \{v\}},\g) \xrightarrow{\Phi_3} \bigoplus_{w\in Y} \frac{H^1(\op{G}_w,\g)}{\mathcal{N}_w}\rightarrow H^1_{\mathcal{M}^{\perp}}(\op{G}_{\Q,Y\cup \{v\}},\g^*)^{\vee}\rightarrow 0
\end{equation*}
\begin{equation*}
 H^1(\op{G}_{\Q,Y},\g) \xrightarrow{\Phi_4} \bigoplus_{w\in Y} \frac{H^1(\op{G}_w,\g)}{\mathcal{N}_w}\rightarrow H^1_{\mathcal{N}^{\perp}}(\op{G}_{\Q,Y},\g^*)^{\vee}\rightarrow 0.
\end{equation*}
Set $t'$ to denote the difference $\dim  \op{image} \Phi_3-\dim  \op{image} \Phi_4$. From the assertion made in $\eqref{dimgreaterone}$ we conclude that $t'\geq 1$.
\par We claim that it suffices to find $\dim \g-t'+1$ elements in $\ker\Phi_3$, no linear combination of which lies in $\mathcal{N}_v$. It follows then that the image of \[\ker\Phi_3\rightarrow \frac{H^1(\operatorname{G}_v,\g)}{\mathcal{N}_v}\] has dimension strictly greater than $\dim  \g-t'$. From the exactness of 
\[0\rightarrow H^1_{\mathcal{N}}(\op{G}_{\Q,Y\cup\{v\}}, \g)\rightarrow \ker\Phi_3\rightarrow \frac{H^1(\operatorname{G}_v,\g)}{\mathcal{N}_v}\]one may deduce that
\[\begin{split}h_{\mathcal{N}}^1(\op{G}_{\Q,Y\cup\{v\}}, \g)<&\dim  \ker \Phi_3-\dim  \g+t'.\\
= & h^1(\op{G}_{\Q,Y\cup\{v\}}, \g)-\dim \g -\dim \op{im} \Phi_4.\\\end{split}\]
Note that $\Sh^1_{Y}(\g)=0$ and thus an application of Wiles' formula \eqref{wilesformula} shows that
\[\begin{split}h^1(\op{G}_{\Q,Y}, \g)=& h^0(\op{G}_{\Q}, \g)-h^0(\op{G}_{\Q},\g^*)\\
+&\sum_{w\in Y\cup \{\infty\}} (h^1(\op{G}_w, \g)-h^0(\op{G}_w, \g))\end{split}\]
and 
\[\begin{split}h^1(\op{G}_{\Q,Y\cup\{v\}}, \g)=& h^0(\op{G}_{\Q}, \g)-h^0(\op{G}_{\Q},\g^*)\\
+&\sum_{w\in Y\cup \{v\}\cup \{\infty\}} (h^1(\op{G}_w, \g)-h^0(\op{G}_w, \g)).\end{split}\]
Therefore, 
\[\begin{split}
    h^1(\op{G}_{\Q,Y\cup \{v\}}, \g)=& h^1(\op{G}_{\Q,Y}, \g)+h^1(\op{G}_v, \g)-h^0(\op{G}_v, \g)\\
    =&h^1(\op{G}_{\Q,Y}, \g)+\dim \g.
\end{split}\]
Therefore, we have that 
\[\begin{split}h_{\mathcal{N}}^1(\op{G}_{\Q,Y\cup\{v\}}, \g)<& h^1(\op{G}_{\Q,Y}, \g)-\dim \op{im} \Phi_4\\
=&\dim \op{ker} \Phi_4\\
=& h_{\mathcal{N}}^1(\op{G}_{\Q,Y}, \g).
\end{split}\]
Therefore in order to complete the proof we proceed to construct $\dim \g-t'+1$ elements in $\ker\Phi_3$ no linear combination of which lies in $\mathcal{N}_v$. We are in fact able to construct $\dim  \g$ elements, which suffices since $t'\geq 1$.
\par Note that $\g/\mathscr{W}$ is isomorphic to $\F_q(\sigma_{-2L_1})$ and hence $H^0(\op{G}_{\Q},\g/\mathscr{W})$ is zero. We find that $H^1(\op{G}_{\Q,Y\cup \{v\}}, \mathscr{W})$ injects into $H^1(\op{G}_{\Q,Y\cup \{v\}}, \g)$ and thereby it follows that $\op{ker}\Phi_2$ is contained in $\op{ker}\Phi_3$. Let $Z_1,\dots, Z_s$ be a basis of $\mathscr{W}$. By the exactness of $\ref{lastequation}$ there exist $\omega_i\in \text{ker}\Phi_2$ such that $\omega_i(\tau_v)=Z_i$ for $i=1,\dots, s$. We show that no linear combination of $\{f, \omega_1,\dots, \omega_s\}$ lies in $\mathcal{N}_v$. Let $Q=c_0 f+\sum_{i=1}^s c_i \omega_i$ be in $\mathcal{N}_v$. Since $f$ is unramified at $v$, $f(\tau_v)=0$. On the other hand,
$Q(\tau_v)=\sum_{i=1}^s c_i Z_i$ is contained in $\mathscr{W}$. Since $Q\in \mathcal{N}_v$, \[Q(\tau_v)=c'(\operatorname{Id}+X_{-\alpha})X_{\alpha}(\operatorname{Id}+X_{-\alpha})^{-1}\] for $\alpha=2L_1$ and some constant $c'$. The root vectors $X_{\alpha}$ and $X_{-\alpha}$ are constant multiples of $e_{1,n+1}$ and $e_{n+1,1}$ respectively. Assume without loss of generality that $X_{\alpha}=e_{1,n+1}$ and $X_{-\alpha}=e_{n+1,1}$. Clearly, $X_{-\alpha}^2=0$ and hence $(1+X_{-\alpha})^{-1}=(1-X_{-\alpha})$. We see that
\[\begin{split}Q(\tau_v)=&c'(\operatorname{Id}+X_{-\alpha})X_{\alpha}(\operatorname{Id}-X_{-\alpha}) \\
=& c'\left(X_{\alpha} +[X_{-\alpha},X_{\alpha}]-X_{-\alpha}X_{\alpha}X_{-\alpha}\right)\\
=& c'\left(e_{1,n+1}-H_1-e_{n+1,1}\right). \end{split}\]We deduce that $Q(\tau_v)=0$ since $e_{n+1,1}\notin \mathcal{W}$. Therefore, $c_i=0$ for all $i=1,\dots, s$. As a consequence, $Q=c_0 f$. However, $f$ is not contained in $\mathcal{N}_v$. It follows that $c_0=0$ and therefore, $Q=0$. Therefore no linear combination of $\{f, \omega_1,\dots, \omega_s\}$ lies in $\mathcal{N}_v$ and this completes the proof.
\end{proof}
To conclude the proof of Theorem $\ref{main}$, we observe that on choosing an appropriately large choice of trivial primes the dual Selmer group vanishes and hence, by the lifting construction outlined in section $\ref{section3}$, $\rho_3$ lifts to a characteristic zero representation $\rho$. Furthermore, $\rho$ can be arranged to have similitude character $\kappa$, and satisfy the local conditions $\mathcal{C}_v$ at the primes $v\in S$. Proposition $\ref{bigimageprop}$ asserts that the image of $\rho_2$ contains \[\widehat{\op{Sp}}_{2n}(\text{W}(\F_q)/p^2):=\left\{\op{Sp}_{2n}(\text{W}(\F_q)/p^2)\rightarrow\op{Sp}_{2n}(\F_q) \right\}\]and it follows that $\rho$ is irreducible.
\section{Examples}\label{examples}
\par Let $p$ be an odd prime. Under certain hypotheses on $p$, we show that there are examples of reducible Galois representations 
$\bar{\rho}:\op{G}_{\Q,\{p\}}\rightarrow \op{GSp}_4(\F_p)$ which satisfy the conditions of Theorem $\ref{main}$. First, we sketch the strategy used. The reader may refer to section $\ref{notationsection}$ for some of the notation used in this section. Recall that $\bar{\chi}$ denotes the mod-$p$ cyclotomic character. Let $\varphi_1,\varphi_2$ and $\bar{\kappa}:\op{G}_{\Q,\{p\}}\rightarrow \op{GL}_1(\F_q)$ be characters to be specificied later and $\bar{r}$ the diagonal representation specified by \[\bar{r}:=\left( {\begin{array}{cccc}
   \varphi_1 & & &  \\
    & \varphi_2 & & \\
    & & \varphi_1^{-1} \bar{\kappa} & \\
    & & & \varphi_2^{-1} \bar{\kappa} 
  \end{array} } \right):\op{G}_{\Q,\{p\}}\rightarrow \op{GSp}_4(\F_p).\] The characters, $\varphi_i$ and $\bar{\kappa}$ will be powers of $\bar{\chi}$.
  Let $D\in \op{GSp}_4(\Q_p)$ be the diagonal matrix \[D:=\left( {\begin{array}{cccc}
   p & & &  \\
    & 1& & \\
    & & p^{-2}  & \\
    & & & p^{-1}
  \end{array} } \right)\] and set $H$ to denote $(D\op{GSp}_4(\Z_p)D^{-1})\cap (D^{-1}\op{GSp}_4(\Z_p)D)$. Note that $H$ is the subgroup of $\op{GSp}_4(\Z_p)$ consisting of matrices \[X=\left( {\begin{array}{cccc}
   a_{1,1} & p a_{1,2}& p^3 a_{1,3}& p^2 a_{1,4} \\
   p a_{2,1} & a_{2,2}& p^2 a_{2,3}& p a_{2,4}\\
   p^3 a_{3,1} & p^2 a_{3,2} & a_{3,3} & p a_{3,4}\\
    p^2 a_{4,1}& p a_{4,2} & p a_{4,3} & a_{4,4}
  \end{array} } \right).\] Note that $D^{-1} X D$ is equal to  \[\left( {\begin{array}{cccc}
   a_{1,1} &  a_{1,2}&  a_{1,3}&  a_{1,4} \\
   p^2 a_{2,1} & a_{2,2}&  a_{2,3}& a_{2,4}\\
   p^6 a_{3,1} & p^4 a_{3,2} & a_{3,3} & p^2 a_{3,4}\\
    p^4 a_{4,1}& p^2 a_{4,2} &  a_{4,3} & a_{4,4}
  \end{array} } \right)\] and thus reduces to the Borel $\op{B}(\F_q)$ modulo $p$. Let $H_0$ denote the intersection of $H$ with $\op{Sp}_4(\Z_p)$. Recall that for $k\geq 1$, $\op{U}_k(\F_p)\subset \op{B}(\F_q)$ is the exponential subgroup generated by $\operatorname{exp}((\g)_k)$. The strategy we adopt is as follows:
  \begin{enumerate}
      \item Under some conditions on $p$, we may choose $\varphi_1,\varphi_2$ and $\bar{\kappa}$ such that $H^2(\op{G}_{\Q,\{p\}},\op{Ad}^0\bar{r})$ is zero. Thus the global deformation problem (unramified outside $\{p\}$) is unobstructed.
      \item We show that there is a lift
      \[r:\op{G}_{\Q,\{p\}}\rightarrow \op{GSp}_4(\Z_p)\] of $\bar{r}$ with image in $H$. Letting $\bar{\rho}$ denote the mod-$p$ reduction of $D^{-1} r D$, we note that the image of $\bar{\rho}$ is contained in $\op{B}(\F_p)$.
      \item Let $\Pi$ denote the intersection of the image of $\bar{\rho}$ with $\op{U}_1(\F_p)$. It is shown that after the mod-$p^2$ lift $r_2$ of $\bar{r}$ may be carefully chosen so that $\Pi$ surjects onto $\op{U}_1(\F_p)/\op{U}_2(\F_p)$. Lemma $\ref{bigimagelemmaU1}$ shows that the image of $\bar{\rho}$ contains $\op{U}_1(\F_p)$. Moreover, the characters $\varphi_1, \varphi_2$ and $\bar{\kappa}$ are suitably chosen so that all the conditions of Theorem $\ref{main}$ are satisfied.
  \end{enumerate}
  Recall that $\Phi^+$ consists of roots $\{2L_1,2L_2, (L_1-L_2), (L_1+L_2)\}$ and the simple roots are $\lambda_1=L_1-L_2$ and $\lambda_2=2L_2$. The root vectors are as follows \[
{\small X_{2L_1}}:={\tiny\left( {\begin{array}{cccc}
    0 &  & 1 &  \\
    & 0 &  &  \\
   &  &  0 &  \\
    &  & & 0
  \end{array} }\right)}, {\small X_{2L_2}}:={\tiny\left( {\begin{array}{cccc}
    0 &  &  &  \\
    & 0 &  & 1 \\
   &  &  0 &  \\
    &  & & 0
  \end{array} }\right)},\]\[{\small X_{L_1+L_2}}:={\tiny\left( {\begin{array}{cccc}
    0 &  &  & 1 \\
    & 0 & 1 &  \\
   &  &  0 &  \\
    &  & & 0
  \end{array} }\right)}\text{ and }{\small X_{L_1-L_2}}:={\tiny\left( {\begin{array}{cccc}
    0 & 1 &  &  \\
    & 0 &  &  \\
   &  &  0 &  \\
    &  & -1 & 0
  \end{array} }\right).}
\] For $m\geq 1$, set $H_0(\Z/p^m)$ (resp. $H(\Z/p^m)$) to denote the image of $H_0$ (resp. $H$) in $\op{Sp}_4(\Z/p^m)$. Let $\mathfrak{h}_m$ denote the kernel of the mod-$p^m$ reduction map $H_0(\Z/p^{m+1})\rightarrow H_0(\Z/p^{m})$. Identify $\mathfrak{h}_m$ with a subspace of $\op{Ad}^0\bar{r}$, so that $\op{Id}+p^m X$ is identified with $X$ in $\op{Ad}^0\bar{r}$. It is easy to see that 
\[\mathfrak{h}_m=\begin{cases}
\F_p\langle H_1, H_2, X_{\pm(L_1-L_2)}, X_{\pm{2L_2}}\rangle\text{, for }m=1,\\
\F_p\langle H_1, H_2, X_{\pm(L_1-L_2)}, X_{\pm{2L_2}}, X_{\pm(L_1+L_2)}\rangle\text{, for }m=2,\\
\op{Ad}^0\bar{r}\text{, for }m\geq 3.

\end{cases}\]
 \par Let $A$ be the Class group of $\Q(\mu_p)$ and let $\mathcal{C}$ denote the mod-$p$ class group $\mathcal{C}:=A\otimes \F_p$. The Galois group $\op{Gal}(\Q(\mu_p)/\Q)$ acts on $\mathcal{C}$ via the natural action. Since the order of $\op{Gal}(\Q(\mu_p)/\Q)$ is prime to $p$, it follows that $\mathcal{C}$ decomposes into eigenspaces 
  \[\mathcal{C}=\bigoplus_{i=0}^{p-2} \mathcal{C}(\bar{\chi}^i),\]
  where $\mathcal{C}(\bar{\chi}^i)=\{x\in \mathcal{C} \mid g\cdot x=\bar{\chi}^{i}(g) x \}$.
   \begin{Lemma}\label{lemma31}
  For $0\leq i\leq p-2$,
  \begin{enumerate}
      \item\label{lemma31p1} the group $\Sh^1_{\{p\}}(\F_p(\bar{\chi}^i))$ injects into $\op{Hom}(\mathcal{C}(\bar{\chi}^i),\F_p)$,
      \item\label{lemma31p2} the group $\Sh^2_{\{p\}}(\F_p(\bar{\chi}^i))$ equals zero if $\mathcal{C}(\bar{\chi}^{p-i})$ equals zero.
  \end{enumerate}
  \end{Lemma}
  \begin{proof}
  Since the order of $\op{Gal}(\Q(\mu_p)/\Q)$ is prime to $p$, it follows that \[H^1(\op{Gal}(\Q(\mu_p)/\Q), \F_p(\bar{\chi}^i))=0.\] As a result, part $\eqref{lemma31p1}$ follows from the inflation-restriction sequence. Part $\eqref{lemma31p2}$ follows from part $\eqref{lemma31p1}$ and Poitou-Tate duality for $\Sh$-groups \cite[Theorem 8.6.7]{NW}.
  \end{proof}

  \begin{Lemma}\label{bigimagelemmaU1}
Suppose that $p>2$ is a prime number and let $\Pi$ be the a subgroup of $\op{U}_1(\F_p)$ such that the quotient map $\Pi\rightarrow \op{U}_1(\F_p)/\op{U}_2(\F_p)$ is surjective. Then $\Pi$ is equal to $\op{U}_1(\F_p)$.
  \end{Lemma}
  \begin{proof}
  For $x,y\in \op{U}_1(\F_p)$, set $\{x,y\}$ to denote the commutator $xyx^{-1}y^{-1}$. As $\F_p$-vector spaces, we have that
  \[\op{U}_k(\F_p)/\op{U}_{k+1}(\F_p)=\begin{cases}
  \F_p\langle \op{exp}(X_{L_1-L_2}), \op{exp}(X_{2L_2})\rangle \text{ if }k=1,\\
  
  \F_p \langle \op{exp}(X_{L_1+L_2})\rangle \text{ if }k=2,\\
    \F_p\langle \op{exp}(X_{2L_1}) \rangle\text{ if }k=3,\\
    0\text{ if }k>3.
  
  \end{cases}\]We check that if $x=\op{exp}(X_{\lambda})$ and $y=\op{exp}(X_{\mu})$, for roots $\mu$ and $\lambda$ in $\Phi^+$, with height $k$ and $l$ respectively, then
  \[\{x,y\}=\op{exp}([X_{\lambda},X_{\mu}])\mod{\op{U}_{k+l+1}(\F_p)}.\]We have the relations
  \begin{equation}\label{relationsrootvectors}[X_{L_1-L_2}, X_{2L_2}]=X_{L_1+L_2}\text{ and } [X_{L_1-L_2}, X_{L_1+L_2}]=2X_{2L_1}\end{equation} and that $X_{\lambda}^2=0$ for $\lambda\in \Phi^+$. We have therefore,
  \[\begin{split}\{x,y\}=&(\op{Id}+X_{\lambda})(\op{Id}+X_{\mu})(\op{Id}-X_{\lambda})(\op{Id}-X_{\mu})\\
  =&\op{Id}+[X_{\lambda}, X_{\mu}]+(X_{\mu}X_{\lambda} X_{\mu}-X_{\lambda}X_{\mu} X_{\lambda})+(X_{\lambda}X_{\mu})^2.\end{split}\]Since $X_{\lambda}^2$ and $X_{\mu}^2$ are both equal to zero, we have that 
  \[X_{\lambda} X_{\mu} X_{\lambda}=\frac{1}{2}[[X_{\lambda}, X_{\mu}],X_{\lambda}]\] and thus $X_{\lambda} X_{\mu} X_{\lambda}\in (\g)_{2k+l}$. Likewise, the same reasoning shows that \[X_{\mu} X_{\lambda} X_{\mu}=\frac{1}{2}[[X_{\mu}, X_{\lambda}],X_{\mu}]\] and therefore, $X_{\mu} X_{\lambda} X_{\mu}\in (\g)_{k+2l}$. Next, observe that $(X_{\lambda} X_{\mu})^2$ is equal to $\frac{1}{2}([X_{\lambda}, X_{\mu}])^2$. This too follows from the relations $X_{\lambda}^2=X_{\mu}^2=0$. From the relations $\eqref{relationsrootvectors}$, we have that if $[X_{\lambda}, X_{\mu}]$ is nonzero, then, $\lambda+\mu$ is a root and there is a constant $c$ such that $[X_{\lambda}, X_{\mu}]=c X_{\lambda+\mu}$. Since, $X_{\lambda+\mu}^2=0$
, it follows that $(X_{\lambda} X_{\mu})^2=0$. Since $X_{\lambda} X_{\mu} X_{\lambda}\in (\g)_{2k+l}$, and the maximal height of any root is $3$, it follows that either $X_{\lambda} X_{\mu} X_{\lambda}$ is zero, or a constant multiple of $X_{2L_1}$. It may be checked that $X_{2L_1} X_{\lambda}=X_{\lambda} X_{2L_1}=0$ for all $\lambda\in \Phi^+$. As a consequence, we arrive at the following relation:
\[\{x,y\}=\op{exp}([X_{\lambda},X_{\mu}])\op{exp}(-\frac{1}{2}[[X_{\lambda}, X_{\mu}],X_{\lambda}])\op{exp}(\frac{1}{2}[[X_{\mu}, X_{\lambda}],X_{\mu}]).\] Note that $\op{exp}(-\frac{1}{2}[[X_{\lambda}, X_{\mu}],X_{\lambda}])$ and $\op{exp}(\frac{1}{2}[[X_{\mu}, X_{\lambda}],X_{\mu}])$ are in $\op{U}_{k+l+1}$. Therefore, we deduce that
 \[\{x,y\}=\op{exp}([X_{\lambda},X_{\mu}])\mod{\op{U}_{k+l+1}(\F_p)}.\]
We deduce from the relations $\eqref{relationsrootvectors}$ that the commutator $(x,y)\mapsto\{x,y\}$ induces a surjective map:
  \[\op{U}_1(\F_p)/\op{U}_2(\F_p)\times \op{U}_k(\F_p)/\op{U}_{k+1}(\F_p)\rightarrow \op{U}_{k+1}(\F_p)/\op{U}_{k+2}(\F_p).\]
  It follows by ascending induction on $k$, that the quotient map \[\Pi\cap \op{U}_k(\F_p)\rightarrow \op{U}_k(\F_p)/\op{U}_{k+1}(\F_p)\] is surjective for $k\geq 1$. By descending induction on $k$, we deduce that $\Pi\cap \op{U}_k(\F_p)=\op{U}_k(\F_p)$ for $k\geq 1$. In particular, $\Pi$ is equal to $\op{U}_1(\F_p)$ and the proof is complete.
  \end{proof}
  \begin{Prop}
  Let $p\geq 23$ be a prime such that $\mathcal{C}(\bar{\chi}^{p-i})=0$ for $i\in \{\pm 3, \pm 6, \pm 9\}$. There exists a Galois representation
  \[\bar{\rho}:=\left( {\begin{array}{cccc}
   \bar{\chi}^3 &\ast & \ast & \ast \\
    & 1 & \ast & \ast \\
    & & \bar{\chi}^6 & \\
    & & \ast & \bar{\chi}^9
  \end{array} } \right):\op{G}_{\Q,\{p\}}\rightarrow \op{B}(\F_p)\] which satisfies the conditions of Theorem $\ref{main}$. The similitude character of $\bar{\rho}$ is the odd character $\bar{\chi}^9$. Let $\kappa$ be a fixed choice of a lift of $\bar{\kappa}$ such that $\kappa=\kappa_0\chi^k$, where $k$ is a positive integer divisible by $p(p-1)$ and $\kappa_0$ is the Teichm\"uller lift of $\bar{\kappa}$. There exists a finite set of auxiliary primes $X$ such that $p\notin X$ and a lift $\rho$ \[\begin{tikzpicture}[node distance = 2.0cm, auto]
      \node (GSX) {$\operatorname{G}_{\Q,\{p\}\cup X}$};
      \node (GS) [right of=GSX] {$\operatorname{G}_{\Q,\{p\}}$};
      \node (GL2) [right of=GS]{$\GSp_{4}(\F_p).$};
      \node (GL2W) [above of= GL2]{$\GSp_{4}(\Z_p)$};
      \draw[->] (GSX) to node {} (GS);
      \draw[->] (GS) to node {$\bar{\rho}$} (GL2);
      \draw[->] (GL2W) to node {} (GL2);
      \draw[dashed,->] (GSX) to node {$\rho$} (GL2W);
      \end{tikzpicture}\] for which 
\begin{enumerate}
\item\label{83p1} $\rho$ is irreducible,
\item\label{83p2} $\rho$ is $p$-ordinary (in the sense of \cite[section 4.1]{patrikisexceptional}),
\item\label{83p3} $\nu\circ \rho= \kappa$.
\end{enumerate}
  \end{Prop}
  \begin{proof}
 We show a representation $\bar{\rho}$ satisfying the conditions of Theorem $\ref{main}$ exists. It shall then follow from Theorem $\ref{main}$ that there exists a lift $\rho$ which satisfies the conditions $\eqref{83p1}$, $\eqref{83p2}$ and $\eqref{83p3}$ above. Let $\bar{r}$ be the representation with image in the diagonal torus:
  \[\bar{r}:=\left( {\begin{array}{cccc}
   \bar{\chi}^3 & & & \\
    & 1 &  &  \\
    & & \bar{\chi}^6  & \\
    & &  & \bar{\chi}^9 
  \end{array} } \right):\op{G}_{\Q,\{p\}}\rightarrow \op{GSp}_4(\F_p).\]
  
  The following matrix aids (in an informal way) in describing the eigenspace decomposition of $\op{Ad}^0\bar{r}$:
  \[\left( {\begin{array}{cccc}
   1 &  \bar{\chi}^3 &  \bar{\chi}^{-3} &  \bar{\chi}^{-6}\\
    \bar{\chi}^{-3} & 1 &  \bar{\chi}^{-6} &  \bar{\chi}^{-9} \\
     \bar{\chi}^{3}&  \bar{\chi}^{6}& 1  & \bar{\chi}^{-3} \\
     \bar{\chi}^{6}&  \bar{\chi}^{9}&  \bar{\chi}^3 & 1
  \end{array} } \right).\]More precisely, $\op{Ad}^0\bar{r}$ is an $11$-dimensional space which decomposes into one-dimensional eigenspaces, and we have that
  \[\sigma_{\pm 2L_1} = \bar{\chi}^{\mp 3},\sigma_{\pm 2L_2} = \bar{\chi}^{\mp 9},\sigma_{\pm(L_1+L_2)} = \bar{\chi}^{\mp 6}\text{ and }\sigma_{\pm (L_1-L_2)} = \bar{\chi}^{\pm 3}. \]
  We show that for $m\geq 1$, the global cohomology group $H^2(\op{G}_{\Q,\{p\}}, \mathfrak{h}_m)$ is zero. As a Galois module, $\mathfrak{h}_m$ decomposes into one-dimensional eigenspaces $\F_p(\sigma)$, where $\sigma$ ranges through some of the characters $1,\bar{\chi}^{\pm 3}, \bar{\chi}^{\pm 6}, \bar{\chi}^{\pm 9}$. It suffices to show that for the above choices of $\sigma$, the cohomology group $H^2(\op{G}_{\Q,\{p\}}, \F_p(\sigma))$ is zero. We show that $H^2(\op{G}_p, \F_p(\sigma))$ is zero and $\Sh^2_{\{p\}}(\F_p(\sigma))$ is zero. The dual $\F_p(\sigma)^*:=\op{Hom}(\F_p(\sigma), \mu_p)$ is isomorphic to $\F_p(\bar{\chi}\sigma^{-1})$. Since $p\geq 13$, the character $\bar{\chi}\sigma^{-1}_{\restriction \op{G}_p}\neq 1$. It follows from local duality that $H^2(\op{G}_p, \F_p(\sigma))$ is zero. It is a standard fact that $\mathcal{C}(\bar{\chi})$ is zero, see \cite[Proposition 6.16]{washington}. It follows from Lemma $\ref{lemma31}$ and the assumptions on $p$ that $\Sh^2_{\{p\}}(\op{Ad}^0\bar{r})$ is zero. We have thus shown that $H^2(\op{G}_{\Q,\{p\}}, \mathfrak{h}_m)$ is zero for all $m\geq 1$.
  \par We stipulate that all deformations of $\bar{r}$ have similitude character equal to $\chi^9$, where we recall that $\chi$ denotes the cyclotomic character. For $m\geq 1$, let $\chi_m$ denote $\chi$ modulo $p^m$. Recall that $\mathfrak{h}_1$ is spanned by $H_1,H_2$ and $\sigma_{\pm \lambda_i}$ for $i=1,2$. Since the characters $\sigma_{\lambda_i}$ for $i=1,2$ are both odd and \[H^2(\op{G}_{\Q,\{p\}}, \F_p(\sigma_{\lambda_i}))=0,\] it follows from the global Euler characteristic formula \cite[Theorem 8.7.4]{NW} that \[\op{dim}H^1(\op{G}_{\Q,\{p\}}, \F_p(\sigma))=1.\] Let $f_i$ be a generator for $H^1(\op{G}_{\Q,\{p\}}, \F_p(\sigma_{\lambda_i}))$ for $i=1,2$. Let $r_2'$ be the mod-$p^2$ lift  \[r_2':=\left( {\begin{array}{cccc}
   {\chi}_2^3 & & & \\
    & 1 &  &  \\
    & & {\chi}_2^6  & \\
    & &  & {\chi}_2^9 
  \end{array} } \right):\op{G}_{\Q,\{p\}}\rightarrow \op{GSp}_4(\F_p)\] and $r_2$ be the twist $(\op{Id}+p(f_1+f_2))r_2'$. Note that the image of $r_2$ is in $H(\Z/p^2)$. The obstruction to lifting $r_2$ to $r_3:\op{G}_{\Q,\{p\}}\rightarrow H(\Z/p^3)$ is in  $H^2(\op{G}_{\Q,\{p\}}, \mathfrak{h}_2)$, hence, is zero. Hence, $r_2$ lifts to $r_3$. Since $H^2(\op{G}_{\Q,\{p\}}, \mathfrak{h}_m)=0$ for all $m\geq 1$, it follows that if $r_m:\op{G}_{\Q,\{p\}}\rightarrow H(\Z/p^m)$ is a lift of $r_2$ (with similitude character $\chi_m^9$), then $r_m$ lifts one more step to $r_{m+1}:\op{G}_{\Q,\{p\}}\rightarrow H(\Z/p^{m+1})$. Furthermore, the lift $r_{m+1}$ can be prescribed to have similitude character $\chi_{m+1}^9$. The key ingredient here is that $H$ is a subgroup of $\op{GSp}_4(\Z_p)$. Since $H$ is a closed subgroup, it follows that $r_2$ lifts to a continuous characteristic zero representation $r:\op{G}_{\Q,\{p\}}\rightarrow H(\Z_p)$ with similitude character $\chi^9$. Let $\bar{\rho}:\op{G}_{\Q,\{p\}}\rightarrow \op{B}(\F_p)$ be the mod-$p$ reduction of $D^{-1} r D$ and let $\Pi$ be $\bar{\rho}(\op{G}_{\Q(\mu_p)})$. Lemma $\ref{bigimagelemmaU1}$ asserts that if $\Pi\rightarrow \op{U}_1(\F_p)/\op{U}_2(\F_p)$ is surjective, then the image of $\bar{\rho}$ contains $\op{U}_1(\F_p)$. Let $\Phi(r_2)\subseteq \mathfrak{h}_1$ be $r_2(\op{ker} \bar{r})$. In fact, $\Phi(r_2)$ is contained in $\F_p\langle H_1, H_2, X_{L_1-L_2}, X_{2L_2}\rangle $. Since the characters $1, \sigma_{\lambda_1}=\bar{\chi}^3, \sigma_{\lambda_2}=\bar{\chi}^{-9} $ are distinct, it follows that $\Phi(r_2)$ decomposes into distinct eigenspaces
  \[\Phi(r_2)=\Phi(r_2)^{\op{Gal}(\Q(\mu_p)/\Q)}\oplus\Phi(r_2)_{\bar{\chi}^3}\oplus \Phi(r_2)_{\bar{\chi}^{-9}} .\]
  Since $\op{Gal}(\Q(\mu_p)/\Q)$ is prime to $p$, it follows from a straightforward application of the inflation restriction sequence that $f_{i\restriction \op{G}_{\Q(\mu_p)}}$ is nonzero. As a result, $\Phi(r_2)_{\bar{\chi}^3}$ and $\Phi(r_2)_{\bar{\chi}^{-9}}$ are nonzero, and hence, $X_{\lambda_i}\in \Phi(r_2)$. Since $\op{exp}(X_{\lambda_1})$ and $\op{exp}(X_{\lambda_2})$ are generators of $\op{U}_1(\F_p)/\op{U}_2(\F_p)$, it follows that $\Pi$ surjects onto $\op{U}_1(\F_p)/\op{U}_2(\F_p)$. Thus, the image of $\bar{\rho}$ contains $\op{U}_1(\F_p)$.
  \par We show that the conditions of Theorem $\ref{main}$ are satisfied.
  \begin{itemize}
      \item Condition $\eqref{thc1}$ asserts that $p>4$, we have assumed that $p\geq 23$.
      \item Condition $\eqref{thc2}$ asserts that $\dim (\g)^{\op{ad}\bar{\rho}(c)}=\dim \mathfrak{n}$. Since $p>2$, up to conjugation, $\bar{\rho}(c)$ is equal to $\left( {\begin{array}{cccc}
   -1 & & & \\
    & 1 &  &  \\
    & & 1 & \\
    & &  & -1
  \end{array} } \right)$. Explicit computation shows that w.r.t this basis, 
  \[(\g)^{\op{ad}\bar{\rho}(c)}=\F_p\langle H_1, H_2, (L_1+L_2), -(L_2+L_2) \rangle, \]and hence, $\dim (\g)^{\op{ad}\bar{\rho}(c)}$ is equal to $4$. On the other hand, there are $4$ positive roots and the dimension of $\mathfrak{n}$ is $4$.
  \item Condition $\eqref{thc3}$ asserts that the image of $\bar{\rho}$ contains the unipotent group $\op{U}_1(\F_p)$. This has been shown to be the case.
  \item For condition $\eqref{thc4}$, consider $\sigma_{\lambda}=\bar{\chi}^i$ and $\sigma_{\lambda'}=\bar{\chi}^j$. Since $i,j\in\{1, \pm 3, \pm 6, \pm 9\}$, we see that $|i-j|\leq 18<p-1$. Hence, the characters $\sigma_{\lambda}$ and $\sigma_{\lambda'}$ are distinct. Note that $\bar{\chi} \sigma_{\lambda'}=\bar{\chi}^{j+1}$ and $|i-(j+1)|\leq 19<p-1$. Hence, the characters $\sigma_{\lambda}$ and $\bar{\chi}\sigma_{\lambda'}$ are distinct.
  \item Condition $\eqref{thc5}$, asserts that each of the roots $\lambda\in \Phi$, the $\F_p$-linear span of the image of $\sigma_{\lambda}$ in $\F_q$ is $\F_q$. But $\F_q$ is $\F_p$ and thus the condition is clearly satisfied.
  \item Condition $\eqref{thc7}$, is satisfied since the only prime at which $\bar{\rho}$ ramifies is $p$.
  \item Condition $\eqref{thc8}$ asserts that Tilouine's regularity conditions $(\operatorname{REG})$ and $(\operatorname{REG})^*$ are satisfied, i.e. 
\[H^0(\op{G}_p, \g/\mathfrak{b})=0\text{ and }H^0(\op{G}_p, (\g/\mathfrak{b})(\bar{\chi}))=0.\] It is clear that $\g/\mathfrak{b}$ and $(\g/\mathfrak{b})(\bar{\chi})$ have no trivial eigenspace for the torus action.
  \end{itemize}
  The proof is now complete.
  \end{proof}
  \begin{Remark}
 \begin{enumerate}
  \item The above hypotheses is satisfied for any regular prime $p\geq 23$. In particular, it is satisfied for $p=23$.
      \item The section simply serves to demonstrate that examples do exist and demonstrate how they may be constructed. We restrict to $\op{GSp}_4$ so that the arguments are simplified. 
      \item Note that $D^{-1} r D$ is not necessarily $p$-ordinary and thus not necessarily geometric in the sense of Fontaine-Mazur. On the other hand, the lift $\rho$ is $p$-ordinary. 
  \end{enumerate}
  \end{Remark}

\end{document}